	\RequirePackage{etex}
\documentclass[a4paper,10pt,leqno]{amsart}
\usepackage[bf,]{style}
	\calclayout
\usepackage{bm,tikz-cd}

	\mathtoolsset{showonlyrefs}
\theoremstyle{plain}
	\newtheorem{mainthm}{Theorem}
	\newtheorem{maincor}[mainthm]{Corollary}
\theoremstyle{plain}
	\newtheorem{thm}{Theorem}[subsection]
\theoremstyle{plain}
	\newtheorem{conj}[mainthm]{Conjecture}
	\newtheorem{cor}[thm]{Corollary}
	\newtheorem{lem}[thm]{Lemma}
	\newtheorem{prop}[thm]{Proposition}
\theoremstyle{definition}
	\newtheorem{quest}[mainthm]{Question}
	\newtheorem{df}[thm]{Definition}
	\newtheorem{ex}[thm]{Example}
	\newtheorem{lempost}[thm]{Lemma-Postulate}
\theoremstyle{remark}
	\newtheorem{rem}[thm]{Remark}

\newcommand{\dfemph}[1]{\textbf{#1}}

\let\Alt\relax
	\DeclareMathOperator{\Alt}{\text{$\Lambda$}}

\newcommand{\Br}{\mathit{Br}}
\newcommand{\IC}{\mathit{IC}}
\newcommand{\Mod}{\mathsf{Mod}}
\newcommand{\Q}{\bm{q}}
\newcommand{\QL}{\bar{\mathbf{Q}}_\ell}
\newcommand{\UCh}{\mathrm{Uch}}
\newcommand{\UDeg}{\mathrm{Deg}}
	\newcommand{\FDeg}{\mathrm{Feg}}
\newcommand{\Vect}{\mathsf{Vect}}

\newcommand{\Tr}[1]{\langle{#1}\rangle_{\Q}}

	\renewcommand{\aff}{\mathit{aff}}
\newcommand{\act}{\mathit{act}}
\newcommand{\point}{\mathit{pt}}
	\renewcommand{\pr}{\mathit{pr}}
\newcommand{\rat}{\mathit{rat}}
\newcommand{\slope}{\nu}

\makeatletter
\g@addto@macro \normalsize {%
 \setlength\abovedisplayskip{10pt plus 2pt minus 2pt}%
 \setlength\belowdisplayskip{10pt plus 2pt minus 2pt}%
}
\makeatother

\begin{document}

\title{From the Hecke Category to the Unipotent Locus}
\author{Minh-T\^{a}m Quang Trinh}
\address{Massachusetts Institute of Technology, 77 Massachusetts Avenue, Cambridge, MA 02139}
\email{mqt@mit.edu}

\maketitle

\begin{abstract}
Let $W$ be the Weyl group of a split semisimple group $G$.
Its Hecke category $\mathsf{H}_W$ can be built from pure perverse sheaves on the double flag variety of $G$.
By developing a formalism of generalized realization functors, we construct a monoidal trace from $\mathsf{H}_W$ to a category of bigraded modules over a certain graded ring: namely, the endomorphisms of the equivariant Springer sheaf over the unipotent locus of $G$.
We prove that: (1) On objects attached to positive braids $\beta$, the output is the weight-graded, equivariant Borel--Moore homology of a generalized Steinberg scheme $\mathcal{Z}(\beta)$.
(2) Our functor contains, as a summand, one used by Webster--Williamson to construct the Khovanov--Rozansky link invariant.
In particular, the Khovanov--Rozansky homology of the link closure of $\beta$ is fully encoded in the Springer theory of $\mathcal{Z}(\beta)$.

Decategorifying, we get a trace on the Iwahori--Hecke algebra, valued in graded virtual characters of $W$.
We give a formula for it in terms of a pairing on characters of $W$ called Lusztig's exotic Fourier transform, which generalizes it to all finite Coxeter groups.
Using this formula, we establish properties of the trace like rationality, symmetry, and compatibility with parabolic induction.
We also show that on periodic braids, the trace produces the characters of explicit virtual modules over Cherendik's rational double affine Hecke algebra.
For $W = S_n$, we recover an identity of Gorsky--Oblomkov--Rasmussen--Shende relating these modules to the HOMFLY polynomials of torus knots.
\end{abstract}

\setcounter{tocdepth}{1}
\tableofcontents
\thispagestyle{empty}


\newpage

\section*{Introduction}

\subsection{Summary}

Let $\bb{V}$ be a finite-dimensional vector space over a subfield of $\bb{R}$, and let $W$ be a finite group of automorphisms of $\bb{V}$ generated by reflections.
Coxeter showed that $W$ admits a remarkable presentation in terms of generators and relations \cite{coxeter_1934}.
Via this presentation, the group ring of $W$ can be deformed to a $\bb{Z}[\Q^{\pm\frac{1}{2}}]$-algebra known as the Iwahori--Hecke algebra $H_W$.
One can further define an infinite group $\Br_W$, generalizing the braid groups studied in geometric topology, such that $W$ and $H_W$ are respective quotients of $\Br_W$ and $\bb{Z}[\Q^{\pm\frac{1}{2}}][\Br_W]$.

If $W = S_n$, the symmetric group on $n$ letters, then $\Br_W = \Br_n$, the classical braid group on $n$ strands.
Any braid $\beta$ can be closed up, end-to-end, to yield a topological link $\hat{\beta}$ in $3$-dimensional space.
If $\beta, \beta' \in \Br_n$ are conjugate, then $\hat{\beta}, \hat{\beta}'$ are isotopic.
Ocneanu, building on work of Jones, used this observation to construct an isotopy invariant of links called the HOMFLY series out of class functions on the groups $\Br_n$.
These functions, in turn, factor through functions called Markov traces on the algebras $H_{S_n}$ \cite{homfly, jones}.
Y.\ Gomi discovered that the latter admit a uniform generalization to any finite Coxeter group $W$ \cite{gomi}.

It is now known that $H_W$ is the (triangulated) Grothendieck ring of a monoidal triangulated category that we will call the Hecke category $\sf{H}_W$.
Gomi's Markov trace is categorified by an additive functor
\begin{align}
\sf{HHH} : \sf{H}_W \to \Vect_3,
\end{align}
where $\Vect_3$ denotes the category of triply-graded vector spaces.
As $W$ runs over the symmetric groups, these functors can be used to construct a link invariant valued in $\Vect_3$, originally constructed by Khovanov--Rozansky \cite{kr, 
khovanov}.
More precisely, there is a map from braids $\beta \in \Br_W$ to objects $\cal{R}(\beta) \in \sf{H}_W$, due to Rouquier \cite{rouquier_2006}, and $\sf{HHH}(\cal{R}(\beta))$ can be used to construct an isotopy invariant of $\hat{\beta}$ that decategorifies to its HOMFLY series.

The most general definition of $\sf{H}_W$ is purely algebraic: the homotopy category of Soergel bimodules over $\Sym(\bb{V})$ \cite{ew_2014, ew_2016}.
However, in representation theory, we are especially interested in the case where $\bb{V}$ is the span of a root lattice and $W$ is the Weyl group of the corresponding semisimple algebraic group $G$ over an algebraically closed field, say, $\bar{\bb{F}}$ for a finite field $\bb{F}$.
Here, the Hecke category can be defined in terms of the geometry of $G$, or rather, its split form $G_0$ over $\bb{F}$.
We are led to hope that $\sf{HHH}(\cal{R}(\beta))$ might be expressed in terms of the (co)homology of an explicit variety over $\bb{F}$ attached to $\beta$.

In the series of papers \cite{ww_2008, ww_2011, ww_2017}, Webster--Williamson relate $\sf{HHH}$ to the weight grading on the $G$-equivariant cohomology of certain varieties with maps to $G$.
The most inexplicit part of their work is a procedure they call ``chromatography,'' which turns an object of the mixed derived category $\sf{D}_{G, m}^b(G_0) = \sf{D}_{G, m}^b(G_0, \QL)$ into a complex up to homotopy whose terms are pure complexes of sheaves.
The use of this procedure means that in general, their work only relates $\sf{HHH}(\cal{R}(\beta))$ to a combination of several varieties, rather than to a single one.
These varieties encode the term-by-term structure of $\cal{R}(\beta)$, when $\cal{R}(\beta)$ is viewed as a complex of Soergel bimodules.

One of our goals is to simplify this situation.
Instead of chromatography, we will use a formalism of weight realization functors, inspired by works of Bezrukavnikov--Yun \cite{by}, Lusztig--Yun \cite{ly_2020}, and Rider \cite{rider}.
For special braids $\beta \in \Br_n$ called Richardson braids, Galashin--Lam have recently used Rider's work to identify a certain summand of $\sf{HHH}(\cal{R}(\beta))$, its so-called ``lowest Hochschild degree,'' with the torus-equivariant compactly-supported cohomology of a positroid variety \cite{gl}.
Their argument generalizes to any Weyl group $W$.
We go further:
For any positive $\beta \in \Br_W$, we realize the entirety of $\sf{HHH}(\cal{R}(\beta))$ within the equivariant Borel--Moore homology of an explicit $G$-variety over the unipotent locus of $G$.

Our viewpoint will lead us to introduce a refinement of $\sf{HHH}$, valued in modules rather than vector spaces.
Henceforth, we base-change $\bb{V}$ to a field of the form $\QL$ for some prime $\ell$ invertible in $\bb{F}$.
Form the semidirect-product ring
\begin{align}
\bb{A}_W = \QL[W] \ltimes \Sym(\bb{V}),
\end{align}
endowed with the grading where $W$ and $\bb{V}$ are respectively placed in degrees $0$ and $2$.
Let $\Mod_2^b(\bb{A}_W)$ be the category of $\bb{A}_W$-modules equipped with a bigrading where the $\bb{A}_W$-action is diagonally graded, both gradings are bounded below, and the difference between the gradings is uniformly bounded.
We will construct a contravariant(!) monoidal trace
\begin{align}
\sf{AH} : \sf{H}_W^\op \to \Mod_2^b(\bb{A}_W).
\end{align}
We will show that $\sf{HHH}^\vee$ is the sum of the $\Alt^i(\bb{V})$-isotypic components of $\sf{AH}$, where $\Alt^i(\bb{V})$ denotes the $i$th exterior power of $\bb{V}$, viewed as a representation of $W$.
The appearance of $\Alt^i(\bb{V})$ is reminiscent of work of Bezrukavnikov--Tolmachov on a monodromic model of $\sf{HHH}$ \cite{bt}.
For positive $\beta$, we will find that
\begin{align}\label{eq:a_w-to-varieties}
\sf{AH}(\cal{R}(\beta)) \simeq \gr_\ast^{\bm{w}} \ur{H}_{-\ast}^{!, G}(\cal{Z}(\beta))
\end{align}
for a certain $G$-scheme $\cal{Z}(\beta)$, where on the right-hand side, $\ur{H}_\ast^{!, G}(-) = \ur{H}_\ast^{!, G}(-, \QL)$ denotes equivariant Borel--Moore homology and $\bm{w}$ denotes the weight filtration coming from an $\bb{F}$-structure on $\cal{Z}(\beta)$.
In the case of the identity $\bb{1} \in \Br_W$, the reduced scheme $\cal{Z}(\bb{1})^\red$ is the variety Steinberg introduced in \cite{steinberg_1976}.
In particular, $\sf{AH}(\cal{R}(\bb{1})) \simeq \bb{A}_W$ by a result of Lusztig \cite{lusztig_1988}.

(To explain the contravariance:
The definition of $\sf{AH}$ will involve a $\Hom$ functor. 
It turns out that we need a contravariant, not covariant, $\Hom$ to get the right-hand side of \eqref{eq:a_w-to-varieties} to match $\sf{AH}(\cal{R}(\beta))$ rather than $\sf{AH}(\cal{R}(\beta^{-1}))$.)

Let us sketch the geometry that underlies $\sf{H}_W$ and $\sf{AH}$.
Let $\cal{B} = \cal{B}_0 \otimes_\bb{F} \bar{\bb{F}}$ and $\cal{U} = \cal{U}_0 \otimes_\bb{F} \bar{\bb{F}}$ denote the flag variety and unipotent locus of $G = G_0 \otimes_\bb{F} \bar{\bb{F}}$.
The Hecke category of $W$ can be interpreted as
\begin{align}
\sf{H}_W = \sf{K}^b(\sf{C}(\cal{B}_0 \times \cal{B}_0)),
\end{align}
where $\sf{C}(\cal{B}_0 \times \cal{B}_0)$ is a certain full, additive subcategory of $\sf{D}_{G, m}^b(\cal{B}_0 \times \cal{B}_0)$ generated by pure shifted perverse sheaves.
(More precisely, it is generated by the shift-twists of the intersection complexes of the $G_0$-orbits of $\cal{B}_0 \times \cal{B}_0$.)
The monoidal product on $\sf{H}_W$ arises from a geometrically-defined convolution.
The operation of passing from a braid to its conjugacy class is \emph{very roughly} realized by pulling and pushing sheaves from $\cal{B}_0 \times \cal{B}_0$ to $G_0$ through the following diagram, implicit in Lusztig's work on character sheaves and now known as the \dfemph{horocycle correspondence}:
\begin{equation}
\begin{tikzpicture}[baseline=(current bounding box.center), >=stealth]
\matrix(m)[matrix of math nodes, row sep=2.5em, column sep=2.5em, text height=2ex, text depth=0.5ex]
{	\cal{B}_0 \times \cal{B}_0
		&G_0 \times \cal{B}_0\\
	{}
		&G_0\\
		};
\path[->,font=\scriptsize, auto]
(m-1-2)		edge node[above]{$\act$} (m-1-1)
(m-1-2)		edge node[right]{$\pr$} (m-2-2);
\end{tikzpicture}
\end{equation}
Above, $\act$ and $\pr$ are the action and projection maps.

We introduce similar subcategories $\sf{C}(G_0) \subseteq \sf{D}_{G, m}^b(G_0)$ and $\sf{C}(\cal{U}_0) \subseteq \sf{D}_{G, m}^b(\cal{U}_0)$.
(The category $\sf{C}(G_0)$ is generated by objects of the form $E \otimes M$, where $E$ is a mixed version of a unipotent character sheaf and $M$ is a pure weight-zero Frobenius module; the category $\sf{C}(\cal{U}_0)$ is generated by the pullbacks of these objects along the inclusion $i : \cal{U}_0 \to G_0$.)
In particular, $\sf{C}(G_0)$ contains the Grothendieck--Springer sheaf $\cal{G}$, while $\sf{C}(\cal{U}_0)$ contains the Springer sheaf $\cal{S} \simeq i^\ast \cal{G}\langle -{\dim \bb{V}}\rangle$.
The horocycle correspondence gives rise to a functor 
\begin{align}
\sf{CH} : \sf{C}(\cal{B}_0 \times \cal{B}_0) \to \sf{C}(\cal{U}_0),
\end{align}
while pullback gives rise to a functor $i^\ast \langle-{\dim \bb{V}}\rangle : \sf{C}(G_0) \to \sf{C}(\cal{U}_0)$.

Crucially, $\sf{C}(\cal{B}_0 \times \cal{B}_0)$ and $\sf{C}(\cal{U}_0)$ satisfy a certain $\Ext$-vanishing condition (see Section \ref{sec:mixed}) that allows us to define realization functors $\rho : \sf{K}^b(\sf{C}(\cal{B}_0^2)) \to \sf{D}_{G, m}^b(\cal{B}_0^2)$ and $\rho : \sf{K}^b(\sf{C}(\cal{U}_0)) \to \sf{D}_{G, m}^b(\cal{U}_0)$.
Even though we are unable to construct a similar functor for $G_0$, the diagram
\begin{equation}
\begin{tikzpicture}[baseline=(current bounding box.center), >=stealth]
\matrix(m)[matrix of math nodes, row sep=2.5em, column sep=4em, text height=2ex, text depth=0.5ex]
{	\sf{K}^b(\sf{C}(\cal{B}_0 \times \cal{B}_0))
		&\sf{K}^b(\sf{C}(G_0))
		&\sf{K}^b(\sf{C}(\cal{U}_0))\\
	\sf{D}_{G, m}^b(\cal{B}_0 \times \cal{B}_0)
		&\sf{D}_{G, m}^b(G_0)
		&\sf{D}_{G, m}^b(\cal{U}_0)\\
		};
\path[->,font=\scriptsize, auto]
(m-1-2) 	edge node{$\sf{K}^b i^\ast \langle -r\rangle$} (m-1-3)
(m-2-2)		edge node{$i^\ast \langle -r\rangle$} (m-2-3)
(m-1-1)		edge node{$\sf{K}^b \sf{CH}$} (m-1-2)
(m-2-1)		edge node{$\sf{CH}$} (m-2-2)
(m-1-3) 	edge node[right]{$\rho$} (m-2-3)	
(m-1-1)		edge node[left]{$\rho$} (m-2-1);
\end{tikzpicture}
\end{equation}
commutes.
The monoidal trace $\sf{AH}$ is the composition
\begin{align}
\sf{H}_W = \sf{K}^b(\sf{C}(\cal{B}_0 \times \cal{B}_0))
	\to \sf{D}_{G, m}^b(\cal{U}_0)
	\xrightarrow{\gr_\ast^{\bm{w}} \SHom^\ast(-, \cal{S})} \Mod_2^b(\bb{A}_W),
\end{align}
where the first arrow is the functor produced by the diagram, and $\SHom^\ast$ denotes an $\Ext$-space endowed with a Frobenius action.
We use:
\begin{enumerate}
\item 	The bottom-left part of the diagram to match $\sf{AH}$ with the homology of the generalized Steinberg scheme $\cal{Z}(\beta)$ (in Section \ref{sec:varieties}).
\item 	The upper-right part of the diagram to match $\Hom_W(\Alt^\ast(\bb{V}), \sf{AH})$ with $\sf{HHH}$ (in Section \ref{sec:kr}).
\end{enumerate}
The main obstacle with (2) is that Webster--Williamson interpret $\sf{HHH}$ using the hypercohomology of objects over $G$, whereas we define $\sf{AH}$ using objects over $\cal{U}$.
Moreover, in their work, the Hochschild degree of $\sf{HHH}$ arises from $\gr_\ast^{\bm{w}}$, while in our work, it arises from $\Alt^\ast(\bb{V})$.
So the matching in (2) is fairly indirect.
Using orthogonality results of Lusztig \cite{lusztig_1985_2} and Rider--Russell \cite{rr}, we reduce the problem to checking certain direct-summand categories $\sf{C}_{\hat{W}}(G_0) \subseteq \sf{C}(G_0)$ and $\sf{C}_{\hat{W}}(\cal{U}_0) \subseteq \sf{C}(\cal{U}_0)$, which are roughly the ``principal blocks'' generated by $\cal{G}$ and $\cal{S}$.
On the level of these blocks, the key input is a geometrically-defined isomorphism between $\SEnd^\ast(\cal{S})$ and the pure part of $\SEnd^\ast(\cal{G})$ (in Section \ref{sec:steinberg}).

As for (1), let us sketch the definition of $\cal{Z}(\beta)$.
Let $\Br_W^+ \subseteq \Br_W$ be the monoid of positive braids.
The map $\Br_W \to W$ admits a set-theoretic section $W \to \Br_W^+$ that we will denote $w \mapsto \sigma_w$.
(In fact, the images of the $\sigma_w$ in $H_W$ form a $\bb{Z}[\Q^{\pm\frac{1}{2}}]$-basis.)
Deligne \cite{deligne_1997} and Brou\'e--Michel \cite{bm_1997} introduced a map
\begin{align}
\beta \mapsto O(\beta)
\end{align}
from elements of $\Br_W^+$ to $G$-varieties over $\cal{B} \times \cal{B}$, such that $O(\sigma_w)$ is the $G$-orbit of $\cal{B} \times \cal{B}$ indexed by $w$, and such that $O$ takes compositions of braids to fiber products over $\cal{B}$, up to explicit isomorphisms.
We form a cartesian square:
\begin{equation}
\begin{tikzpicture}[baseline=(current bounding box.center), >=stealth]
\matrix(m)[matrix of math nodes, row sep=2em, column sep=2.5em, text height=2ex, text depth=0.5ex]
{	O(\beta)
		&\cal{U}(\beta)\\
	\cal{B} \times \cal{B}
		&\cal{U} \times \cal{B}\\
		};
\path[->,font=\scriptsize, auto]
(m-1-2)		edge (m-1-1)
(m-1-2)		edge (m-2-2)
(m-1-1)		edge (m-2-1)
(m-2-2)		edge node[above]{$\act$} (m-2-1);
\end{tikzpicture}
\end{equation}
Then the map $\cal{U}(\bb{1}) \to \cal{U}$ is the Springer resolution.
We define $\cal{Z}(\beta)$ be the fiber product $\cal{U}(\bb{1}) \times_\cal{U} \cal{U}(\beta)$.

In the rest of the paper, we study the decategorification of $\sf{AH}$.
Let $R(W)$ be the representation ring of $W$.
Any $M \in \Mod_2^b(\bb{A}_W)$ decategorifies to an element
\begin{align}
[M]_{\Q} = \sum_{i,j} {(-1)^{i - j}}\Q^{\frac{j}{2}} M^{i, j} \in R(W)(\!(\Q^{\frac{1}{2}})\!)
\end{align}
that we call its graded character.
As a result, $\sf{AH}$ decategorifies to a $\bb{Z}[\Q^{\pm\frac{1}{2}}]$-linear function $H_W \to R(W)(\!(\Q^{\frac{1}{2}})\!)$.
After multiplying by a normalization factor that clears denominators, we get a function
\begin{align}
\Tr{-} : H_W \to R(W)[\Q^{\pm\frac{1}{2}}]
\end{align}
that admits an especially simple formula in terms of the character theory of $H_W$.
If we write $\hat{W}$ for the set of irreducible characters of $W$, then it is
\begin{align}
\Tr{\beta} = \sum_{\phi, \psi \in \hat{W}} {\{\phi, \psi\}}\phi_{\Q}(\beta)\psi,
\end{align}
where $\phi_{\Q} : H_W \to \bb{Z}[\Q^{\pm\frac{1}{2}}]$ corresponds to $\phi : \bb{Z}[W] \to \bb{Z}$ under Tits's deformation theorem and $\{-, -\}$ is a $\bb{Q}$-valued pairing on $\hat{W}$ derived from Lusztig's work on the character theory of $G(\bb{F})$ \cite{lusztig_1984}.
The function $\Tr{-}$ can be used to recover Gomi's Markov trace, and our formula is closely related to his formula.
In particular, we will show that various properties of the Markov trace extend to $\Tr{-}$: its rationality, the palindromic symmetry of its coefficients, and its compatibility with parabolic subgroups of $W$.

The function $w \mapsto \Tr{\sigma_w}$, evaluated on elements $w \in W$ of minimal Bruhat length in their conjugacy classes, recovers the function that Lusztig--Yun introduce in \cite{ly_2019}, up to a normalization factor.
Thus it also encodes the map from conjugacy classes in $W$ to unipotent classes in $G$ that Lusztig constructs in \cite{lusztig_2011}.

The definition of $\Tr{-}$ can be extended to any finite Coxeter group $W$, as a consequence of \cite{lusztig_1993, lusztig_1994, malle}.
In this generality, we will relate special values of $\Tr{-}$ to the cyclotomic block theory of $H_W$, \emph{i.e.}, to the representation theory of the ring obtained from $H_W$ by specializing $\Q^{\frac{1}{2}}$ to a root of unity.
To this end, we need a central element $\pi \in \Br_W^+$ known as the full twist.
We will say that an element $\beta \in \Br_W$ is periodic of slope $\slope = \frac{m}{n} \in \bb{Q}$ iff 
\begin{align}
\beta^n = \pi^m.
\end{align}
For such $\beta$, we give an explicit formula for $\Tr{\beta}$ in terms of constants of the form $\UDeg_\phi(e^{2\pi i\slope})$, where $\UDeg_\phi(\Q)$ is the so-called generic or unipotent degree of $\phi$ \cite[\S{8.1.8}]{gp}.
The constants $\UDeg_\phi(e^{2\pi i\slope})$ are integers that roughly control the sizes of the blocks of $H_W(e^{\pi i\slope}) = H_W|_{\Q^{1/2} = e^{\pi i\slope}}$.

As an application, we relate $[\sf{AH}(\cal{R}(\beta))]_{\Q}$ to the representations of a symplectic reflection algebra $\bb{D}_\slope^\rat \supseteq \bb{A}_W$ known as the rational Cherednik algebra or rational double affine Hecke algebra (DAHA).
We can view it as a deformation of $\bb{D}_0^\rat = \QL[W] \ltimes \cal{D}(\bb{V})$ that depends on $\slope$, where $\cal{D}(\bb{V})$ is the Weyl algebra of differential operators on $\bb{V}$.
Like a semisimple Lie algebra, $\bb{D}_\slope^\rat$ admits a ``category $\sf{O}$'' of well-behaved modules, except that the simple objects are parametrized by $\hat{W}$.
We respectively write $L_\slope(\phi)$ and $\Delta_\slope(\phi)$ for the simple and Verma objects indexed by $\phi \in \hat{W}$.
There is a functor from category $\sf{O} = \sf{O}_\slope$ to the category of $H_W(e^{\pi i\slope})$-modules, known as the Knizhnik--Zamolodchikov functor, that sends $\Delta_\slope(\phi)$ to the $H_W(e^{\pi i\slope})$-module associated with $\phi_{\Q}$.

Any object $M \in \sf{O}_\slope$ gives rise to a graded character $[M]_{\Q} \in R(W)(\!(\Q^{\frac{1}{2}})\!)$.
When $\beta$ is periodic of slope $\slope$, we can express $[\sf{AH}(\cal{R}(\beta))]_{\Q}$ as an explicit virtual sum of the Verma characters $[\Delta_\slope(\phi)]_{\Q}$.
At special values of $\slope$, we can say more.
For instance, at so-called cuspidal values, it simplifies to an explicit \emph{nonnegative} sum of \emph{simple} characters $[L_\slope(\phi)]_{\Q}$.

When $W$ is a Weyl group, we conjecture that this relation between periodic braids and the rational DAHA is categorified by a $\bb{D}_\slope^\rat$-action on a modified version of $\sf{AH}(\cal{R}(\beta))$.
When $W = S_n$, our work recovers a theorem of Gorsky--Oblomkov--Rasmussen--Shende relating the HOMFLY polynomials of torus knots to the characters $[L_\slope(1)]_{\Q}$, where $1 \in \hat{W}$ is the trivial character of $W$ \cite{gors}.

\subsection{Outline and Results}

Throughout the paper, we use stacky conventions for (co)homological degree and assume that $\bb{F}$ has large characteristic.
In what follows, we also assume that $\bb{V}^W = 0$.
We set $r = \dim \bb{V}$.

In Section \ref{sec:mixed}, we review background on $G_0$-varieties and their derived categories of $\ell$-adic complexes, with special attention to Deligne's yoga of weights.
Then we introduce the formalism for realization functors that will serve our purposes.
In particular, we introduce a condition called $\SHom$-purity that slightly generalizes the condition called Frobenius invariance in \cite{rider}.

In Section \ref{sec:hecke}, we give explicit definitions of the categories $\sf{C}(\cal{B}_0 \times \cal{B}_0)$, $\sf{C}(G_0)$, $\sf{C}_{\hat{W}}(G_0)$, $\sf{C}(\cal{U}_0)$, $\sf{C}_{\hat{W}}(\cal{U}_0)$, as well as the functors $\sf{CH}$, $i^\ast\langle -r\rangle$, $\sf{PR}$ that relate them.
We also introduce the objects $\cal{G}$ and $\cal{S}$, along with their summands.
We give a thorough exposition of the geometric interpretation of the Hecke category that we were unable to find elsewhere.
The main novelty in this section is a proof that $\sf{C}(\cal{U}_0)$ is $\SHom$-pure, using results from \cite{lusztig_1984_intersection, rr}.

In Section \ref{sec:a_w}, we introduce the ring $\bb{A}_W$ and the monoidal trace $\sf{AH}$.

In Section \ref{sec:steinberg}, we first review equivariant Borel--Moore homology, then introduce the Steinberg scheme $\cal{Z} \to \cal{U}$ and its multiplicative analogue $\cal{Z}' \to G$.
Using an inductive argument inspired by an unpublished note of Riche, we establish a comparison between the Borel--Moore homologies of $\cal{Z}$ and $\cal{Z}'$, which leads to the following theorem:

\begin{mainthm}\label{thm:pure-iso}
The composition $\bigoplus_n \gr_n^{\bm{w}} \SEnd^n(\cal{G}) \to \SEnd^\ast(\cal{G}) \xrightarrow{i^\ast} \SEnd^\ast(\cal{S})$ is an isomorphism of graded $\QL$-algebras.
\end{mainthm}

In Section \ref{sec:varieties}, we first review the positive braid monoid $\Br_W^+ \subseteq \Br_W$.
Then we review the map $\beta \mapsto O(\beta)$ studied by Deligne and Brou\'e--Michel and the map $\beta \mapsto \cal{R}(\beta)$ introduced by Rouquier.
Using our realization formalism, we rederive the folklore comparison between the variety $O(\beta)$ and the complexes $\cal{R}(\beta)$, $\cal{R}(\beta^{-1})$ for positive $\beta$.
Using the work in Section \ref{sec:a_w}, we prove:

\begin{mainthm}\label{thm:a_w-to-varieties}
We have
\begin{align}
\sf{AH}^{i, j}(\cal{R}(\beta))
\simeq \gr_{j + 2r}^{\bm{w}} \ur{H}_{-(i + 2r)}^{!, G}(\cal{Z}(\beta))
\end{align}
for all $\beta \in \Br_W^+$ and $i, j \in \bb{Z}$.
\end{mainthm}

In Section \ref{sec:kr}, we review Khovanov--Rozansky homology and its geometric interpretation in terms of the horocycle correspondence due to Webster--Williamson.
We use the work of Sections \ref{sec:hecke}, \ref{sec:a_w}, and \ref{sec:steinberg} to show:

\begin{mainthm}\label{thm:a_w-to-kr}
With respect to the grading conventions on $\sf{AH}$ in Section \ref{sec:a_w} and on $\sf{HHH}$ in Section \ref{sec:kr}, we have
\begin{align}
(\sf{HHH}^{i, i + j, k})^\vee
	\simeq \Hom_W(\Alt^i(\bb{V}), \sf{AH}^{j + k, j}(-))
\end{align}
as functors $\sf{H}_W \to \sf{Vect}$ for all $i, j, k \in \bb{Z}$.
\end{mainthm}

Let $\cal{X}(\beta) = \cal{X}(\beta)_0 \otimes \bar{\bb{F}}$ be the fiber of $\cal{U}(\beta) = \cal{U}(\beta)_0 \otimes \bar{\bb{F}}$ above $1 \in \cal{U}$.
Using well-known properties of the $W$-action on the Springer resolution, we deduce:

\begin{maincor}\label{cor:kr-to-varieties}
We have
\begin{align}
\sf{HHH}^{0, j, k}(\cal{R}(\beta))^\vee
	&\simeq \gr_{j + 2r}^{\bm{w}} \ur{H}_{-(j + k + 2r)}^{!, G}(\cal{U}(\beta)),\\
\sf{HHH}^{r, r + j, k}(\cal{R}(\beta))^\vee
	&\simeq \gr_{j + 2(r - N)}^{\bm{w}} \ur{H}_{-(j + k + 2(r - N))}^{!, G}(\cal{X}(\beta)).
\end{align}
for all $\beta \in \Br_W^+$, where $N = \dim \cal{B}$.
\end{maincor}

Combining the corollary with a ``Serre duality'' result for $\sf{HHH}$ \cite{ghmn}, we find that for $G = \SL_n$ and any $\beta \in \Br_W^+$, there is a weight-preserving isomorphism between the equivariant Borel--Moore homologies of the varieties $\cal{U}(\beta)$ and $\cal{X}(\beta\pi)$.
We revisit this observation in \S\ref{subsubsec:serre-duality}.

In Section \ref{sec:decat}, we introduce the function $\Tr{-} : H_W \to R(W)[\Q^{\pm\frac{1}{2}}]$ in terms of the background covered in Appendix \ref{sec:coxeter}.
Its definition can be extended to any finite Coxeter group $W$.
In the case of Weyl groups, we show that $\Tr{-}$ is the decategorification of $\sf{AH}$ up to a normalization factor.
The key ingredient is a multiplicity formula for character sheaves that involves $\{-, -\}$, due to Lusztig \cite[Cor.\ 14.11]{lusztig_1985_3} \cite[Thm.\ 23.1]{lusztig_1986}.

\begin{mainthm}\label{thm:decat}
If $W$ is a Weyl group, and $\cal{K} \in \sf{H}_W$ decategorifies to $\beta \in H_W$, then
\begin{align}
[\sf{AH}^\vee(\cal{K})]_{\Q} = \Tr{\beta} \cdot [\Sym(\bb{V})]_{\Q},
\end{align}
where $(\sf{AH}^\vee)^{i, j} = (\sf{AH}^{i, j})^\vee$ for all $i, j$ and $[\Sym(\bb{V})]_{\Q} = \sum_i {\Q^i} \Sym^i(\bb{V})$.
\end{mainthm}

Next, we explain the precise relationship between $\Tr{-}$ and the Markov trace studied by Ocneanu and Gomi:
See Proposition \ref{prop:markov}.
It is most conveniently stated in terms of the class function $\Tr{-}^0 : \Br_W \to R(W)(\!(\Q^{\frac{1}{2}})\!)$ defined by
\begin{align}
\Tr{\beta}^0 = (-\Q^{\frac{1}{2}})^{|\beta|}\Tr{\beta} \cdot \varepsilon [\Sym(\bb{V})]_{\Q},
\end{align}
where $|\beta|$ is the writhe of $\beta$ and $\varepsilon$ is the sign character of $W$ (see \S\ref{subsec:representation}).

Using properties of $\{-, -\}$, we show that certain rationality and symmetry properties of the Markov trace generalize to $\Tr{-}$ and $\Tr{-}^0$.
For Weyl groups, the strongest form of the rationality property will be Theorem \ref{thm:pole}, which we prove in Section \ref{sec:dl}.
For general Coxeter groups, the remaining properties are:

\begin{mainthm}\label{thm:symmetry}
For all $\beta \in \Br_W$, we have 
\begin{align}
\Tr{\beta} 
&\in \Q^{-\frac{|\beta|}{2}}R(W)[\Q] \cap R(W)[\Q^{\frac{1}{2}} - \varepsilon \Q^{-\frac{1}{2}}],\\
\Tr{\beta}^0
&\in 
	R(W)[\![\Q]\!] \cap R(W)(\Q) \cap
	(\Q^{\frac{1}{2}})^{|\beta| - r}
	R(W)[(\Q^{\frac{1}{2}} - \Q^{-\frac{1}{2}})^{\pm 1}].
\end{align}
In particular, $\Tr{\beta}$ and $\Tr{\beta}^0$ are invariant under $\Q^{\frac{1}{2}} \mapsto -\varepsilon\Q^{-\frac{1}{2}}$ and $\Q^{\frac{1}{2}} \mapsto -\Q^{-\frac{1}{2}}$, respectively.
Moreover, $\Tr{\beta}^0$ has degree $|\beta| - r$ as a rational function in $\Q$.
\end{mainthm}

In Section \ref{sec:dl}, we again assume that $W$ is the Weyl group of $G$.
After reviewing background on Deligne--Lusztig virtual representations, we use their properties to express $\Tr{-}^0$ in terms of the $\bb{F}$-point counts of the fibers of the map $\cal{U}(\beta) \to \cal{U}$.
Namely, we need results of Deligne--Lusztig--Springer \cite[Thm.\ 6.9]{dl} \cite[Thm.\ 5.6]{springer_1976}, Kazhdan \cite{kazhdan}, Lusztig \cite[Cor.\ 4.24-4.25]{lusztig_1984}, and Geck--Lusztig \cite[Prop.\ 1.3]{geck}.
To state the formula, let $\cal{B}_g$ denote the Springer fiber above $g \in G$, and let
\begin{align}
Q_g(w) = \sum_n
{\tr(wF \mid \ur{H}^{2n}(\cal{B}_g, \QL))}
\end{align}
for all $w \in W$.
When $G$ is a \emph{special} algebraic group in the sense of Serre \cite{grothendieck_1958}, the formula below is a kind of virtual weight series for $[\cal{Z}(\beta)/G]$, so it is a corollary of Theorem \ref{thm:decat} via \cite[Appendix]{hr} and \cite[Thm.\ 4.10]{joyce} (see Remark \ref{rem:special}).
However, not all algebraic groups are special.

\begin{mainthm}\label{thm:virtual}
For all $\beta \in \Br_W^+$, we have
\begin{align}
\Tr{\beta}^0|_{\Q = q}
&= \frac{(-1)^{r - |\beta|}}{|G^F|}
	\sum_{u \in \cal{U}^F}
	|\cal{U}(\beta)_u^F| Q_u
\end{align}
in $\bb{Q} \otimes R(W)$, where $q = |\bb{F}|$.
\end{mainthm}

For $G = \PGL_n$, taking $\varepsilon$-isotypic components on both sides and applying the Serre-duality result of \cite{ghmn}, or K\'alm\'an's earlier result \cite{kalman}, recovers a theorem of Shende--Treumann--Zaslow, matching the ``lowest'' $a$-degree of the HOMFLY series of the link closure of $\beta$ with the stacky point count of $[\cal{X}(\beta\pi)/G]$ \cite[Thm.\ 6.34]{stz}.
In fact, K\'alm\'an's result generalizes to a statement about Gomi's Markov trace.
For general $G$, this generalization and Theorem \ref{thm:virtual} together imply \cite[Prop.\ 2.2]{lusztig_2021}, which Lusztig attributes to Kawanaka \cite{kawanaka}: See Corollary \ref{cor:virtual}.

Recall that a subgroup $W' \subseteq W$ is parabolic iff it is the stabilizer of a vector in $\bb{V}$.
In this case, there is an inclusion of Iwahori--Hecke algebras $H_{W'} \subseteq H_W$.
Using Theorem \ref{thm:virtual} and the compatibility of Deligne--Lusztig characters with parabolic induction, we show that $\Tr{-}$ is similarly compatible with parabolic induction.

\begin{mainthm}\label{thm:induction}
If $W$ is a Weyl group of rank $r$ and $W' \subseteq W$ is a parabolic subgroup of rank $r'$, then we have commutative diagrams:
\begin{equation}
\begin{tikzpicture}[baseline=(current bounding box.center), >=stealth]
\matrix(m)[matrix of math nodes, row sep=2.5em, column sep=3em, text height=2ex, text depth=0.5ex]
{	H_{W'}
		&R(W')[\Q^{\pm \frac{1}{2}}]\\
	H_W
		&R(W)[\Q^{\pm \frac{1}{2}}]\\
		};
\path[->,font=\scriptsize, auto]
(m-1-1)		edge node{$\Tr{-}$} (m-1-2)
			edge (m-2-1)
(m-1-2) 	edge node{$\Ind_{W'}^W$} (m-2-2)
(m-2-1) 	edge node{$\Tr{-}$} (m-2-2);
\end{tikzpicture}
\qquad
\begin{tikzpicture}[baseline=(current bounding box.center), >=stealth]
\matrix(m)[matrix of math nodes, row sep=2.5em, column sep=3em, text height=2ex, text depth=0.5ex]
{	\Br_{W'}
		&R(W')[\![\Q]\!]\\
	\Br_W
		&R(W)[\![\Q]\!]\\
		};
\path[->,font=\scriptsize, auto]
(m-1-1)		edge node{$\Tr{-}^0$} (m-1-2)
			edge (m-2-1)
(m-1-2) 	edge node{$\frac{1}{(1 - \Q)^{r - r'}} \Ind_{W'}^W$} (m-2-2)
(m-2-1) 	edge node{$\Tr{-}^0$} (m-2-2);
\end{tikzpicture}
\end{equation}
In particular, $\Tr{\bb{1}}$ is the character of the regular representation of $W$, placed in $\Q$-degree zero, and $\Tr{\bb{1}}^0 = (1 - \Q)^{-r}\Tr{\bb{1}}$.
\end{mainthm}

To conclude Section \ref{sec:dl}, we bound the pole of $\Tr{-}^0$ at $\Q = 1$ using Theorem \ref{thm:virtual} and an idea of Lusztig from \cite[\S{5}]{lusztig_2011}.
For all $w \in W$, let $r(w)$ be the dimension of the fixed subspace $\bb{V}^w \subseteq \bb{V}$.

\begin{mainthm}\label{thm:pole}
If $W$ is a Weyl group and $\beta \in \Br_W^+$ maps to $w \in W$, then
\begin{align}
\Tr{\beta}^0 \in \frac{1}{(1 - \Q)^{r(w)}}\, R(W)[\Q].
\end{align}
That is, the pole of $\Tr{-}^0$ at $\Q = 1$ is at most $r(w)$.
\end{mainthm}

In the case where $\beta = \sigma_w$ for some $w \in W$, this result follows from work of Lusztig--Yun in \cite{ly_2019}.
Indeed, if $w \in W$ has minimal Bruhat length in its conjugacy class, then Theorem \ref{thm:virtual} essentially implies
\begin{align}
\Tr{\sigma_w}^0 = \frac{\Psi(w)}{(1 - \Q)^{r(w)}},
\end{align}
where $\Psi : W \to R(W)[\Q]$ is the function introduced in \emph{ibid.}

In Sections \ref{sec:periodic}-\ref{sec:daha}, we allow $W$ to be any finite Coxeter group and $\bb{V}$ any realization of $W$ such that $\bb{V}^W = 0$.
In Section \ref{sec:periodic}, we compute $\Tr{\beta}$ for periodic $\beta$.
The key input is that, essentially by Schur's lemma, the full twist $\pi$ acts on any simple $H_W$-module by a monomial in $\Q^{\frac{1}{2}}$.
This, together with a deformation argument of Jones \cite[\S{9}]{jones}, allows us to determine the character values $\phi_{\Q}(\beta)$ whenever some nonzero power of $\beta$ equals a power of the full twist.

\begin{mainthm}\label{thm:periodic}
If $\beta \in \Br_W$ is periodic of slope $\slope$, then
\begin{align}
\Tr{\beta}
=	\sum_{\phi \in \hat{W}}
	{\Q^{\slope\bb{c}(\phi)}}
	\UDeg_\phi(e^{2\pi i\slope}) 
	\phi,
\end{align}
where $\UDeg_\phi(\Q) \in \bb{Q}[\Q]$ and $\bb{c}(\phi) \in \bb{Z}$ are the unipotent degree and content of $\phi$, respectively (see \S\ref{subsec:degrees}, \S\ref{subsec:families}).
\end{mainthm}

In the setting of the above theorem, we will have $\UDeg_\phi(e^{2\pi i\slope}) \in \bb{Z}$ for all $\phi$ and $\slope$.
It will be convenient to set
\begin{align}
\Omega_\slope = \bigoplus_{\phi \in \hat{W}} \UDeg_\phi(e^{2\pi i\slope})\Delta_\slope(\phi),
\end{align}
viewed as a virtual $\bb{D}_\slope^\rat$-module.
In Section \ref{sec:daha}, we deduce:

\begin{maincor}\label{cor:periodic}
If $\beta \in \Br_W$ is periodic of slope $\slope$, then
\begin{align}
{[\Omega_\slope]_{\Q}}
=	(\Q^{\frac{1}{2}})^{r - |\beta|} \cdot \Tr{\beta}^0
\end{align}
in $R(W)[(\Q^{\frac{1}{2}} - \Q^{-\frac{1}{2}})^{\pm 1}]$.
\end{maincor}

In what follows, let $d_1, \ldots, d_r$ be the invariant degrees of the $W$-action on $\bb{V}$ \cite[149]{gp}.
We say that $\slope \in \bb{Q}$ is:
\begin{enumerate}
\item 	A singular slope iff $\slope \in \frac{1}{d_i}\bb{Z}$ for some $i$.
\item 	A regular slope iff $W$ contains a $e^{2\pi i\slope}$-regular element $w$.
		This means $w$ has an eigenvector in $\bb{V}$ with eigenvalue $e^{2\pi i\slope}$ and trivial $W$-stabilizer \cite{springer_1974}.
\item 	A regular elliptic slope iff, in (2), we can pick $w$ so that $\bb{V}^w = 0$.
\item 	A cuspidal slope iff $\slope \in \frac{1}{d_i}\bb{Z}$ for a unique index $i$.
\end{enumerate}
Conditions (1)-(4) are ordered by increasing strictness.
By work of Brou\'e--Michel \cite{bm_1997}, condition (2) is equivalent to the existence of a periodic braid of slope $\slope$.

In Section \ref{sec:daha}, we use tools from the representation theory of $\bb{D}_\slope^\rat$, especially the Knizhnik--Zamolodchikov functor from category $\sf{O}_\slope$ to the category of $H_W(e^{\pi i\slope})$-modules \cite{ggor}, to prove:

\begin{mainthm}\label{thm:slopes}
Let $\slope \in \bb{Q}$.
\begin{enumerate}
\item 	$L_\slope(1)$ occurs in $\Omega_\slope$ with multiplicity one.
\item 	If $\slope$ is a singular slope, then $\Omega_\slope$ is a virtual linear combination of torsion modules over the subring $\Sym(\bb{V}^\vee) \subseteq \bb{D}_\slope^\rat$.
\item 	If $\slope$ is a regular elliptic slope, $\slope > 0$, and $W$ is a Weyl group, then 
		\begin{align}
		[\Omega_\slope]_{\Q} \in R(W)[\Q^{\pm\frac{1}{2}}].
		\end{align}
		That is, $\Omega_\slope$ is a difference of finite-dimensional $\bb{D}_\slope^\rat$-modules.
\item 	If $\slope$ is a cuspidal slope and $\slope > 0$, then
		\begin{align}
		[\Omega_\slope]_{\Q} = \sum_{\psi \in \hat{W}_1}
		{[L_\slope(\psi)]_{\Q}},
		\end{align}
		where $\hat{W}_1 \subseteq \hat{W}$ is the set of elements of greatest $\bb{a}$-value within their  $e^{2\pi i\slope}$-blocks, among the $e^{2\pi i\slope}$-blocks of $H_W$ of defect $1$ (see \S\ref{subsec:cyclotomic}).
\end{enumerate}
\end{mainthm}

Parts (1) and (2) together imply a result of Dunkl--de Jeu--Opdam \cite[Thm.\ 4.9(1)]{djo}, while part (3) implies a result of Varagnolo--Vasserot \cite[Thm.\ 2.8.1(a)]{vv}.
The key machinery in the proof of part (4) is work of Geck on the block theory of $H_W$, which shows that when $\slope$ is cuspidal, the values $\UDeg_\phi(e^{2\pi i\slope})$ are computable and all belong to $\{-1, 0, 1\}$ \cite{geck}.
Part (4) implies:

\begin{maincor}\label{cor:cuspidal}
Keeping the hypotheses of Theorem \ref{thm:slopes}, suppose that $W$ is irreducible (see \S\ref{subsec:coxeter-group}) and at least one of the following is true:
\begin{itemize}
\item 	$W$ is not of type $E_8$ or $H_4$.
\item 	$\slope \notin \frac{1}{15}\bb{Z}$.
\end{itemize}
Then $[\Omega_\slope]_{\Q} = [L_\slope(1)]_{\Q}$.
\end{maincor}

For $W = S_n$, Theorem \ref{thm:a_w-to-kr}, Corollary \ref{cor:periodic}, and Corollary \ref{cor:cuspidal} together recover \cite[Thm.\ 1.1]{gors}, which was originally proved by explicit calculation in the ring of symmetric functions.

In Appendix \ref{sec:coxeter}, we review facts about the structure and representation theory of $W$ and $H_W$ that we need elsewhere, including the definitions of the pairing $\{-, -\}$ and the unipotent degrees $\UDeg_\phi$.

In Appendix \ref{sec:varieties-other}, we state precise comparisons between the schemes $\cal{U}(\beta)$, $\cal{Z}(\beta)$, $\cal{X}(\beta)$ and the varieties in the works of Lusztig \cite{lusztig_2011}, Shende--Treumann--Zaslow \cite{stz}, Mellit \cite{mellit}, Casals--Gorsky--Gorsky--Simental \cite{cggs}, and others.

In Appendix \ref{sec:examples}, we list values for the characters $\Tr{\beta}$ and $\Tr{\beta}^0$ in low rank, as well as an example where $\beta$ is an iterated torus braid.

\subsection{Conjectures and Questions}

\subsubsection{}

The appearance of the ring $\bb{A}_W$ suggests that $\sf{AH}$ is closely related to a functor on $\sf{H}_W$ called the horizontal trace, which \cite{qrs, gw, ghw} studied in the case where $W = S_n$ via diagrammatic rather than geometric methods.
If $\star$ denotes the monoidal product on $\sf{H}_W$, then the horizontal trace is roughly the universal additive functor under which $\cal{K} \star \cal{L}$ and $\cal{L} \star \cal{K}$ become isomorphic for all objects $\cal{K}, \cal{L} \in \sf{H}_W$.

In \cite{qrs}, Queffelec--Rose--Sartori show that the horizontal trace refines $\sf{HHH}$, and in \cite{gw}, Gorsky--Wedrich roughly show that it takes values in the homotopy category of $\bb{A}_W$-modules.
(The difference is that Gorsky--Wedrich take $\bb{V}$ to be the $n$-dimensional permutation representation of $S_n$, rather than its $(n - 1)$-dimensional irreducible summand.)

\begin{conj}
After we replace the semisimple group $G$ with the reductive group $\ur{GL}_n$, our procedure to construct $\sf{AH}$ yields (the objectwise dual of) (the homology of) the horizontal trace studied in \cite{gw}.
\end{conj}

\subsubsection{}\label{subsubsec:serre-duality}

We expect the geometry of $\cal{U}(\beta)$ to manifest a ``Serre duality'' phenomenon involving the full twist.
Gorsky--Hogancamp--Mellit--Nakagane showed that for $G = \SL_n$ and any $\beta \in \Br_W$, the lowest Hochschild degree of $\sf{HHH}(\cal{R}(\beta))$ and the highest Hochschild degree of $\sf{HHH}(\cal{R}(\beta \pi))$ coincide \cite{ghmn}, which categorifies a theorem of K\'alm\'an about the HOMFLY series \cite{kalman}.
Corollary \ref{cor:kr-to-varieties} leads us to expect a geometric explanation of this coincidence.

In earlier versions of this work, we hoped for the existence of an algebraic homeomorphism from $[\cal{X}(\beta\pi)/G]$ onto $[\cal{U}(\beta)/G]$ that would explain the weight-preserving isomorphism between their Borel--Moore homologies.
Unfortunately, as we will explain in future work, this conjecture already fails for $G = \SL_3$.

\begin{quest}
Is there a geometric explanation for the (shifted) isomorphism $\gr_\ast^{\bm{w}} \ur{H}_\ast^{!, G}(\cal{X}(\beta\pi)) \simeq \gr_\ast^{\bm{w}} \ur{H}_\ast^{!, G}(\cal{U}(\beta))$?
\end{quest}

\subsubsection{}

Oblomkov--Rozansky have recently proved that the Khovanov--Rozansky invariant of a \emph{knot} enjoys a bigraded symmetry that categorifies the $\Q^{\frac{1}{2}} \mapsto -\Q^{-\frac{1}{2}}$ symmetry of its HOMFLY series \cite{or}, resolving a conjecture from \cite{dgr}.
We expect this result to extend to $\sf{AH}$.
More precisely, we expect that for certain positive braids $\beta$, the stack $[\cal{Z}(\beta)/G]$ enjoys a ``curious Lefschetz'' property that implies a corresponding bigraded symmetry of the module $\sf{AH}(\cal{R}(\beta))$, by way of Theorem \ref{thm:a_w-to-varieties}.

Following Hausel--Rodriguez-Villegas \cite{hr}, if $X_0$ is (a stack quotient of) a variety over $\bb{F}$ of dimension $d$ and $X = X_0 \otimes \bar{\bb{F}}$, then we say that the curious Lefschetz property holds for $X$ iff there are isomorphisms
\begin{align}
\gr_{d - j}^{\bm{w}} \ur{H}_!^i(X, \QL)
\xrightarrow{\sim}
\gr_{d + j}^{\bm{w}} \ur{H}_!^{i + j}(X, \QL)
\end{align}
of compactly-supported cohomology for all $i$ and $j$, induced by the cup product with a fixed class in $\ur{H}_!^2$.
This means the series
\begin{align}
(\Q^{\frac{1}{2}})^{-d} \sum_{i, j}
	\Q^{\frac{j}{2}} t^i
	\gr_j^{\bm{w}} \ur{H}_!^i(X, \QL)
\end{align}
is invariant under $\Q^{\frac{1}{2}} \mapsto \Q^{-\frac{1}{2}} t^{-1}$.
Mellit has recently established this property for character varieties with structure group $G = \GL_n$ \cite{mellit}.
In the course of his proof, he also establishes the property for certain varieties closely related to the fibers of the maps $\cal{X}(\beta) \to \cal{B}$ (see Appendix \ref{sec:varieties-other}).
His results, as well as our Theorem \ref{thm:symmetry}, suggest evidence for:

\begin{conj}
For all $\beta \in \Br_W^+$, the curious Lefschetz property holds for $\cal{U}(\beta)$ and $\cal{Z}(\beta)$.
If, in addition, $\beta$ maps to an element $w \in W$ such that $\bb{V}^w = 0$, then the property also holds for their stack quotients by $G$.
\end{conj}

Note that for $W = S_n$, the hypothesis that $w$ satsfies $\bb{V}^w = 0$ is equivalent to the hypothesis that the link closure of $\beta$ is a knot.

In \cite{gh_2017}, Gorsky--Hogancamp introduce a refinement of $\sf{HHH}$ called \emph{$y$-ified homology} and denoted $\sf{HY}$.
They conjecture that the bigraded symmetry of the Khovanov--Rozansky invariant can be extended from knots to arbitrary links by replacing $\sf{HHH}$ with $\sf{HY}$.
In future work, we will relate $\sf{HY}$ to a \emph{monodromic} refinement of $\sf{AH}$.

\subsubsection{}

Lastly, we propose a (partial) categorification of Theorem \ref{thm:periodic}, which would strengthen some of the conjectures from \cite{gors} relating knot invariants and the rational DAHA. 

Let $W = S_n$ and $\slope = \frac{m}{n}$, with $m, n > 0$ coprime.
Gorsky, Oblomkov, Rasmussen, and Shende construct three filtrations on the simple $\bb{D}_\slope^\rat$-module $L_\slope(1)$ that are compatible with the grading on its character.
They conjecture that all three are the same filtration $\bm{f}_{\leq \ast}$, and that their theorem can be promoted to an isomorphism between the Khovanov--Rozansky homology of the $(m, n)$ torus knot and the sum of the $\Alt^i(\bb{V})$-isotypics of $\gr_\ast^{\bm{f}} L_\slope(1)$.
One of the candidate filtrations is constructed geometrically, as we now explain.

For any almost-simple, simply-connected group $G$ and regular elliptic slope $\slope$, Oblomkov--Yun have related the DAHA module $L_\slope(1)$ to a complex projective $\bb{G}_m$-variety that we will denote $\cal{B}_\slope^\aff$ \cite{oy}. 
In technical terms, $\cal{B}_\slope^\aff$ is an affine Springer fiber of Iwahori type with structure group $G$.
By means of the topology of a generalized Hitchin fibration, its equivariant cohomology is endowed with an increasing filtration $\bm{p}_{\leq \ast}$ called the perverse filtration.

There is a flat deformation of $\bb{A}_W$ over $\QL[\epsilon]$ that we will denote $\bb{A}_{W, \epsilon}$, known as Lusztig's graded or degenerate affine Hecke algebra and satisfying $\bb{A}_{W, \epsilon}|_{\epsilon = 0} = \bb{A}_W$.
Oblomkov--Yun construct an $\bb{A}_{W, \epsilon}$-action on $\ur{H}_{\bb{G}_m}^\ast(\cal{B}_\slope^\aff, \QL)$ that sends $\epsilon$ to the equivariant parameter.
Similarly, there is a flat deformation of $\bb{D}_\slope^\rat$ over $\QL[\epsilon]$ to a graded algebra that we will denote $\bb{D}_{\slope, \epsilon}^\rat$, satisfying $\bb{D}_{\slope, \epsilon}^\rat|_{\epsilon = 1} = \bb{D}_\slope^\rat$ \cite[\S{4.2.4}]{oy}.
The $\bb{A}_{W, \epsilon}$-action on $\ur{H}_{\bb{G}_m}^\ast(\cal{B}_\slope^\aff, \QL)$ induces an $\bb{A}_W$-action that extends to a $\bb{D}_{\slope, \epsilon}^\rat$-action.
For $G = \SL_n$ and regular elliptic $\slope$, it turns out that
\begin{align}\label{eq:oy}
\gr_\ast^{\bm{p}}\ur{H}_{\bb{G}_m}^\ast(\cal{B}_\slope^\aff, \QL)|_{\epsilon = 1} \simeq L_\slope(1)
\end{align}
as $\bb{D}_\slope^\rat$-modules, via an isomorphism that sends the $\bm{p}$-grading to the grading on $[L_\slope(1)]_{\Q}$.
The authors of \cite{gors} define the geometric candidate for $\bm{f}_{\leq \ast}$ to be the difference of the perverse and cohomological filtrations on the left-hand side of \eqref{eq:oy}, up to a shift.

In forthcoming work, we will study a conjectural \emph{transcendental} retraction from the stack $[\cal{Z}(\beta)/G]$ to the variety $\cal{B}_\slope^\aff$ for any periodic braid $\beta \in \Br_W^+$ of regular elliptic slope $\slope$.
It is expected to be a kind of nonabelian Hodge correspondence, where $[\cal{Z}(\beta)/G]$ plays the role of the ``Betti'' side and $\cal{B}_\slope^\aff$ plays the role of the ``Dolbeault'' side.
In particular, the retraction should transport the doubled weight filtration $\bm{w}_{\leq 2\ast}$ on the compactly-supported cohomology of $[\cal{Z}(\beta)/G]$ to the perverse filtration $\bm{p}_{\leq \ast}$ on the cohomology of $\cal{B}_\slope^\aff$, up to a degree shift.
We are thus led to conjecture a Betti origin for the filtration $\bm{f}_{\leq \ast}$.

As evidence, note that the $G$-action on $\cal{Z}(\bb{1})$ extends to a $(G \times \bb{G}_m)$-action.
The isomorphism $\bb{A}_W \simeq \ur{H}_{-\ast}^{!, G}(\cal{Z}(\bb{1}))$ can be lifted to an isomorphism
\begin{align}
\bb{A}_{W, \epsilon} \simeq \ur{H}_{-\ast}^{!, G \times \bb{G}_m}(\cal{Z}(\bb{1})),
\end{align}
as shown in \cite{lusztig_1988}.
Unfortunately, as Zhiwei Yun pointed out to me, there is no obvious $(G \times \bb{G}_m)$-action on $\cal{Z}(\beta)$ for other $\beta$, even when $\beta$ is periodic.
Even so, I hope for the following:

\begin{conj}\label{conj:daha}
There is a flat deformation of $\sf{AH}$ over $\QL[\epsilon]$ to a functor 
\begin{align}
\sf{AH}_\epsilon : \sf{H}_W^\op \to \Mod_2^b(\bb{A}_{W, \epsilon})
\end{align}
such that for all $\cal{K} \in \sf{H}_W$:
\begin{enumerate}
\item 	The $\bb{A}_W$-action on $\sf{AH}(\cal{K})$ extends to a $\bm{w}$-\emph{filtered} $\bb{A}_{W, \epsilon}$-action on $\sf{AH}_\epsilon(\cal{K})$, where $\bm{w}$ denotes the second grading.
By \cite[Prop.\ 4.3.1]{oy}, the latter action induces a $\bm{w}$-\emph{graded} $\bb{A}_W$-action on $\sf{AH}_\epsilon(\cal{K})$.
\end{enumerate}
If $\beta \in \Br_W^+$ is periodic of regular elliptic slope $\slope$, then:
\begin{enumerate}\setcounter{enumi}{1}
\item 	The $\bb{A}_W$-action on $\sf{AH}_\epsilon(\cal{R}(\beta))$ extends to a $\bb{D}_{\slope, \epsilon}^\rat$-action.
\item 	The $\bb{D}_\slope^\rat$-module formed by $\sf{AH}_1(\cal{R}(\beta))$ contains $L_\slope(1)$ as a summand with multiplicity one.
\item 	The filtration $\bm{f}_{\leq \ast}$ on $L_\slope(1)$ defined in \cite{gors} arises from the difference of the two gradings on $\sf{AH}_1(\cal{R}(\beta))$.
\end{enumerate}
\end{conj}

Note that in part (1), the differences between the various specializations of $\epsilon$ would still vanish at the level of graded characters: 
Thus, $[\sf{AH}_1^\vee(\cal{R}(\beta))]_{\Q}$ would remain the same as $[\sf{AH}^\vee(\cal{R}(\beta))]_{\Q}$.

\begin{quest}
Is there a $\bb{G}_m$-action on the stack $[\cal{Z}(\beta)/G]$ that could give rise to the flat deformation of $\sf{AH}$ in Conjecture \ref{conj:daha}?
If not, does there still exist a non-algebraic $\bb{C}^\times$-action in the complex setting?
\end{quest}

\subsection{Acknowledgments}

I warmly thank Olivier Dudas, Ben Elias, Pavel Etingof, Pavel Galashin, Victor Ginzburg, Eugene Gorsky, Quoc H\^{o}, Matt Hogancamp, Oscar Kivinen, Thomas Lam, George Lusztig, Andrei Negu\c{t}, Ch\^{a}u Ng\^{o}, Emily Norton, Alexei Oblomkov, Jay Taylor, David Treumann, and Zhiwei Yun for their interest and for many helpful discussions.

Olivier Dudas, Emily Norton, and Jay Taylor provided references to the literature that were helpful for Sections \ref{sec:decat}, \ref{sec:daha}, and \ref{sec:dl}, respectively. 

Discussions with Quoc H\^{o} and Zhiwei Yun were very helpful to me when I was working out the formalism for Sections \ref{sec:mixed}-\ref{sec:hecke}.

George Lusztig patiently helped me to correct many attributions throughout the paper, especially in Section \ref{sec:dl} and Appendix \ref{sec:varieties-other}, as well as other errata.
The proof of Theorem \ref{thm:induction} is very close to an argument in \cite{ly_2019}.

This work expands and revises a portion of my PhD thesis at the University of Chicago, advised by Ch\^{a}u Ng\^{o} and Victor Ginzburg.
I am grateful to Laura Ball, Oscar Kivinen, and Ch\^{a}u Ng\^{o} for their feedback on those chapters.
I am also grateful to Victor Ginzburg, Eugene Gorsky, and Andrei Negu\c{t} for their encouragement much earlier in the research process.

My work was supported by a Graduate Research Fellowship (grant \#1746045) and a Mathematical Sciences Postdoctoral Fellowship (grant \#2002238) from the National Science Foundation.
In graduate school, I also benefited from workshops held at the Mathematical Sciences Research Institute (MSRI) and the American Institute for Mathematics (AIM).

\newpage
\section{Purity, Mixedness, and Realization}\label{sec:mixed}

\subsection{}

Fix a finite field $\bb{F}$ and an algebraic closure $\bar{\bb{F}}$.
Throughout the paper, \emph{we assume that the characteristic of $\bb{F}$ is large}.
Recall that the Galois group of $\bar{\bb{F}}$ over $\bb{F}$ is pro-generated by the automorphism $a \mapsto a^q$.
We write $F$ for the inverse generator, also known as the geometric Frobenius element.

Fix a root datum $(\Phi, \bb{X}, \Phi^\vee, \bb{X}^\vee)$.
Let $G$ be the reductive group over $\bar{\bb{F}}$ defined by the root datum, and let $G_0$ be its split form over $\bb{F}$, so that $G = G_0 \otimes_\bb{F} \bar{\bb{F}}$.
In some later sections, we will require $G$ to be semisimple, not merely reductive.

\subsection{}

Let $X_0$ be a $G_0$-variety over $\bb{F}$, equipped with a fixed $G_0$-stable stratification by locally closed subvarieties.
Let $X = X_0 \otimes \bar{\bb{F}}$.
In general, we use a subscript $0$ to denote varieties over $\bb{F}$, and omit the subscript to denote their pullback to $\bar{\bb{F}}$.

Fix a prime $\ell$ invertible in $\bb{F}$.
We write 
\begin{align}
\sf{D}_G^b(X) &= \sf{D}_G^b(X, \QL),\\
\emph{resp.}\qquad
\sf{D}_{G, m}^b(X_0) &= \sf{D}_{G, m}^b(X_0, \QL),
\end{align}
for the bounded derived category of complexes, \emph{resp.}\ mixed complexes, of $\QL$-sheaves on the lisse-\'etale site of the stack $[X/G]$, \emph{resp.}\ $[X_0/G_0]$, with constructible cohomology with respect to the stratification.
Mixed complexes were introduced by Deligne in the non-stacky, \'etale setting \cite{deligne_1980}; we refer to \cite[Ch.\ 5]{bbd} for an overview.
We fix a square root $q^{\frac{1}{2}} \in \QL^\times$ to define a half-Tate twist $(\frac{1}{2})$.
The \dfemph{shift-twist} $\langle 1 \rangle = [1](\frac{1}{2})$ is a weight-preserving endofunctor of $\sf{D}_{G, m}^b(X_0)$.
We write 
\begin{align}
\xi : \sf{D}_{G, m}^b(X_0) \to \sf{D}_G^b(X)
\end{align}
for the pullback functor.

In \cite{lo_2008}, Laszlo--Olsson extended the Grothendieck--Verdier yoga of the six operations to lisse-\'etale sheaves on Artin stacks locally of finite type over a finite-dimensional affine excellent base.
In \cite{lo_2009}, they did the same for the theory of $t$-structures and perverse sheaves developed in \cite[Ch.\ 1-2]{bbd}.
In \cite{sun}, S.\ Sun extended the theory of mixed complexes to this setting, under the hypothesis that \emph{the stabilizers of the stack are affine algebraic groups}.
For all choices of $X_0$ that we will consider, this hypothesis will hold for $[X_0/G_0]$ and $[X/G]$.
Sun's work covers all of the results from \cite[Ch.\ 5]{bbd} that we need.

Throughout this paper, all maps between $G_0$-varieties or $G$-varieties will be equivariant and compatible with the relevant stratifications.
The six operations on $\sf{D}_G^b$ and $\sf{D}_{G, m}^b$ are always assumed to be derived:
That is, we will drop $\ur{R}$ and $\ur{L}$ from our notation for derived functors.

\begin{rem}\label{rem:bounded-below}
The categories $\sf{D}_{G, m}^b(X_0)$ and $\sf{D}_G^b(X)$ need not be closed under internal Hom.
But for any $K, L \in \sf{D}_{G, m}^b(X_0)$, the object $\SHom(K, L)$ does exist in the \emph{bounded-below} derived category of mixed complexes over $[X_0/G_0]$, and similarly with $\sf{D}_G^b(X)$ in place of $\sf{D}_{G, m}^b(X_0)$ \cite{lo_2008}.
We will not comment further on this issue since it will not raise problems, except to note that the hypercohomology of $\SHom(K, L)$ will always be bounded below.
\end{rem}

Let $\point_0 = \Spec \bb{F}$, and let $a : X_0 \to \point_0$ be the structure map.
For all objects $K, L \in \sf{D}_{G, m}^b(X_0)$, we write
\begin{align}
\SHom^n(K, L) &= \ur{H}^n(\xi a_\ast \SHom(K, L)).
\end{align}
Similarly, we write $\SEnd^n(K) = \SHom^n(K, K)$.
We can view $a_\ast \SHom(K, L)$ as a sheaf over $\point_0$, so its pullback along $\xi$ can be viewed as a sheaf over $\point$ endowed with an $F$-equivariant structure.
In particular, $\SHom^n(K, L)$ is the vector space $\Hom^n(\xi K, \xi L)$ endowed with an $F$-action.
The $\QL[F]$-module formed by $\SHom^\ast(K, L)$ is related to the vector space $\Hom^\ast(K, L)$ by short exact sequences
\begin{align}\label{eq:bbd}
0 \to \SHom^{n - 1}(K, L)_F \to \Hom^n(K, L) \to \SHom^n(K, L)^F \to 0
\end{align}
for all $n$ \cite[(5.1.2.5)]{bbd}.
The morphism $\Hom^n(K, L) \to \Hom^n(\xi K, \xi L)$ factors through the morphism $\Hom^n(K, L) \to \SHom^n(K, L)^F$ above.

Throughout this paper, all $\QL[F]$-modules will arise from mixed complexes, and all of their eigenvalues under the action of $F$ will be pure of integer weight in the sense of \cite[\S{1.2.1}]{deligne_1980}.
We write $\bm{w}_{\leq \ast}$ to denote the weight filtration on such a module $M$.
Explicitly, $\bm{w}_{\leq n} M \subseteq M$ is the sum of the generalized eigenspaces of $F$ where the eigenvalues are pure of weight $\leq n$.

\subsection{}

We will introduce a mild generalization of the realization functors used by Rider in \cite{rider}.
Our formalism also resembles the ones used by Bezrukavnikov--Yun in \cite[Appendix B]{by} and Lusztig--Yun in \cite[\S{9}]{ly_2020}, though we do not need to develop ours as extensively.
All of these formalisms are ultimately based on the work of Beilinson in \cite{beilinson}.

Let $\sf{C}(X_0) \subseteq \sf{D}_{G, m}^b(X_0)$ be a full, additive subcategory.
We say that $\sf{C}(X_0)$ is \dfemph{realizable} iff, for all $K, L \in \sf{C}(X_0)$, we have
\begin{align}\label{eq:realizable}
n \neq 0 \implies \gr_0^{\bm{w}} \SHom^n(K, L) = 0,
\end{align}
where $\gr_0^{\bm{w}}$ denotes the $0$th graded piece of the weight filtration.
(Compare to Assumption $\mathrm{C}_1$ in \cite[\S{B.1}]{by}.)
In practice, the objects of $\sf{C}(X_0)$ will be direct sums of pure shifted perverse sheaves.

\begin{lem}
If $\sf{C}(X_0)$ is realizable, then
\begin{align}\label{eq:zero-one}
n \neq 0, 1 \implies \Hom^n(K, L) = 0
\end{align}
for all $K, L \in \sf{C}(X_0)$.
\end{lem}

\begin{proof}
If $n \neq 0, 1$, then $\SHom^{n - 1}(K, L)_F,\,\SHom^n(K, L)^F = 0$ by the realizability of $\sf{C}(X_0)$, so $\Hom^n(K, L) = 0$ by the exactness of \eqref{eq:bbd}.
\end{proof}

We write $\underline{\sf{C}}(X_0)$ for the category in which the objects are the same as in $\sf{C}(X_0)$, but the morphisms are given by
\begin{align}
\Hom_{\underline{\sf{C}}(X_0)}(K, L) = \gr_0^{\bm{w}} \SHom^0(K, L).
\end{align}
By construction, it is a full, additive subcategory of $\underline{\sf{D}}_{G, m}^b(X_0)$.
There is a functor 
\begin{align}
\kappa : \sf{C}(X_0) \to \underline{\sf{C}}(X_0)
\end{align}
induced by the composition of morphisms 
\begin{align}\label{eq:c-to-underline-c}
\Hom(K, L) \to \SHom(K, L)^F \to \bm{w}_{\leq 0}\,\SHom(L, K) \to \Hom_{\underline{\sf{C}}(X_0)}(K, L).
\end{align}
If $\sf{C}(X_0)$ is realizable, then \eqref{eq:c-to-underline-c} is injective by the exactness of \eqref{eq:bbd}, so in this case, $\kappa$ is faithful.

\subsection{}

Given any additive category $\sf{A}$, we write $\sf{C}^b(\sf{A})$ for the category of bounded complexes of objects of $\sf{A}$ and $\sf{K}^b(\sf{A})$ for the homotopy category of $\sf{C}^b(\sf{A})$.
We write $[1]_\triangle$ for the endofunctors of these categories that shift each complex down by one degree.

\emph{Throughout the rest of this subsection, we assume that $\sf{C}(X_0)$ is realizable}.
We will construct \dfemph{realization functors}
\begin{align}
\rho : \sf{K}^b(\sf{C}(X_0)) &\to \sf{D}_{G, m}^b(X_0),\\
\emph{resp.}\qquad
\rho : \sf{K}^b(\underline{\sf{C}}(X_0)) &\to \underline{\sf{D}}_{G, m}^b(X_0),
\end{align}
that are triangulated in the sense that $\rho([1]_\triangle) = [1]$ and restrict to the identity functor on $\sf{C}(X_0)$, \emph{resp.}\ $\underline{\sf{C}}(X_0)$, viewed as the full subcategory of $\sf{K}^b(\sf{C}(X_0))$, \emph{resp.}\ $\sf{K}^b(\underline{\sf{C}}(X_0))$, of complexes concentrated in degree zero.
The explicit construction of these functors will show that the diagram
\begin{equation}\label{eq:kappa-rho}
\begin{tikzpicture}[baseline=(current bounding box.center), >=stealth]
\matrix(m)[matrix of math nodes, row sep=2em, column sep=3em, text height=2ex, text depth=0.5ex]
{	\sf{K}^b(\sf{C}(X_0))
		&\sf{K}^b(\underline{\sf{C}}(X_0))\\
	\sf{D}_{G, m}^b(X_0)
		&\underline{\sf{D}}_{G, m}^b(X_0)\\
		};
\path[->,font=\scriptsize, auto]
(m-1-1)		edge node{$\sf{K}^b\kappa$} (m-1-2)
(m-2-1)		edge node{$\sf{K}^b\kappa$} (m-2-2)
(m-1-1)		edge node[left]{$\rho$} (m-2-1)
(m-1-2)		edge node{$\rho$} (m-2-2);
\end{tikzpicture}
\end{equation}
commutes.

Let $\sf{D}$ and $\sf{C}$ be defined by one of the following situations:
\begin{enumerate}
\item 	$\sf{D} = \sf{D}_{G, m}^b(X_0)$ and $\sf{C} = \sf{C}(X_0)$.
\item 	$\sf{D} = \underline{\sf{D}}_{G, m}^b(X_0)$ and $\sf{C} = \underline{\sf{C}}(X_0)$.
\end{enumerate}
Let $\tilde{\sf{D}}$ be a filtered version of $\sf{D}$ in the sense of Beilinson \cite{beilinson}, and let $\omega : \tilde{\sf{D}} \to \sf{D}$ be the forgetful functor.
We will construct a full and faithful embedding $\tilde{\rho} : \sf{C}^b(\sf{C}) \to \tilde{\sf{D}}$, fitting into a commutative diagram:
\begin{equation}
\begin{tikzpicture}[baseline=(current bounding box.center), >=stealth]
\matrix(m)[matrix of math nodes, row sep=2em, column sep=2.5em, text height=2ex, text depth=0.5ex]
{	\sf{C}^b(\sf{C})
		&\sf{K}^b(\sf{C})\\
	\tilde{\sf{D}}
		&\sf{D}\\
		};
\path[->,font=\scriptsize, auto]
(m-1-1)		edge node[left]{$\tilde{\rho}$} (m-2-1)
(m-2-1)		edge node{$\omega$} (m-2-2);
\path[->,font=\scriptsize, auto, densely dotted]
(m-1-1)		edge (m-1-2)
(m-1-2)		edge node{$\rho$} (m-2-2);
\end{tikzpicture}
\end{equation}
More precisely, we will show that $\omega \circ \tilde{\rho}$ factors through the quotient functor $\sf{C}^b(\sf{C}) \to \sf{K}^b(\sf{C})$, allowing us to define $\rho$.

The filtered structure of $\tilde{\sf{D}}$ is given by full triangulated subcategories $\tilde{\sf{D}}^{\leq n}, \tilde{\sf{D}}^{\geq n}$ for each integer $n$, subject to certain axioms \cite[Def.\ 3.1]{rider}.
Let 
\begin{align}
\bm{w}_{\leq n} : \tilde{\sf{D}} &\to \tilde{\sf{D}}^{\leq n},\\
\emph{resp.}\qquad
\bm{w}_{\geq n} : \tilde{\sf{D}} &\to \tilde{\sf{D}}^{\geq n},
\end{align}
be right adjoint to the inclusion $\tilde{\sf{D}}^{\leq n} \to \tilde{\sf{D}}$, \emph{resp.}\ left adjoint to the inclusion $\tilde{\sf{D}}^{\geq n} \to \tilde{\sf{D}}$.
We set
\begin{align}
\gr_n = \bm{w}_{\leq n}\,\bm{w}_{\geq n} = \bm{w}_{\geq n}\,\bm{w}_{\leq n} : \tilde{\sf{D}} \to \tilde{\sf{D}}^{\leq n} \cap \tilde{\sf{D}}^{\geq n} \subseteq \tilde{\sf{D}}.
\end{align}
Let $s : \tilde{\sf{D}} \to \tilde{\sf{D}}$ be the shift-of-filtration functor, so that $s(\tilde{\sf{D}}^{\leq n}) = \tilde{\sf{D}}^{\leq n + 1}$ and $s(\tilde{\sf{D}}^{\geq n}) = \tilde{\sf{D}}^{\geq n + 1}$.

Let $\tilde{\sf{A}} \subseteq \tilde{\sf{D}}$ be the full, additive subcategory of objects $K$ such that $\gr_n(K) \in s^n \sf{C}[n]$ for all $n$.
We construct an additive functor $\alpha : \tilde{\sf{A}} \to \sf{C}^b(\sf{C})$ as follows.
Let
\begin{align}
\alpha(K)^n = \omega \gr_{-n}(K)[n],
\end{align}
and let $\delta^n : \alpha(K)^n \to \alpha(K)^{n + 1}$ be the morphism we obtain by applying $\omega(-)[n]$ to the boundary morphism in the exact triangle
\begin{align}\begin{split}
\gr_{-(n + 1)}(K)
\to \bm{w}_{\leq -n}\, \bm{w}_{\geq -(n + 1)}(K)
&\to \gr_{-n}(K)
\xrightarrow{\delta^n} \gr_{-(n + 1)} (K)[1].
\end{split}\end{align}
The map $K \mapsto \alpha(K)$ is functorial and additive, so it remains to show that:

\begin{lem}
For all $K \in \tilde{\sf{A}}$, the sequence $\alpha(K)$ is a well-defined object of $\sf{C}^b(\sf{C})$, \emph{i.e.}, $\delta^{n + 1} \circ \delta^n = 0$ for all $n$.
\end{lem}

\begin{proof}
The proof is analogous to that of \cite[Lem.\ 3.6]{rider}, except that we use \eqref{eq:realizable} or \eqref{eq:zero-one} in place of \cite[Lem.\ 2.4]{rider} to ensure that $\Hom(\gr_{-n - 2}(K), \gr_{-n}(K))$ vanishes for all $n$.
\end{proof}

We want to show that $\alpha$ is an equivalence.
That $\alpha$ is faithful amounts to the following statement:

\begin{lem}\label{lem:faithful}
If $\phi : K \to L$ is a morphism in $\tilde{\sf{A}}$, then $\phi = 0$ if and only if $\gr_n(\phi) = 0$ for all $n$.
\end{lem}

\begin{proof}
The proof is analogous to that of \cite[Lem.\ 3.5]{rider}, except that we use \eqref{eq:realizable} or \eqref{eq:zero-one} in place of \cite[Lem.\ 2.4]{rider} to ensure that $\Hom(\gr_n(K), \bm{w}_{\leq n - 1}(L))$ vanishes for all $K, L \in \tilde{\sf{A}}$ and $n$.
\end{proof}

\begin{lem}
The functor $\alpha : \tilde{\sf{A}} \to \sf{C}^b(\sf{C})$ is an equivalence.
Moreover, $\tilde{\rho} \vcentcolon= \alpha^{-1}$ factors through the quotient functor $\sf{C}^b(\sf{C}) \to \sf{K}^b(\sf{C})$.
\end{lem}

\begin{proof}
The proof is analogous to that of \cite[Prop.\ 3.7]{rider}, with Lemma \ref{lem:faithful} in place of \cite[Lem.\ 3.5]{rider}.
\end{proof}

\begin{rem}\label{rem:compatible}
Realization functors are compatible with other functors in the following sense.
We keep the notation above, but assume that we have two copies of each category and functor, \emph{e.g.}, $\sf{D}_1, \sf{D}_2$ in place of $\sf{D}$.
Let $\Phi : \sf{D}_1 \to \sf{D}_2$ be an exact functor that admits a canonical lift $\tilde{\Phi} : \tilde{\sf{D}}_1 \to \tilde{\sf{D}}_2$ compatible with the filtered structures on these categories in the sense of \cite{beilinson}.
In particular, this condition means $\tilde{\Phi}$ is compatible with the functors $\gr_n$.
If $\Phi$ restricts to a functor $\sf{C}_1 \to \sf{C}_2$, then there is a commutative diagram:
\begin{equation}
\begin{tikzpicture}[baseline=(current bounding box.center), >=stealth]
\matrix(m)[matrix of math nodes, row sep=2em, column sep=2.5em, text height=2ex, text depth=0.5ex]
{	\tilde{\sf{A}}_1
		&\tilde{\sf{A}}_2\\
	\sf{C}^b(\sf{C}_1)
		&\sf{C}^b(\sf{C}_2)\\
		};
\path[->,font=\scriptsize, auto]
(m-1-1)		edge node{$\tilde{\Phi}$} (m-1-2)
(m-2-1)		edge node{$\Phi$} (m-2-2)
(m-1-1)		edge node[left]{$\alpha_1$} (m-2-1)
(m-1-2)		edge node{$\alpha_2$} (m-2-2);
\end{tikzpicture}
\end{equation}
It in turn induces a commutative diagram:
\begin{equation}
\begin{tikzpicture}[baseline=(current bounding box.center), >=stealth]
\matrix(m)[matrix of math nodes, row sep=2em, column sep=3em, text height=2ex, text depth=0.5ex]
{	\sf{K}^b(\sf{C}_1)
		&\sf{K}^b(\sf{C}_2)\\
	\sf{D}_1
		&\sf{D}_2\\
		};
\path[->,font=\scriptsize, auto]
(m-1-1)		edge node{$\sf{K}^b\Phi$} (m-1-2)
(m-2-1)		edge node{$\Phi$} (m-2-2)
(m-1-1)		edge node[left]{$\rho$} (m-2-1)
(m-1-2)		edge node{$\rho$} (m-2-2);
\end{tikzpicture}
\end{equation}
So in this situation, $\Phi$ is compatible with realization.

In practice, it will be easier to check the conditions on $\Phi$ when the categories $\sf{D}_i$ take the form $\sf{D}_{G, m}^b(X_0)$ than when they take the form $\underline{\sf{D}}_{G, m}^b(X_0)$.
\end{rem}

\subsection{}

We say that a full, additive subcategory $\sf{C}(X_0) \subseteq \sf{D}_{G, m}^b(X_0)$ is \dfemph{$\SHom$-pure} iff:
\begin{enumerate}
\item
 	$\sf{C}(X_0)$ is stable under $\langle n\rangle$ for all $n \in \bb{Z}$.
\item
 	$\SHom^0(K, L)$ is concentrated in weight zero for all $K, L \in \sf{C}(X_0)$.
\end{enumerate}
These hypotheses imply that 
\begin{align}
\SHom^n(K, L) \simeq \SHom^0(K, L\langle n\rangle)(-\tfrac{n}{2})
\end{align}
is concentrated in weight $n$ for all $K, L \in \sf{C}(X_0)$ and $n \in \bb{Z}$, so $\SHom$-purity implies realizability.
The following result is inspired by \cite[Thm.\ 4.3]{rider}:

\begin{prop}\label{prop:realization}
If $\sf{C}(X_0)$ is $\SHom$-pure, then $\rho : \sf{K}^b(\underline{\sf{C}}(X_0)) \to \underline{\sf{D}}_{G, m}^b(X_0)$ induces an isomorphism of vector spaces
\begin{align}\label{eq:realization}
\Hom_{\sf{K}^b(\underline{\sf{C}}(X_0))}(\cal{K}, \cal{L}\langle n\rangle [-n]_\triangle)
\xrightarrow{\sim}
\gr_n^{\bm{w}} \SHom^0(\rho(\cal{K}), \rho(\cal{L}))
\end{align}
for all $\cal{K}, \cal{L} \in \sf{K}^b(\underline{\sf{C}}(X_0))$ and $n \in \bb{Z}$.
\end{prop}

\begin{proof}
First, suppose that $\cal{K} = K$ and $\cal{L} = L[i]_\triangle$, where $K, L \in \underline{\sf{C}}(X_0)$, viewed as objects of $\sf{K}^b(\underline{\sf{C}}(X_0))$ represented by complexes concentrated in degree zero.
In this case, the right-hand side of \eqref{eq:realization} becomes $\gr_n^{\bm{w}} \SHom^i(K, L)$.
If $i \neq n$, then
\begin{align}
\Hom_{\sf{K}^b(\underline{\sf{C}}(X_0))}(\cal{K}, \cal{L}\langle n\rangle [-n]_\triangle)
\simeq 0 \simeq \gr_n^{\bm{w}} \SHom^i(K, L),
\end{align}
where the second isomorphism uses the $\SHom$-purity of $\sf{C}(X_0)$.
If $i = n$, then 
\begin{align}\begin{split}
\Hom_{\sf{K}^b(\underline{\sf{C}}(X_0))}(\cal{K}, \cal{L}\langle n\rangle [-n]_\triangle)
&\simeq 	\Hom_{\underline{\sf{C}}(X_0)}(K, L\langle n\rangle)\\
&\simeq		\gr_0^{\bm{w}} \SHom^0(K, L\langle n\rangle)\\
&\simeq		\gr_n^{\bm{w}} \SHom^n(K, L).
\end{split}\end{align}
Now, we allow $\cal{K}$ and $\cal{L}$ to be arbitrary, and proceed by double induction on their amplitudes.
As in \cite{rider}, we only illustrate the induction on the amplitude of $\cal{K}$.

By taking (brutal) truncations of $\cal{K}$, we can build an exact triangle $\cal{K}' \to \cal{K} \to \cal{K}'' \to \cal{K}'[1]$ in which $\cal{K}', \cal{K}'' \in \sf{K}^b(\underline{\sf{C}}(X_0))$ are of strictly smaller amplitude than $\cal{K}$.
We get the following commutative diagram in which the rows are exact:
\begin{equation}
{\small
\begin{tikzpicture}[baseline=(current bounding box.center), >=stealth]
\matrix(m)[matrix of math nodes, row sep=2em, column sep=1em, text height=2ex, text depth=0.5ex]
{	{}
		&\Hom(\cal{K}', \cal{L}\langle n\rangle[-n - 1]_\triangle)
		&\Hom(\cal{K}'', \cal{L}\langle n\rangle[-n]_\triangle)
		&\Hom(\cal{K}, \cal{L}\langle n\rangle[-n]_\triangle)
		&{}\\
	{}
		&\gr_n^{\bm{w}} \SHom^{-1}(\rho(\cal{K}'), \rho(\cal{L}))
		&\gr_n^{\bm{w}} \SHom^0(\rho(\cal{K}''), \rho(\cal{L}))
		&\gr_n^{\bm{w}} \SHom^0(\rho(\cal{K}), \rho(\cal{L}))
		&{}\\
	{\cdots}
		&\Hom(\cal{K}', \cal{L}\langle n\rangle[-n]_\triangle)
		&\Hom(\cal{K}'', \cal{L}\langle n\rangle[n + 1]_\triangle)\\
	{\cdots}
		&\gr_n^{\bm{w}} \SHom^0(\rho(\cal{K}'), \rho(\cal{L}))
		&\gr_n^{\bm{w}} \SHom^1(\rho(\cal{K}''), \rho(\cal{L}))\\
		};
\path[->,font=\scriptsize, auto]
(m-1-2)		edge (m-1-3)
(m-1-3)		edge (m-1-4)
(m-1-4)		edge (m-1-5)
(m-2-2)		edge (m-2-3)
(m-2-3)		edge (m-2-4)
(m-2-4)		edge (m-2-5)
(m-3-1)		edge (m-3-2)
(m-3-2)		edge (m-3-3)
(m-4-1)		edge (m-4-2)
(m-4-2)		edge (m-4-3)
(m-1-2)		edge (m-2-2)
(m-1-3)		edge (m-2-3)
(m-1-4)		edge (m-2-4)
(m-3-2)		edge (m-4-2)
(m-3-3)		edge (m-4-3);
\end{tikzpicture}}
\end{equation}
By the inductive hypothesis, the first, second, fourth, and fifth vertical arrows are isomorphisms, so by the five lemma, the third vertical arrow is an isomorphism as well, completing the induction.
\end{proof}

Generalizing \cite[Def.\ 2.3]{rider}, we say that $\sf{C}(X_0)$ is \dfemph{Frobenius invariant} iff it is $\SHom$-pure and, strengthening hypothesis (2), $\SHom^0(K, L)$ is $F$-invariant for all $K, L$.
In this case, $\underline{\sf{C}}(X_0) \simeq \sf{C}(X_0)$.
In particular, by the commutativity of \eqref{eq:kappa-rho}, Proposition \ref{prop:realization} is actually a generalization of \cite[Thm.\ 4.3]{rider}.

\subsection{}

It is not obvious when functors between categories of the form $\sf{C}(X_0)$ can be extended to functors between categories of the form $\underline{\sf{C}}(X_0)$, because $\Hom_{\underline{\sf{C}}(X_0)}$ is defined in terms of the global sections of a sheaf.
However, we will only need this property in the case of pullback functors, where we can make use of the well-known adjunctions.
The arguments below are essentially the same as the arguments showing that the morphisms on homology/cohomology induced by pullbacks and pushforwards between $G_0$-varieties are weight-preserving, \emph{cf.}\ Section \ref{sec:steinberg}.

Namely, if $f : Y_0 \to X_0$ is an (equivariant, stratification-preserving) map of $G_0$-varieties over $\bb{F}$, then we have a composition of natural transformations of sheaf-valued bifunctors on $\sf{D}_{G, m}^b(X_0)$:
\begin{align}
\SHom(-, -) \to \SHom(-, f_\ast f^\ast(-)) \xrightarrow{\sim} f_\ast \SHom(f^\ast(-), f^\ast(-)).
\end{align}
Let $a : X_0 \to \point_0$ and $b : Y_0 \to \point_0$ be the structure maps, so that $b = a \circ f$.
Then we have $b_\ast = a_\ast f_\ast$, so applying $a_\ast$ gives:
\begin{align}\label{eq:a-to-b}
a_\ast \SHom(-, -) \to b_\ast \SHom(f^\ast(-), f^\ast(-)).
\end{align}
Applying $\gr_0^{\bm{w}}\,\ur{H}^0 \circ \xi$ gives:
\begin{align}
f^\ast : \Hom_{\underline{\sf{D}}_{G, m}^b(X_0)}(-, -) \to \Hom_{\underline{\sf{D}}_{G, m}^b(Y_0)}(f^\ast(-), f^\ast(-)).
\end{align}
We can check that these natural transformations together define a functor 
\begin{align}\label{eq:pullback}
f^\ast : \underline{\sf{D}}_{G, m}^b(X_0) \to \underline{\sf{D}}_{G, m}^b(Y_0)
\end{align}
that sends $K \mapsto f^\ast K$ at the level of objects.

\begin{lem}\label{lem:pullback}
For $f : Y_0 \to X_0$ as above, we have a commutative diagram:
\begin{equation}
\begin{tikzpicture}[baseline=(current bounding box.center), >=stealth]
\matrix(m)[matrix of math nodes, row sep=2em, column sep=3em, text height=2.5ex, text depth=0.5ex]
{	\sf{D}_{G, m}^b(X_0)
		&\sf{D}_{G, m}^b(Y_0)\\
	\underline{\sf{D}}_{G, m}^b(X_0)
		&\underline{\sf{D}}_{G, m}^b(Y_0)\\
		};
\path[->,font=\scriptsize, auto]
(m-1-1)		edge node{$f^\ast$} (m-1-2)
(m-1-1)		edge node[left]{$\kappa$} (m-2-1)
(m-1-2)		edge node{$\kappa$} (m-2-2)
(m-2-1)		edge node{$f^\ast$} (m-2-2);
\end{tikzpicture}
\end{equation}
\end{lem}

\begin{proof}
After we apply $\ur{H}^0$ to \eqref{eq:a-to-b}, the resulting natural transformation is $F$-equivariant and commutes with shifts.
So the result follows from the exactness of \eqref{eq:bbd}.
\end{proof}

\newpage
\section{Mixed Categories in Representation Theory}\label{sec:hecke}

\subsection{}

In this section, we introduce the various choices of $X_0$ and $\sf{C}(X_0)$ that we need in the rest of the paper.
They, and the functors between them, will be summarized by the following commutative diagram:
\begin{equation}\label{eq:big}
\begin{tikzpicture}[baseline=(current bounding box.center), >=stealth]
\matrix(m)[matrix of math nodes, row sep=2.5em, column sep=4em, text height=2ex, text depth=0.5ex]
{	{}
		&\sf{K}^b(\sf{C}_{\hat{W}}(G_0))
		&\sf{K}^b(\sf{C}_{\hat{W}}(\cal{U}_0))\\
	\sf{K}^b(\sf{C}(\cal{B}_0 \times \cal{B}_0))
		&\sf{K}^b(\sf{C}(G_0))
		&\sf{K}^b(\sf{C}(\cal{U}_0))\\
	\sf{D}_{G, m}^b(\cal{B}_0 \times \cal{B}_0)
		&\sf{D}_{G, m}^b(G_0)
		&\sf{D}_{G, m}^b(\cal{U}_0)\\
		};
\path[->,font=\scriptsize, auto]
(m-1-2) 	edge node{$\sf{K}^b i^\ast \langle -r\rangle$} (m-1-3)
(m-2-2)		edge node{$\sf{K}^b i^\ast \langle -r\rangle$} (m-2-3)
(m-3-2)		edge node{$i^\ast \langle -r\rangle$} (m-3-3)
(m-1-2)		edge node[right]{$\oplus$} (m-2-2)
(m-1-3)		edge node[left]{$\oplus$} (m-2-3)
(m-2-1)		edge node{$\sf{K}^b \sf{CH}$} (m-2-2)
(m-3-1)		edge node{$\sf{CH}$} (m-3-2)
(m-2-3) 	edge node[right]{$\rho$} (m-3-3)	
(m-2-1)		edge node[left]{$\rho$} (m-3-1);
\path[->,font=\scriptsize, auto, densely dotted]
(m-2-2)		edge [bend left=50] node[left]{$\sf{K}^b\sf{PR}$} (m-1-2)
(m-2-3)		edge [bend right=50] node[right]{$\sf{K}^b\sf{PR}$} (m-1-3);
\end{tikzpicture}
\end{equation}
Note that we will not discuss the ``missing'' functor $\sf{K}^b(\sf{C}(G_0)) \to \sf{D}_{G, m}^b(G_0)$.

\subsection{}

Let $W$ be the Weyl group of $G$.
Henceforth, we fix a Coxeter presentation of $W$ (see \S\ref{subsec:coxeter-group}), so that $W$ is endowed with a Bruhat order.
As in \S\ref{subsec:representation}, we write $\hat{W}$ for the set of irreducible characters of $W$.
In particular, $1 \in \hat{W}$ is the trivial character.

It will also be convenient to set 
\begin{align}
r &= \rk_\bb{Z} \bb{X},\\
N &= \tfrac{1}{2} |\Phi|.
\end{align}
We have $\dim G = r + 2N$.

\subsection{}

Let $\cal{B}_0$ be the flag variety of $G_0$, \emph{i.e.}, the variety that parametrizes its Borel subgroups.
Let $G_0$ act on $\cal{B}_0$ by right conjugation.

We stratify $\cal{B}_0 \times \cal{B}_0$ by its diagonal $G_0$-orbits.
The poset formed by the orbits under the closure relation is isomorphic to the poset formed by $W$ under the Bruhat order.
For each $w \in W$, let
\begin{align}
j_w : O_{w, 0} \to \cal{B}_0 \times \cal{B}_0
\end{align}
denote the inclusion of the corresponding orbit.
We have $\dim O_w = |w| + N$, where $|w|$ is the Bruhat length of $w$.
Thus the (shift-twisted) intersection complex
\begin{align}
\IC_w = j_{w, !\ast}\QL\langle |w| - r - N\rangle \in \sf{D}_{G, m}^b(\cal{B}_0 \times \cal{B}_0)
\end{align}
is a simple perverse sheaf of weight zero on $[(\cal{B}_0 \times \cal{B}_0)/G_0]$.

\begin{df}
Let $\sf{C}(\cal{B}_0 \times \cal{B}_0) \subseteq \sf{D}_{G, m}^b(\cal{B}_0 \times \cal{B}_0)$ be the full, replete, additive subcategory generated by the objects $\IC_w\langle n\rangle$ with $w \in W$ and $n \in \bb{Z}$.
\end{df}

\begin{lem}[Achar--Riche]\label{lem:frob-invariance-bruhat}
$\sf{C}(\cal{B}_0 \times \cal{B}_0)$ is Frobenius-invariant, hence $\SHom$-pure.
\end{lem}

\begin{proof}
\cite[Lem.\ 7.8]{ar} proves a Frobenius-invariance statement for any variety equipped with an affine even stratification.
Their argument generalizes to the stacks in \cite{sun} as long as the same condition holds on the stratification.

Fix a Borel $B_0 \subseteq G_0$.
The stratification of the coset space $B_0\backslash G_0$ into right $B_0$-orbits, \emph{i.e.}, into double cosets, is affine even.
We have an isomorphism 
\begin{align}
[B_0\backslash G_0/B_0] \simeq [(\cal{B}_0 \times \cal{B}_0)/G_0],
\end{align}
so the stacky version of \cite[Lem.\ 7.8]{ar} implies the desired result.
\end{proof}

\subsection{}

The \dfemph{convolution} on $\sf{D}_{G, m}^b(\cal{B}_0 \times \cal{B}_0)$ is the monoidal product
\begin{align}
(-) \star (-) : \sf{D}_{G, m}^b(\cal{B}_0 \times \cal{B}_0) \times \sf{D}_{G, m}^b(\cal{B}_0 \times \cal{B}_0) \to \sf{D}_{G, m}^b(\cal{B}_0 \times \cal{B}_0)
\end{align}
defined by
\begin{align}
K \star L = \pr_{1, 3, \ast}(\pr_{1, 2} \times \pr_{2, 3})^\ast (K \boxtimes L)\langle r + N\rangle,
\end{align}
where $\pr_{1, 2}, \pr_{1, 3}, \pr_{2, 3} : \cal{B}_0^3 \to \cal{B}_0^2$ are the three projection maps.

We claim that $\star$ is weight-preserving by the results of \cite{deligne_1980}.
Indeed, $\pr_{1, 3}$ is proper, and it can be checked that the map of stacks $[\cal{B}_0^3/G_0] \to [\cal{B}_0^2/G_0] \times [\cal{B}_0^2/G_0]$ induced by $\pr_{1, 2} \times \pr_{2, 3}$ is smooth of relative dimension $\dim G - \dim \cal{B} = r + N$.
More strongly, we have \cite[Prop.\ 3.2.5]{by}:

\begin{lem}[Bezrukavnikov--Yun]
The convolution $\star$ restricts to a monoidal product on $\sf{C}(\cal{B}_0 \times \cal{B}_0)$.
\end{lem}

\begin{rem}\label{rem:by}
Let $\sf{C}(\cal{B} \times \cal{B})$ be the the essential image of $\sf{C}(\cal{B}_0 \times \cal{B}_0)$ in $\sf{D}_G^b(\cal{B} \times \cal{B})$.
Prior to \cite{by}, it was known that the convolution on $\sf{D}_G^b(\cal{B} \times \cal{B})$ preserves the class of objects in $\sf{C}(\cal{B} \times \cal{B})$, \emph{cf.}\ \cite[Prop.\ 2.13]{springer_1982}. 
As explained in \cite[Rem.\ 3.2.6]{by}, the force of the lemma is to assert that for all $K, L \in \sf{C}(\cal{B}_0 \times \cal{B}_0)$, the object $K \star L$ is semisimple \emph{before} its pullback to $\cal{B} \times \cal{B}$.
The proof involves developing a Frobenius-enriched version of some of the results of \cite{springer_1982}.
\end{rem}

The convolution on $\sf{C}(\cal{B}_0 \times \cal{B}_0)$ induces one on $\sf{K}^b(\sf{C}(\cal{B}_0 \times \cal{B}_0))$.
In the case where \emph{$G$ is semisimple}, we define the \dfemph{Hecke category} of $W$ to be the resulting monoidal triangulated category, and abbreviate
\begin{align}\begin{split}
\sf{H}_W
&= \sf{K}^b(\sf{C}(\cal{B}_0 \times \cal{B}_0)).
\end{split}\end{align}
By Remark \ref{rem:compatible}, the convolution on $\sf{H}_W$ is compatible with realization functors in the sense that the diagram
\begin{equation}\label{eq:convolution-realization}
\begin{tikzpicture}[baseline=(current bounding box.center), >=stealth]
\matrix(m)[matrix of math nodes, row sep=2em, column sep=2em, text height=2ex, text depth=0.5ex]
{	\sf{H}_W \times \sf{H}_W
		&\sf{H}_W\\
	\sf{D}_{G, m}^b(\cal{B}_0 \times \cal{B}_0) \times \sf{D}_{G, m}^b(\cal{B}_0 \times \cal{B}_0)
		&\sf{D}_{G, m}^b(\cal{B}_0 \times \cal{B}_0)\\
		};
\path[->,font=\scriptsize, auto]
(m-1-1)		edge node{$\star$} (m-1-2)
(m-2-1)		edge node{$\star$} (m-2-2)
(m-1-1)		edge node[left]{$\rho \times \rho$} (m-2-1)
(m-1-2)		edge node{$\rho$} (m-2-2);
\end{tikzpicture}
\end{equation}
commutes.

\begin{rem}
This definition of the Hecke category originates in work of Beilinson, Ginzburg, Soergel, and others on Koszul duality patterns in the geometry of mixed complexes.
See \cite[\S{4}]{bgs} and the references therein.
\end{rem}

\subsection{}\label{subsec:decat}

Let $H_W$ be the Iwahori--Hecke algebra of $W$ (see \S\ref{subsec:hecke}).
Below, we review how $\sf{H}_W$ decategorifies to $H_W$.

If $\sf{A}$ is an additive category, then we write $[\sf{A}]_\oplus$ for its \emph{split} Grothendieck group.
The elements of $[\sf{A}]_\oplus$ are the isomorphism classes $[K]$ of the objects $K$, and the relations are $[K \oplus L] = [K] + [L]$ for all objects $K, L$.
If $\sf{A}$ is given a monoidal product $\star$, then $[\sf{A}]_\oplus$ forms a ring under the multiplication $[K] \cdot [L] = [K \star L]$.

We endow $[\sf{C}(\cal{B}_0 \times \cal{B}_0)]_\oplus$ with the action of $\bb{Z}[\Q^{\pm\frac{1}{2}}]$ in which $\Q^{\frac{1}{2}}$ acts by $\langle -1\rangle$.
For each $K \in \sf{C}(\cal{B}_0 \times \cal{B}_0)$, let
\begin{align}
\chi_w(K) = \sum_{i, j}
{(-1)^{i + j}} \Q^{\frac{j}{2}} \dim \gr_j^{\bm{w}} \ur{H}^i(j_w^\ast K),
\end{align}
the virtual weight polynomial of the stalk of $K$ over $[O_{w, 0}/G]$.
Thm.\ 2.8 of \cite{springer_1982} implies if $G$ is \emph{semisimple}, then the map
\begin{align}
\begin{array}{rcl}
{[\sf{C}(\cal{B}_0 \times \cal{B}_0)]_\oplus} &\to &H_W\\
{[K]} &\mapsto &\displaystyle\sum_{w \in W} \Q^{\frac{|w|}{2}} \sigma_w \chi_w(K)
\end{array}
\end{align}
is a morphism of $\bb{Z}[\Q^{\pm \frac{1}{2}}]$-algebras.
In fact, it is an isomorphism, and if $\gamma_w$ is the image of $[\IC_w]$, then the set $\{\gamma_w\}_{w \in W}$ forms a free $\bb{Z}[\Q^{\pm\frac{1}{2}}]$-linear basis of $H_W$, \emph{cf.}\ \cite[407]{ww_2011}.
The element $\gamma_w$ is called the \dfemph{Kazhdan--Lusztig element} attached to $w$.

If $\sf{T}$ is a triangulated (additive) category, then we write $[\sf{T}]_\triangle$ for its \emph{triangulated} Grothendieck group.
The elements of $[\sf{T}]_\triangle$ are the isomorphism classes of the objects of $\sf{T}$, and the relations are $[\cal{K}] = [\cal{K}'] + [\cal{K}'']$ for each exact triangle $\cal{K}' \to \cal{K} \to \cal{K}'' \to \cal{K}'[1]$ in $\sf{T}$.

If $\sf{A}$ is any additive category, then there is an isomorphism $[\sf{K}^b(\sf{A})]_\triangle \xrightarrow{\sim} [\sf{A}]_\oplus$ that sends any complex to the alternating sum of its terms \cite{rose}.
In particular,
\begin{align}\begin{split}
[\sf{H}_W]_\triangle 
&\simeq [\sf{C}(\cal{B}_0 \times \cal{B}_0)]_\oplus.
\end{split}\end{align}
Since the convolutions on these categories are compatible with each other, this is an isomorphism of rings, not just groups.
We endow $[\sf{H}_W]_\triangle$ with the action of $\bb{Z}[\Q^{\pm\frac{1}{2}}]$ in which $\Q^{\frac{1}{2}}$ acts by $\langle -1\rangle$.
Then we obtain a $\bb{Z}[\Q^{\pm\frac{1}{2}}]$-algebra isomorphism $[\sf{H}_W]_\triangle \xrightarrow{\sim} H_W$ that again sends $[\IC_w] \mapsto \gamma_w$.

\begin{rem}\label{rem:iwahori}
Recall that $q^{\frac{1}{2}} \in \QL^\times$ is a fixed square root of $q = |\bb{F}|$.
In \cite{iwahori}, Iwahori gave an isomorphism of $\QL$-algebras:
\begin{align}\label{eq:iwahori}
\QL \otimes H_W|_{\Q^{1/2} = q^{1/2}} \simeq \End_{G^F}(\QL[\cal{B}^F]).
\end{align}
The isomorphism $H_W \simeq [\sf{C}(\cal{B}_0 \times \cal{B}_0)]_\oplus$ is an integral version in the following sense.
We have a commutative diagram
\begin{equation}
\begin{tikzpicture}[baseline=(current bounding box.center), >=stealth]
\matrix(m)[matrix of math nodes, row sep=2em, column sep=2.5em, text height=2ex, text depth=0.5ex]
{	{[\sf{C}(\cal{B}_0 \times \cal{B}_0)]_\oplus}
		&H_W\\
	\bb{Z}[q^{\pm\frac{1}{2}}][\cal{B}^F \times \cal{B}^F]^{G^F}
		&H_W|_{\Q^{1/2} = q^{1/2}}\\
		};
\path[->,font=\scriptsize, auto]
(m-1-1)		edge node{$\sim$}(m-1-2)
(m-1-1)		edge (m-2-1)
(m-1-2)		edge (m-2-2)
(m-2-1)		edge node{$\sim$} (m-2-2);
\end{tikzpicture}
\end{equation}
where the left-hand arrow is induced by the function-sheaf dictionary.
Iwahori's isomorphism can be factored as
\begin{align}
\End_{G^F}(\QL[\cal{B}^F]) \xrightarrow{\sim} \QL[\cal{B}^F \times \cal{B}^F]^{G^F} \xrightarrow{\sim} \QL \otimes H_W|_{\Q^{1/2} = q^{1/2}},
\end{align}
where the second arrow is induced by the bottom arrow of the diagram.
\end{rem}

\subsection{}

Let $G_0$ act on itself by right conjugation.
There is a $G_0$-stable stratification of $G_0$ by locally-closed subvarieties \cite[\S{3.1}]{lusztig_1984_intersection} \cite[\S{3.11}]{lusztig_1985_1}, now known as Jordan classes.
Let 
\begin{align}
\act : G_0 \times \cal{B}_0 &\to \cal{B}_0 \times \cal{B}_0,\\
\pr : G_0 \times \cal{B}_0 &\to G_0
\end{align}
be defined by 
\begin{align}
\act(g, B) &= (g^{-1}Bg, B),\\
\pr(g, B) &= g.
\end{align}
The maps above are $G_0$-equivariant with respect to the diagonal $G_0$-actions on $G_0 \times \cal{B}_0$ and $\cal{B}_0 \times \cal{B}_0$.
Moreover, we can stratify $G_0 \times \cal{B}_0$ so that $\act$, \emph{resp.}\ $\pr$, is stratification-preserving with respect to the Bruhat stratification of $\cal{B}_0 \times \cal{B}_0$, \emph{resp.}\ the Jordan stratification of $G_0$.

Following \cite{lusztig_1985_1}, we define the \dfemph{character functor}
\begin{align}
\sf{CH} : \sf{D}_{G, m}^b(\cal{B}_0 \times \cal{B}_0) \to \sf{D}_{G, m}^b(G_0)
\end{align}
to be the composition
\begin{align}\label{eq:horocycle-transform}\begin{split}
\sf{CH} = \bigoplus_i {}^p\cal{H}^i[-i] \circ \pr_\ast\, \act^!\langle r + N\rangle.
\end{split}\end{align}
Since $\act$ is smooth of relative dimension $r + N$ and $\pr$ is proper, $\sf{CH}$ preserves the class of pure complexes of weight zero \cite{deligne_1980}.
We define a \dfemph{mixed unipotent character sheaf} to be a simple perverse sheaf $E \in \sf{D}_{G, m}^b(G_0)$ such that $E\langle -i\rangle$ is a subquotient of $\sf{CH}(\IC_w)$, \emph{resp.}\ $\xi\sf{CH}(\IC_w)$, for some $w \in W$ and $i \in \bb{Z}$.
This means $E$ is also pure of weight zero.

Recall that by \cite[Prop.\ 5.3.9]{bbd}, every indecomposable pure perverse sheaf takes the form $E \otimes M$, where $E$ is a simple perverse sheaf and $M$ is the pullback of a sheaf on $\point_0$ that corresponds to a $\QL[F]$-module on which $F$ acts by a single unipotent Jordan block.
Thus, a simple perverse sheaf $E \in \sf{D}_{G, m}^b(G_0)$ is a mixed unipotent character sheaf if and only if $E\langle -i\rangle \otimes M$ occurs as a direct summand of $\sf{CH}(\IC_w)$ for some $w, i, M$ as above.

\begin{df}
Let $\sf{C}(G_0) \subseteq \sf{D}_{G, m}^b(G_0)$ be the full, replete, additive subcategory generated by the objects $E \otimes M$, where $E$ is a mixed unipotent character sheaf and $M$ is the pullback of a pure object of $\sf{D}_{G, m}^b(\point_0)$ of weight zero.
\end{df}

\begin{rem}
The objects $\xi \sf{CH}(\IC_w) \in \sf{D}_G^b(G)$ coincide with the objects $\bar{K}_w^\cal{L}$ studied by Lusztig in \cite{lusztig_1985_3} when $\cal{L}$ is the trivial local system on a fixed maximal torus of $G$.
They also coincide with the objects $\bb{K}_w$ studied by Webster--Williamson in \cite{ww_2011} up to a shift-twist.
\end{rem}

\subsection{}

The character functor is a monoidal trace on $\sf{D}_{G, m}^b(\cal{B}_0 \times \cal{B}_0)$ in the following sense.
Abbreviate
\begin{align}
\sf{CH}(- \star -) \vcentcolon= \sf{CH} \circ ((-) \star (-)) : \sf{D}_{G, m}^b(\cal{B}_0 \times \cal{B}_0)^2 \to \sf{D}_{G, m}^b(G_0).
\end{align}
The result below can be deduced from a paper of Lusztig \cite{lusztig_2015}, identifying the monoidal center of $\sf{D}_G^b(\cal{B} \times \cal{B})$ with a certain category of unmixed unipotent character sheaves, but we give a self-contained proof.

\begin{lem}\label{lem:monoidal-trace}
We have ${\sf{CH}(- \star -) \circ (\pr_2 \times \pr_1)^\ast}\, \simeq \,{\sf{CH}(- \star -)}$, where $\pr_2 \times \pr_1$ is the involution of $[(\cal{B}_0 \times \cal{B}_0)/G_0]^2$ that swaps the factors.
\end{lem}

\begin{proof}
Observe that $\sf{CH}(- \star -)$ is induced by pullback and pushforward along the solid arrows in the diagram below:
\begin{equation}
\begin{tikzpicture}[baseline=(current bounding box.center), >=stealth]
\matrix(m)[matrix of math nodes, row sep=2.5em, column sep=2.5em, text height=2ex, text depth=0.5ex]
{	{[(G_0 \times \cal{B}_0 \times \cal{B}_0)/G_0]}
		&{[(G_0 \times \cal{B}_0)/G_0]}
		&{[G_0/G_0]}\\
	{[(\cal{B}_0 \times \cal{B}_0 \times \cal{B}_0)/G_0]}
		&{[(\cal{B}_0 \times \cal{B}_0)/G_0]}
		&\\
	{[(\cal{B}_0 \times \cal{B}_0)/G_0]^2}\\
		};
\path[->,font=\scriptsize, auto]
(m-1-2)		edge node[above]{$\pr$}( m-1-3)
			edge node[right]{$\act$} (m-2-2)
(m-2-1)		edge node[below]{$\pr_{1,3}$} (m-2-2)
			edge node[left]{$\pr_{1,2} \times \pr_{2,3}$} (m-3-1);
\path[->,font=\scriptsize, densely dotted]
(m-1-1)		edge (m-1-2)
			edge (m-2-1);
\end{tikzpicture}
\end{equation}
Let the square in the diagram be cartesian.
By proper base change across that square, we see that
\begin{align}
\sf{CH}(- \star -) = \bigoplus_i {}^p\cal{H}^i[-i] \circ \Psi_\ast \circ \Phi^\ast \langle 2(r + N)\rangle,
\end{align}
where the functors
\begin{align}
\Phi : [(G_0 \times \cal{B}_0 \times \cal{B}_0)/G_0] &\to [(\cal{B}_0 \times \cal{B}_0)/G_0]^2,\\
\Psi : [(G_0 \times \cal{B}_0 \times \cal{B}_0)/G_0] &\to [G_0/G_0]
\end{align}
are defined by
\begin{align}
\Phi([g, B, B']) &= ([g^{-1}B'g, B], [B, B']),\\
\Psi([g, B, B']) &= [g].
\end{align}
Now consider the diagram
\begin{equation}
\begin{tikzpicture}[baseline=(current bounding box.center), >=stealth]
\matrix(m)[matrix of math nodes, row sep=1em, column sep=2.5em, text height=2ex, text depth=0.5ex]
{	{[(\cal{B}_0 \times \cal{B}_0)/G_0]^2}
		&{[(G_0 \times \cal{B}_0 \times \cal{B}_0)/G_0]}
		&\\
	{}
		&
		&{[G_0/G_0]}\\
	{[(\cal{B}_0 \times \cal{B}_0)/G_0]^2}
		&{[(G_0 \times \cal{B}_0 \times \cal{B}_0)/G_0]}
		&\\
		};
\path[->,font=\scriptsize, auto]
(m-1-2) edge node[above]{$\Phi$} (m-1-1)
(m-1-2.south east)	 edge node[above]{$\Psi$} (m-2-3)
(m-3-2) edge node[above]{$\Phi$} (m-3-1)
(m-3-2.north east)	 edge node[below]{$\Psi$} (m-2-3)
(m-1-1) edge node[left]{$\pr_2 \times \pr_1$} (m-3-1)
(m-1-2) edge node{$\Xi$} (m-3-2);
\end{tikzpicture}
\end{equation}
in which $\Xi([g, B, B']) = [g, g^{-1}B'g, B]$.
We can check that the square on the left and the triangle on the right commute.
(The former is where it is essential to use orbits, not orbit representatives.)
Thus the whole diagram commutes, giving us the desired equivalence of functors.
\end{proof}

\begin{rem}
The cocenter of an associative algebra $A$ is the largest commutative quotient, \emph{i.e.}, the algebra $A/[A, A]$.
We expect that the functor $\sf{CH}$ categorifies the map from the Iwahori--Hecke algebra to its cocenter, by way of the function-sheaf dictionary (see Remark \ref{rem:iwahori}).
\end{rem}

\subsection{}

Let $\tilde{G}_0$ be defined by the cartesian diagram:
\begin{equation}
\begin{tikzpicture}[baseline=(current bounding box.center), >=stealth]
\matrix(m)[matrix of math nodes, row sep=2em, column sep=3em, text height=2.5ex, text depth=0.5ex]
{	\tilde{G}_0
		&G_0 \times \cal{B}_0\\
	O_{1, 0} = \cal{B}_0
		&\cal{B}_0 \times \cal{B}_0\\
		};
\path[->,font=\scriptsize, auto]
(m-1-1)		edge (m-1-2)
(m-1-1)		edge (m-2-1)
(m-1-2)		edge node{$\act$} (m-2-2)
(m-2-1)		edge node{$j_1$} (m-2-2);
\end{tikzpicture}
\end{equation}
The composition $p' : \tilde{G}_0 \to G_0 \times \cal{B}_0 \to G_0$ is a small, hence proper, map called the \dfemph{Grothendieck--Springer map}.
The \dfemph{Grothendieck--Springer sheaf} is 
\begin{align}
\cal{G} = p'_\ast\QL.
\end{align}
By base change and smallness, $\cal{G} \simeq \sf{CH}(\IC_1)$.

The Grothendieck--Springer map restricts to a Galois $W$-cover over the regular semisimple locus of $G$.
Therefore, the decomposition of $\cal{G}$ into indecomposable perverse sheaves takes the form
\begin{align}
\cal{G} \simeq \bigoplus_{\phi \in \hat{W}} \cal{G}_\phi \otimes V_\phi,
\end{align}
where each $\cal{G}_\phi$ is a simple perverse sheaf and $V_\phi$ is the irreducible representation of $W$ over $\QL$ of character $\phi$.
We pin down $\cal{G}_\phi$ up to unique isomorphism by declaring that Frobenius acts trivially on $V_\phi$.

\begin{df}
Let $\sf{C}_{\hat{W}}(G_0) \subseteq \sf{C}(G_0)$ be the full, replete, additive subcategory generated by the objects $\cal{G}_\phi \otimes M$, where $\phi \in \hat{W}$ and $M$ is the pullback of a pure object of $\sf{D}_G^b(\point_0)$ of weight zero.
\end{df}

\begin{rem}
The work in Section \ref{sec:steinberg} will show that $\sf{C}_{\hat{W}}(G_0)$ and thus $\sf{C}(G_0)$ are not realizable, let alone $\SHom$-pure.
\end{rem}

\begin{rem}
If $G$ is of type $A$, meaning $W = S_n$ for some $n$, then every (unmixed) unipotent character sheaf on $G$ is isomorphic to a summand of $\xi\cal{G}$ (\emph{e.g.}, by Lusztig's classification of character sheaves).
So in this case, we have $\sf{C}_{\hat{W}}(G_0) \simeq \sf{C}(G_0)$.
In other types, a unipotent character sheaf that does not occur as a summand of $\xi \cal{G}$ is said to be \dfemph{cuspidal}.
\end{rem}

\subsection{}

Let $i : \cal{U}_0 \to G_0$ be the inclusion of the unipotent locus.
The action of $G_0$ on itself by conjugation restricts to an action on $\cal{U}_0$.
We stratify $\cal{U}_0$ by its $G_0$-orbits.
Every orbit is a Jordan class of $G_0$, so the inclusion $i : \cal{U}_0 \to G_0$ is stratification-preserving.
We have $\dim \cal{U} = 2N$, from which $\dim \cal{U} - \dim G = -r$.

\begin{df}\label{df:unipotent}
Let $\sf{C}(\cal{U}_0) \subseteq \sf{D}_{G, m}^b(\cal{U}_0)$ be the full, replete, additive subcategory generated by the objects $i^\ast E \otimes M$, where $E$ is a mixed unipotent character sheaf and $M$ is the pullback of a pure object of $\sf{D}_{G, m}^b(\point_0)$ of weight zero.
\end{df}

Since the map $i$ is merely a closed immersion, there is a priori no reason that $i^\ast$ sends unipotent character sheaves to well-behaved objects.
Nevertheless:

\begin{thm}[Lusztig]\label{thm:lusztig-restriction}
If $E$ is a mixed unipotent character sheaf, then there exist a $G_0$-orbit $j_C : C_0 \to \cal{U}_0$ and a pure $\ell$-adic local system $\cal{L} \in \sf{D}_{G, m}^b(C_0)$ of weight zero such that $i^\ast E \simeq j_{C, !\ast} \cal{L}\langle n_C\rangle$ for some $n_C \in \bb{Z}$ that only depends on $C$.
\end{thm}

\begin{proof}
The unmixed version of this statement is \cite[Thm.\ 6.5(c)]{lusztig_1984_intersection}.
Below, we explain how to deduce our mixed version.
The notation $\IC(-, -)$ will mean the intersection complex of a local system (\emph{not} necessarily shifted to be perverse).

First, $\xi E$ is an object of $\sf{D}_G^b(G)$ that, in Lusztig's notation, takes the form $\phi_!K[\dim \bar{Y}]$, where $\phi$ is a map $X \to \bar{Y} \subseteq G$ and $K = \IC(X, \bar{\cal{E}}_1)$ for some local system $\bar{\cal{E}}_1$ over $X$ \cite[\S{4}]{lusztig_1984_intersection}.
As long as we start with $\bb{F}$-structures on $X, Y, \phi, \bar{\cal{E}}_1$ such that $\bar{\cal{E}}_1$ is pure of weight zero, the construction descends to $G_0$ to yield $E = \phi_!K\langle \dim \bar{Y}\rangle$.
In the notation of \cite[\S{6.6}]{lusztig_1984_intersection}, we have $i^\ast K \simeq K^a$, where $K^a = \IC(X_0^a, \bar{\cal{E}}_1^a)$ and $\bar{\cal{E}}_1^a = i^\ast \bar{\cal{E}}_1$ (and the subscript $0$ means the $\bb{F}$-structure on $X^a$ coming from the $\bb{F}$-structure we chose on $X$).
Therefore, $i^\ast E \simeq \phi_!^a K^a \langle \dim \bar{Y}\rangle$.
Since $\phi^a$ is proper, $\bar{\cal{E}}_1^a$ being pure implies $\phi_!^a K^a$ being pure of the same weight \cite{deligne_1980}, \emph{cf.}\ (6.6.1) of \cite{lusztig_1984_intersection}.

In our situation, $E$ is a unipotent character sheaf, which means that $\bar{\cal{E}}_1$ is a constant sheaf \cite[\S{3.1}, \S{4.4}]{lusztig_1985_1}.
Therefore, $\bar{\cal{E}}_1^a$ is the constant sheaf on $X_0^a$, hence pure of weight zero.
Therefore $i^\ast E$ is pure of weight zero.
Now, the rest of \cite[Thm.\ 6.5(c)]{lusztig_1984_intersection} implies that at the unmixed level, $i^\ast E$ is a shift of some intersection complex $j_{C, !\ast}\cal{E}$.
Since $\cal{E}$ is the shifted restriction of $i^\ast E$ to $C$, it must be pure.
So up to twisting $\cal{E}$, we can find a pure local system $\cal{L}$ of weight zero such that $i^\ast E \simeq j_{C, !\ast} \cal{L}\langle n\rangle$ for some $n$.
Finally, \cite[Thm.\ 6.5(c)]{lusztig_1984_intersection} also implies that $n$ only depends on $\dim \bar{Y}^a$ and $\dim C$, and the former is determined by a variety $C_1$ that only depends on $C$.
\end{proof}

\subsection{}

Let $p : \tilde{\cal{U}}_0 \to \cal{U}_0$ be defined by the cartesian diagram:
\begin{equation}
\begin{tikzpicture}[baseline=(current bounding box.center), >=stealth]
\matrix(m)[matrix of math nodes, row sep=2em, column sep=3em, text height=2.5ex, text depth=0.5ex]
{	\tilde{\cal{U}}_0
		&\tilde{G}_0\\
	\cal{U}_0
		&G_0\\
		};
\path[->,font=\scriptsize, auto]
(m-1-1)		edge (m-1-2)
(m-1-1)		edge node[left]{$p$} (m-2-1)
(m-1-2)		edge node{$p'$} (m-2-2)
(m-2-1)		edge node{$i$} (m-2-2);
\end{tikzpicture}
\end{equation}
Then $p$ (or rather, $p^\red : \tilde{\cal{U}}_0^\red \to \cal{U}_0$) is a semismall resolution of singularities called the \dfemph{Springer resolution}.
The \dfemph{Springer sheaf} is
\begin{align}
\cal{S} = p_\ast \QL\langle -r\rangle.
\end{align}
By proper base change, $\cal{S} = i^\ast \cal{G}\langle -r\rangle$.
Setting $\cal{S}_\phi = i^\ast \cal{G}_\phi \langle -r\rangle$, we have that
\begin{align}
\cal{S} \simeq \bigoplus_{\phi \in \hat{W}} \cal{S}_\phi \otimes V_\phi
\end{align}
is the decomposition of $\cal{S}$ into indecomposable perverse sheaves.

\begin{df}
Let $\sf{C}_{\hat{W}}(\cal{U}_0) \subseteq \sf{C}(\cal{U}_0)$ be the full, replete, additive subcategory generated by the objects $\cal{S}_\phi \otimes M$, where $\phi \in \hat{W}$ and $M$ is the pullback of a pure object of $\sf{D}_{G, m}^b(\point_0)$ of weight zero.
\end{df}

\begin{lem}\label{lem:unipotent}
$\sf{C}(\cal{U}_0)$ and thus $\sf{C}_{\hat{W}}(\cal{U}_0)$ are $\SHom$-pure.
\end{lem}

\begin{proof}
$\sf{C}(\cal{U}_0)$ is closed under $\langle n\rangle$ for all $n$ because the same is true of the class of pure objects of weight zero in $\sf{D}_{G, m}^b(\point_0)$.
It remains to check that $\SHom^0(i^\ast E, i^\ast E')$ is concentrated in weight zero for all mixed unipotent character sheaves $E, E'$.

In the notation of Theorem \ref{thm:lusztig-restriction}, we can write $i^\ast E \simeq j_{C, !\ast} \cal{L}\langle n_C\rangle$ and $i^\ast E' \simeq j_{C', !\ast} \cal{L}'\langle n_{C'}\rangle$ for some $C, C', \cal{L}, \cal{L}', n_C, n_{C'}$.
It remains to show that 
\begin{align}\begin{split}
&\SHom^0(j_{C, !\ast} \cal{L}\langle n_C\rangle, j_{C', !\ast} \cal{L}'\langle n_{C'}\rangle) \\
&= \SHom^{d}(j_{C, !\ast} \cal{L}\langle \dim C\rangle, j_{C', !\ast} \cal{L}'\langle \dim C'\rangle)(\tfrac{d}{2})
\end{split}\end{align}
is concentrated in weight zero, where $d = (n_{C'} - \dim C') - (n_C - \dim C)$.
Note that $j_{C, !\ast} \cal{L}\langle \dim C\rangle$ and $j_{C', !\ast} \cal{L}'\langle \dim C'\rangle$ are both pure perverse of weight zero.
So the desired claim follows from \cite[Cor.\ 4.8]{rr}.
\end{proof}

\subsection{}

To conclude this section, we discuss the relationships between the categories $\sf{C}(G_0)$, $\sf{C}_{\hat{W}}(G_0)$, $\sf{C}(\cal{U}_0)$, and $\sf{C}_{\hat{W}}(\cal{U}_0)$.

\begin{thm}[Lusztig, Rider--Russell]\label{thm:summand}
The horizontal arrows in the commutative diagrams below are inclusions of direct summands:
\begin{equation}
\begin{tikzpicture}[baseline=(current bounding box.center), >=stealth]
\matrix(m)[matrix of math nodes, row sep=2em, column sep=3em, text height=2.5ex, text depth=0.5ex]
{	\sf{C}_{\hat{W}}(G_0)
		&\sf{C}(G_0)\\
	\underline{\sf{C}}_{\hat{W}}(G_0)
		&\underline{\sf{C}}(G_0)\\
		};
\path[->,font=\scriptsize, auto]
(m-1-1)		edge (m-1-2)
(m-1-1)		edge node[left]{$\kappa$} (m-2-1)
(m-1-2)		edge node{$\kappa$} (m-2-2)
(m-2-1)		edge (m-2-2);
\end{tikzpicture}
\qquad
\begin{tikzpicture}[baseline=(current bounding box.center), >=stealth]
\matrix(m)[matrix of math nodes, row sep=2em, column sep=3em, text height=2.5ex, text depth=0.5ex]
{	\sf{C}_{\hat{W}}(\cal{U}_0)
		&\sf{C}(\cal{U}_0)\\
	\underline{\sf{C}}_{\hat{W}}(\cal{U}_0)
		&\underline{\sf{C}}(\cal{U}_0)\\
		};
\path[->,font=\scriptsize, auto]
(m-1-1)		edge (m-1-2)
(m-1-1)		edge node[left]{$\kappa$} (m-2-1)
(m-1-2)		edge node{$\kappa$} (m-2-2)
(m-2-1)		edge (m-2-2);
\end{tikzpicture}
\end{equation}
\end{thm}

\begin{proof}
In the diagram for $G_0$, the statement for the bottom arrow is essentially a theorem of Lusztig \cite[Prop.\ 7.2]{lusztig_1985_2}.
The only modification needed is $G$-equivariance, which is done in \cite[Prop.\ 8]{ww_2008}.
In the diagram for $\cal{U}_0$, the statement for the bottom arrow follows from \cite[Thm.\ 3.5]{rr}.
Finally, the statements for the top arrows follow from the exactness of \eqref{eq:bbd}.
\end{proof}

We define the functors
\begin{align}
\sf{PR} : \sf{C}(G_0) &\to \sf{C}_{\hat{W}}(G_0),\\
\sf{PR} : \sf{C}(\cal{U}_0) &\to \sf{C}_{\hat{W}}(\cal{U}_0),
\end{align}
as well as their analogues with $\underline{\sf{C}}$ in place of $\sf{C}$, to be the projections that correspond to these summand inclusions.
By Lemma \ref{lem:pullback}, we have:

\begin{lem}\label{lem:pullback-kappa}
We have a commutative diagram:
\begin{equation}
\begin{tikzpicture}[baseline=(current bounding box.center), >=stealth]
\matrix(m)[matrix of math nodes, row sep=2em, column sep=3em, text height=2.5ex, text depth=0.5ex]
{	\sf{C}(G_0)
		&\sf{C}(\cal{U}_0)\\
	\underline{\sf{C}}(G_0)
		&\underline{\sf{C}}(\cal{U}_0)\\
		};
\path[->,font=\scriptsize, auto]
(m-1-1)		edge node{$i^\ast\langle -r\rangle$} (m-1-2)
(m-1-1)		edge node[left]{$\kappa$} (m-2-1)
(m-1-2)		edge node{$\kappa$} (m-2-2)
(m-2-1)		edge node{$i^\ast\langle -r\rangle$} (m-2-2);
\end{tikzpicture}
\end{equation}
Here, the bottom arrow is defined by \eqref{eq:pullback}.
\end{lem}

\newpage
\section{The $\mathbf{A}_W$ Trace}\label{sec:a_w}

\subsection{}

If $A$ is a $\bb{Z}$-graded $\QL$-algebra, then we write:
\begin{enumerate}
\item 	$\Mod_1(A)$ for the category of graded left $A$-modules in which the grading is bounded below and each graded piece is finite-dimensional.
\item 	$\Mod_2^b(A)$ for the category of left $A$-modules $M$ equipped with a $\bb{Z}^2$-grading such that $M \in \sf{Mod}_1(A)$ with respect to both gradings and the value of $i - j$ is uniformly bounded among the nonzero components $M^{i, j}$.
\end{enumerate}
If $R$ is an (ungraded) $\QL$-algebra and $A$ forms an $R$-module, then we write $R \ltimes A$ for the vector space $R \otimes A$ endowed with the ring structure where:
\begin{itemize}
\item 	$R$ and $A$ form subrings.
\item 	The only relations between $R$ and $A$ are those of the form $ra = (r \cdot a)r$ for all $r \in R$ and $a \in A$.
\end{itemize}
If the $R$-action on $A$ preserves the grading, then we extend the grading on $A$ to a grading on $R \ltimes A$ in which $R$ is placed in degree zero, so that $R \ltimes A$ forms a graded $A$-algebra.

Recall that $\bb{X}$ is the character lattice in the root datum of $G$.
We set
\begin{align}
\bb{V} = \bb{X} \otimes \QL.
\end{align}
Let the symmetric algebra $\Sym(\bb{V})$ be endowed with the grading where $\Sym^i(\bb{V})$ is placed in degree $2i$, and let
\begin{align}
\bb{A}_W = \QL[W] \ltimes \Sym(\bb{V}).
\end{align}
Below, we construct a contravariant, monoidal trace from the Hecke category to $\Mod_2^b(\bb{A}_W)$.

\subsection{}

Suppose that $X_0$ is a $G_0$-variety over $\bb{F}$ and that $\sf{C}(X_0)$ is a full, additive subcategory of $\sf{D}_{G, m}^b(X_0)$ stable under $\langle n\rangle$ for all $n \in \bb{Z}$.
Any object $K \in \underline{\sf{C}}(X_0)$ gives rise to a graded $\QL$-algebra:
\begin{align}
\bb{A}_K = \bigoplus_n
{\Hom_{\underline{\sf{C}}(X_0)}(K, K \langle n\rangle)}.
\end{align}
We write the multiplication in $\bb{A}_K$ as right-to-left composition: $g \cdot f = g \circ f$.
Then we have additive functors
\begin{align}
\bigoplus_n
{\Hom_{\underline{\sf{C}}(X_0)}(-, K\langle n\rangle)}
	: \underline{\sf{C}}(X_0)^\op &\to \Mod_1(\bb{A}_K),\\	
\bigoplus_n
{\Hom_{\underline{\sf{C}}(X_0)}(K, (-)\langle n\rangle)}
	: \underline{\sf{C}}(X_0) &\to \Mod_1(\bb{A}_K^\op).
\end{align}
Above, the gradings are bounded below by Remark \ref{rem:bounded-below}.
These functors can be respectively extended to
\begin{align}
\bigoplus_{i, j} 
{\Hom_{\sf{K}^b(\underline{\sf{C}}(X_0))}(-, K\langle j\rangle [i - j]_\triangle)}
	: \sf{K}^b(\underline{\sf{C}}(X_0))^\op &\to \Mod_2^b(\bb{A}_K),\\	
\bigoplus_{i, j} 
{\Hom_{\sf{K}^b(\underline{\sf{C}}(X_0))}(K, (-)\langle j\rangle [i - j]_\triangle)}
	: \sf{K}^b(\underline{\sf{C}}(X_0)) &\to \Mod_2^b(\bb{A}_K^\op),
\end{align}
once we embed $\underline{\sf{C}}(X_0)$ into $\sf{K}^b(\underline{\sf{C}}(X_0))$ as the full subcategory of complexes concentrated in degree zero.
Above, in either direct sum, the value of $i - j$ is uniformly bounded among the nonzero terms precisely because the objects of $\sf{K}^b(\underline{\sf{C}}(X_0))$ can be represented by bounded complexes.

In the rest of this section, we only use the contravariant functors; the covariant ones will appear again in Section \ref{sec:kr}.
The following lemma is a consequence of Proposition \ref{prop:realization}.

\begin{lem}\label{lem:hom-pure}
If $\sf{C}(X_0)$ is $\SHom$-pure, then we have a commutative diagram:
\begin{equation}
\begin{tikzpicture}[baseline=(current bounding box.center), >=stealth]
\matrix(m)[matrix of math nodes, row sep=3em, column sep=9.5em, text height=2ex, text depth=0.5ex]
{	\sf{K}^b(\sf{C}(X_0))^\op
		&\sf{K}^b(\underline{\sf{C}}(X_0))^\op\\
	\sf{D}_{G, m}^b(X_0)^\op
		&\Mod_2^b(\bb{A}_K)\\
		};
\path[->,font=\scriptsize, auto]
(m-1-1) 	edge node{$\sf{K}^b\kappa$} (m-1-2)
			edge node[left]{$\rho$} (m-2-1)
(m-1-2)		edge node{$\bigoplus_{i, j} \Hom(-, K\langle j\rangle [i - j]_\triangle)$} (m-2-2)
(m-2-1)		edge node{$\bigoplus_{i, j} \gr_j^{\bm{w}} \SHom^i(-, K)$} (m-2-2);
\end{tikzpicture}
\end{equation}
\end{lem}

\subsection{}

We now set $X_0 = \cal{U}_0$ and $\sf{C}(X_0) = \sf{C}(\cal{U}_0)$ as in Definition \ref{df:unipotent}.
Recall that $\cal{S} \in \sf{C}(\cal{U}_0)$ denotes the (mixed) Springer sheaf.

\begin{thm}[Lusztig]\label{thm:lusztig-algebra}
We have $\bb{A}_W \simeq \bb{A}_\cal{S}$ as graded $\QL$-algebras.
\end{thm}

\begin{proof}
By \cite[Cor.\ 6.4]{lusztig_1988}, $\bb{A}_W 
\simeq \End^\ast(\xi \cal{S})$.
By Lemma \ref{lem:unipotent} (or \cite[Lem.\ 5.4]{rider}), we have $\End^i(\xi\cal{S}) \simeq \SHom^0(\cal{S}, \cal{S}\langle i\rangle)$ for all $i$ after we forget Frobenius.
\end{proof}

\begin{rem}
Although we will not use it, the proof of \cite[Prop.\ 6.5]{rider} implies that the functor $\bigoplus_i \Hom(\cal{S}, (-)\langle i\rangle)$ induces a triangulated equivalence
\begin{align}
\sf{K}^b(\sf{C}_{\hat{W}}(\cal{U}_0))
\simeq \sf{D}^b(\Mod_1(\bb{A}_W)),
\end{align}
where the right-hand side is the bounded derived category of $\Mod_1(\bb{A}_W)$.
\end{rem}

\begin{df}
The \dfemph{$\bb{A}_W$ trace} on the Hecke category is the contravariant functor
\begin{align}
\sf{AH} : \sf{H}_W^\op \to \Mod_2^b(\bb{A}_W)
\end{align}
defined by the commutative diagram
\begin{equation}\label{eq:a_w}
{\small
\begin{tikzpicture}[baseline=(current bounding box.center), >=stealth]
\matrix(m)[matrix of math nodes, row sep=2.5em, column sep=6.5em, text height=2ex, text depth=0.5ex]
{	\sf{K}^b(\sf{C}(\cal{B}_0 \times \cal{B}_0))^\op
		&\sf{K}^b(\sf{C}(\cal{U}_0))^\op
		&\sf{K}^b(\underline{\sf{C}}(\cal{U}_0))^\op\\
	\sf{D}_{G, m}^b(\cal{B}_0 \times \cal{B}_0)^\op
		&\sf{D}_{G, m}^b(\cal{U}_0)^\op
		&\Mod_2^b(\bb{A}_W)\\
		};
\path[->,font=\scriptsize, auto]
(m-1-1)		edge node{$\sf{K}^b(i^\ast\langle -r\rangle \circ \sf{CH})$} (m-1-2)
(m-2-1)		edge node{$i^\ast\langle -r\rangle \circ \sf{CH}$} (m-2-2)
(m-1-2) 	edge node{$\sf{K}^b\kappa$} (m-1-3)
(m-1-1)		edge node[left]{$\rho$} (m-2-1)
(m-1-2)		edge node{$\rho$} (m-2-2)
(m-1-3)		edge (m-2-3)
(m-2-2)		edge (m-2-3);
\end{tikzpicture}}
\end{equation}
that we obtain by combining Remark \ref{rem:compatible}, Lemma \ref{lem:unipotent}, Lemma \ref{lem:hom-pure}, and Theorem \ref{thm:lusztig-algebra}.
Explicitly,
\begin{align}\begin{split}
\sf{AH}^{i, j} 
&=	\Hom_{\sf{K}^b(\underline{\sf{C}}(\cal{U}_0))}(\sf{K}^b(\kappa \circ i^\ast\langle -r\rangle \circ \sf{CH})(-), \cal{S}\langle j\rangle [i - j]_\triangle)\\
&\simeq 
	\gr_j^{\bm{w}} \SHom^i((\rho \circ \sf{K}^b(i^\ast\langle -r\rangle \circ \sf{CH}))(-), \cal{S})\\
&\simeq
	\gr_j^{\bm{w}} \SHom^i((i^\ast\langle -r\rangle \circ \sf{CH} \circ \rho)(-), \cal{S})
\end{split}\end{align}
for all $i, j$.
It is a monoidal trace in the sense of Lemma \ref{lem:monoidal-trace} because it factors through $\sf{CH}$.
\end{df}

\newpage
\section{Steinberg Schemes}\label{sec:steinberg}

\subsection{}

In Section \ref{sec:hecke}, we reviewed the (mixed) Springer sheaf $\cal{S} \in \sf{D}_{G, m}^b(\cal{U}_0)$ and the (mixed) Grothendieck--Springer sheaf $\cal{G} \in \sf{D}_{G, m}^b(G_0)$.
In Section \ref{sec:a_w}, we gave a construction of a graded algebra $\bb{A}_K$ from a mixed complex $K$: explicitly,
\begin{align}\begin{split}
\bb{A}_K 
&= \bigoplus_n
{\gr_0^{\bm{w}}}
\SHom^0(K, K\langle n\rangle)\\
&\simeq \bigoplus_n
{\gr_n^{\bm{w}}}
\SEnd^n(K).
\end{split}\end{align}
In this section, we use topological methods to prove Theorem \ref{thm:pure-iso}:
The map 
\begin{align}
i^\ast : \bb{A}_\cal{G} \to \bb{A}_\cal{S}
\end{align}
induced by $i : \cal{U}_0 \to \cal{G}_0$ via Lemma \ref{lem:pullback} is an isomorphism of graded $\QL$-algebras.

\begin{rem}\label{rem:impure}
By Lemma \ref{lem:unipotent}, $\SEnd^n(\cal{S})$ is concentrated in weight $n$ (see also \cite[Lem.\ 5.4]{rider}.)
However, if $n > 0$, then $\SEnd^n(\cal{G})$ is not concentrated in weight $n$, essentially because the cohomology of the Grothendieck--Springer resolution of $G$ contains the cohomology of a torus.
See Example \ref{ex:identity}.
\end{rem}

\begin{rem}
Our proof of Theorem \ref{thm:pure-iso} is inspired by an unpublished note of Riche about the analogous statement at the level of Lie algebras, which is well-known.
Riche states that his note was inspired by \cite[\S{3.4}]{cg}.
\end{rem}

\subsection{}

If $X_0$ is a $G_0$-variety over $\bb{F}$, with structure map $a : X_0 \to \point_0$, then its mixed equivariant dualizing complex is 
\begin{align}
\omega = \omega_{X_0/G_0} \vcentcolon= a^!\QL \in \sf{D}_{G, m}^b(X_0).
\end{align}
Its equivariant Borel--Moore homology (with coefficients in $\QL$) is
\begin{align}
\ur{H}_{-\ast}^{!, G}(X) \vcentcolon= \SHom^\ast(\QL, \omega),
\end{align}
viewed as a graded $\QL[F]$-module.
Borel--Moore homology is functorial in the following senses:
\begin{enumerate}
\item 	If $f : Y \to X$ is proper, then it induces a pushforward 
\begin{align}
f_\ast : \ur{H}_{-\ast}^{!, G}(Y) \to \ur{H}_{-\ast}^{!, G}(X)
\end{align}
via the composition $f_\ast \omega_{Y/G} \xrightarrow{\sim} f_! f^! \omega_{X/G} \to \omega_{X/G}$.
\item 	If $f : Y \to X$ is smooth of relative dimension $d$, then it induces a pullback 
\begin{align}
f^\ast : \ur{H}_{-\ast}^{!, G}(X) \to \ur{H}_{-\ast}^{!, G}(Y)\langle -2d\rangle
\end{align}
via the composition $\omega_{X/G} \to f_\ast f^\ast \omega_{X/G}  \xrightarrow{\sim} f_\ast \omega_{Y/G}\langle -2d\rangle$.
		If $f$ is an affine-space bundle, then $f^\ast$ is a shifted isomorphism.
\item 	If $i : Z \to X$ is a closed immersion with complement $j : U \to X$, then it induces a long exact sequence
\begin{align}\label{eq:long-exact}
\cdots \to \ur{H}_i^{!, G}(Z) \to \ur{H}_i^{!, G}(X) \to \ur{H}_i^{!, G}(U) \to \ur{H}_{i - 1}^{!, G}(Z) \to \cdots
\end{align}
via the exact triangle $i_!i^! \to \id \to j_\ast j^\ast \to i_!i^![1]$.
\item 	If we have a cartesian square
\begin{equation}
\begin{tikzpicture}[baseline=(current bounding box.center), >=stealth]
\matrix(m)[matrix of math nodes, row sep=2em, column sep=2.5em, text height=2ex, text depth=0.5ex]
{	Y
		&Y'\\
	X
		&X'\\
		};
\path[->,font=\scriptsize, auto]
(m-1-1)		edge node{$f_Y$} (m-1-2)
(m-1-1)		edge node[left]{$g$} (m-2-1)
(m-1-2)		edge node{$g'$} (m-2-2)
(m-2-1)		edge node{$f$} (m-2-2);
\end{tikzpicture}
\end{equation}
in which $[X/G], [X'/G]$ are smooth of respective dimensions $d, d'$, then $f$ induces a \dfemph{pullback with restricted support}
\begin{align}
f_Y^\ast : \ur{H}_{-\ast}^{!, G}(Y') \to \ur{H}_{-\ast}^{!, G}(Y)\langle 2(d' - d)\rangle
\end{align}
via the composition
\begin{align}
\omega_{Y'/G}\langle -2d'\rangle \xrightarrow{\sim} {g'}^! \QL \to f_{Y, \ast} g^! f^\ast \QL \xrightarrow{\sim} f_{Y, \ast} \omega_{Y/G}\langle -2d\rangle.
\end{align}
Here, we use proper base change \cite[\S{5}]{lo_2008} to identify ${g'}^! f_\ast \simeq f_{Y, \ast} g^!$.
\end{enumerate}

\subsection{}

We can express $\SHom$-spaces between pushforward sheaves, like $\bb{A}_\cal{S}$ and $\bb{A}_\cal{G}$, in terms of the Borel--Moore homology of associated schemes.

\begin{lem}\label{lem:borel-moore}
If we have a cartesian square of $G$-varieties
\begin{equation}
\begin{tikzpicture}[baseline=(current bounding box.center), >=stealth]
\matrix(m)[matrix of math nodes, row sep=2em, column sep=2.5em, text height=2ex, text depth=0.5ex]
{	\tilde{Y}
		&\tilde{X}\\
	Y
		&X\\
		};
\path[->,font=\scriptsize, auto]
(m-1-1)		edge node{$\tilde{f}$} (m-1-2)
(m-1-1)		edge node[left]{$p_Y$} (m-2-1)
(m-1-2)		edge node{$p$} (m-2-2)
(m-2-1)		edge node{$f$} (m-2-2);
\end{tikzpicture}
\end{equation}
in which $[\tilde{X}/G]$ is smooth of dimension $\tilde{d}$, then
\begin{align}\label{eq:borel-moore}
\SHom^\ast(f_! \QL, p_\ast \QL)
\simeq \ur{H}_{-\ast}^{!, G}(\tilde{Y})\langle -2\tilde{d}\rangle.
\end{align}
In particular, if $p$ is proper and $K = p_\ast \QL$, then
\begin{align}
\bb{A}_K^n \simeq \gr_{n - 2\tilde{d}}^{\bm{w}} \ur{H}_{-(n - 2\tilde{d})}^{!, G}(\tilde{X} \times_X \tilde{X})
\end{align}
for all $n$, and $\ur{H}_{-\ast}^{!, G}(\tilde{Y})$ forms a left graded $\bb{A}_K$-module.
\end{lem}

\begin{proof}
By proper base change,
\begin{align}\begin{split}
f^! p_\ast \QL 
&\simeq p_{Y, \ast} \tilde{f}^! \QL 
\simeq p_{Y, \ast} \tilde{f}^! \omega_{\tilde{X}/G} \langle -2\tilde{d}\rangle 
\simeq p_{Y, \ast} \omega_{\tilde{Y}/G} \langle -2\tilde{d}\rangle,
\end{split}\end{align}
giving $\SHom(f_! \QL, p_\ast \QL) \simeq f_\ast \SHom(\QL, f^! p_\ast \QL) \simeq f_\ast p_{Y, \ast} \SHom(\QL, \omega_{\tilde{Y}/G})\langle -2\tilde{d}\rangle$.
\end{proof}

\subsection{}

In the notation of Section \ref{sec:hecke}, let
\begin{align}
\cal{Z}_0 &= \tilde{\cal{U}}_0 \times_{\cal{U}_0} \tilde{\cal{U}}_0,\\
\cal{Z}'_0 &= \tilde{G}_0 \times_{G_0} \tilde{G}_0.
\end{align}
In the literature, the underlying reduced scheme of $\cal{Z}$ is called the \dfemph{Steinberg variety}, since it was introduced in \cite{steinberg_1976}.
Lemma \ref{lem:borel-moore} implies that
\begin{align}
\label{eq:z}
\bb{A}_\cal{S}^n &\simeq \gr_{n + 2r}^{\bm{w}} \ur{H}_{-(n + 2r)}^{!, G}(\cal{Z}) \pa{\simeq \ur{H}_{-(n + 2r)}^{!, G}(\cal{Z})},\\
\label{eq:z-prime}
\bb{A}_\cal{G}^n &\simeq \gr_n^{\bm{w}} \ur{H}_{-n}^{!, G}(\cal{Z}').
\end{align}
It remains to compare the Borel--Moore homologies of $\cal{Z}$ and $\cal{Z}'$.

Let us outline the strategy.
We will lift the Bruhat strata of $\cal{B} \times \cal{B}$ to strata $\cal{Z}_w \subseteq \cal{Z}$ and $\cal{Z}'_w \subseteq \cal{Z}'$.
To compare the Borel--Moore homologies of $\cal{Z}_w$ and $\cal{Z}'_w$, we will further replace these strata with affine-space bundles $\cal{Y}_w$ and $\cal{Y}'_w$ such that there is an explicit map $\cal{Y}_w \times T \to \cal{Y}'_w$, where $T$ is a maximal torus of $G$.
Finally, we will use the long exact sequence \eqref{eq:long-exact} to patch the stratum-wise comparisons together.

\subsection{}

Let $\tilde{i} : \tilde{\cal{U}}_0 \to \tilde{G}_0$ be the lift of the inclusion $i : \cal{U}_0 \to G_0$.
Let $\tilde{p}_1, \tilde{p}_2 :\cal{Z}_0 \to \tilde{\cal{U}}_0$ be the two projection maps, and let $\tilde{p}'_1, \tilde{p}'_2 : \cal{Z}'_0 \to \tilde{G}_0$ be defined similarly.
Below, we will continually refer to the commutative diagram:
\begin{equation}\label{eq:tilde}
\begin{tikzpicture}[baseline=(current bounding box.center), >=stealth]
\matrix(m)[matrix of math nodes, row sep=2em, column sep=2.5em, text height=2ex, text depth=0.5ex]
{	\cal{Z}_0
		&\cal{Z}'_0\\
	\tilde{\cal{U}}_0
		&\tilde{G}_0\\
	\cal{U}_0
		&G_0\\
		};
\path[->,font=\scriptsize, auto]
(m-1-1)		edge node{$i_\cal{Z}$} (m-1-2)
(m-1-1)		edge node[left]{$\tilde{p}_j$} (m-2-1)
(m-1-2)		edge node{$\tilde{p}'_j$} (m-2-2)
(m-2-1)		edge node{$\tilde{i}$} (m-2-2)
(m-2-1)		edge node[left]{$p$} (m-3-1)
(m-2-2)		edge node{$p'$} (m-3-2)
(m-3-1)		edge node{$i$} (m-3-2);
\end{tikzpicture}
\end{equation}
Note that both squares in the diagram are cartesian.

\begin{lem}
We have a commutative diagram
\begin{equation}
\begin{tikzpicture}[baseline=(current bounding box.center), >=stealth]
\matrix(m)[matrix of math nodes, row sep=2em, column sep=2.5em, text height=2ex, text depth=0.5ex]
{	\bb{A}_\cal{G}^n
		&\bb{A}_\cal{S}^n\\
	\gr_n^{\bm{w}} \ur{H}_{-n}^{!, G}(\cal{Z}')
		&\ur{H}_{-(n + 2r)}^{!, G}(\cal{Z})\\
		};
\path[->,font=\scriptsize, auto]
(m-1-1)		edge node{$i^\ast$} (m-1-2)
(m-1-1)		edge node[rotate=90,above]{$\sim$} (m-2-1)
(m-1-2)		edge node[rotate=90,below]{$\sim$} (m-2-2)
(m-2-1)		edge node{$i_\cal{Z}^\ast$} (m-2-2);
\end{tikzpicture}
\end{equation}
where $i^\ast$ is defined by \eqref{eq:pullback}, the vertical isomorphisms are \eqref{eq:z} and \eqref{eq:z-prime}, and $i_\cal{Z}^\ast$ is the pullback with restricted support induced by $\tilde{i}$ via the top square of \eqref{eq:tilde}.
\end{lem}

\begin{proof}
The morphism $i_\cal{Z}^\ast$ is induced by the morphism 
\begin{align}
(\tilde{p}')_1^! \to (\tilde{p}')_1^! \tilde{i}_\ast \tilde{i}^\ast
\xrightarrow{\sim} i_{\cal{Z}, \ast} {\tilde{p}}_1^! \tilde{i}^\ast,
\end{align}
while the morphism $i^\ast$ is induced by the composition
\begin{align}\begin{split}
\tilde{p}'_{2, \ast} (\tilde{p}')_1^! 
&\xrightarrow{\sim} {p'}^! p'_\ast \\
&\to {p'}^! i_\ast i^\ast p'_\ast\\
&\xrightarrow{\sim} \tilde{i}_\ast p^! p_! \tilde{i}^\ast\\
&\xrightarrow{\sim} \tilde{i}_\ast \tilde{p}_{2, !} \tilde{p}_1^! i^\ast\\
&\xrightarrow{\sim} \tilde{p}'_{2, \ast} i_{\cal{Z}, \ast} \tilde{p}_1^! i^\ast.
\end{split}\end{align}
By further diagram-chasing, the latter is obtained by applying $\tilde{p}'_{2, \ast}$ to the former.
\end{proof}

To analyze $i_\cal{Z}^\ast$, we decompose $\cal{Z}_0$ and $\cal{Z}'_0$ into strata.
For each $w \in W$, we define subschemes $\cal{Z}_{w, 0} \subseteq \cal{Z}_0$ and $\cal{Z}'_{w, 0} \subseteq \cal{Z}'_0$ by stipulating that both squares in the commutative diagram below are cartesian:
\begin{equation}\label{eq:z-w}
\begin{tikzpicture}[baseline=(current bounding box.center), >=stealth]
\matrix(m)[matrix of math nodes, row sep=2em, column sep=2.5em, text height=2ex, text depth=0.5ex]
{	\cal{Z}_{w, 0}
		&\cal{Z}'_{w, 0}
		&O_{w, 0}\\
	\cal{Z}_0
		&\cal{Z}'_0
		&\cal{B}_0 \times \cal{B}_0\\
		};
\path[->,font=\scriptsize, auto]
(m-1-1)		edge node{$i_{\cal{Z}, w}$} (m-1-2)
(m-1-1)		edge (m-2-1)
(m-1-2)		edge (m-1-3)
(m-1-2)		edge (m-2-2)
(m-2-1)		edge node{$i_\cal{Z}$} (m-2-2)
(m-2-2)		edge (m-2-3)
(m-1-3)		edge node{$j_w$} (m-2-3);
\end{tikzpicture}
\end{equation}
Here, the map $\cal{Z}'_0 \to \cal{B}_0 \times \cal{B}_0$ is induced by the map $\cal{Z}'_0 \xrightarrow{\tilde{p}'_1 \times \tilde{p}'_2} \tilde{G}_0 \times \tilde{G}_0$.

To analyze $i_{\cal{Z}, w} : \cal{Z}_{w, 0} \to \cal{Z}'_{w, 0}$, we rewrite these strata in a way that expresses their structure as fiber bundles over $O_{w, 0}$.
Choose a split maximal torus $T_0 \subseteq G_0$.
Choose a Borel $B_0 \supseteq T_0$ with unipotent radical $U_0 \subseteq B_0$, so that $B_0 \simeq T_0 \ltimes U_0$.
Let $N_0$ be the normalizer of $T_0$ in $G_0$.
For each $w \in W$, fix a lift $\dot{w} \in N_0$.
Then $\dot{w}B_0 \dot{w}^{-1}$ contains $\dot{w}T_0 \dot{w}^{-1} = T_0$ and only depends on $w$.
If we write 
\begin{align}
B_{w, 0} &= B_0 \cap \dot{w}B_0 \dot{w}^{-1},\\
U_{w, 0} &= U_0 \cap \dot{w}U_0 \dot{w}^{-1},
\end{align}
then $B_{w, 0} \simeq T_0 \ltimes U_{w, 0}$.
In what follows, let $B_{w, 0}$ act on $G_0$ by right multiplication and on $U_{w, 0}$ and itself by right conjugation.

\begin{lem}\label{lem:assoc-bundle}
We have a commutative diagram
\begin{equation}
\begin{tikzpicture}[baseline=(current bounding box.center), >=stealth]
\matrix(m)[matrix of math nodes, row sep=2em, column sep=1.5em, text height=2ex, text depth=0.5ex]
{	(G_0 \times U_{w, 0})/B_{w, 0}
		&(G_0 \times B_{w, 0})/B_{w, 0}
		&G_0/B_{w, 0}\\
	\cal{Z}_{w, 0}
		&\cal{Z}'_{w, 0}
		&O_{w, 0}\\
		};
\path[->,font=\scriptsize, auto]
(m-1-1)		edge (m-1-2)
(m-1-2)		edge (m-1-3)
(m-1-1)		edge node[rotate=90,above]{$\sim$} (m-2-1)
(m-1-2)		edge node[rotate=90,below]{$\sim$} (m-2-2)
(m-1-3)		edge node[rotate=90,below]{$\sim$} (m-2-3)
(m-2-1)		edge node{$i_{\cal{Z}, w}$} (m-2-2)
(m-2-2)		edge (m-2-3);
\end{tikzpicture}
\end{equation}
where the middle vertical map sends $[g, b] \mapsto [gbg^{-1}, gBg^{-1}, g\dot{w}B\dot{w}^{-1}g^{-1}]$.
\end{lem}

\begin{rem}
The vertical isomorphisms respectively transport the $G$-actions on $\cal{Z}_w$, $\cal{Z}'_w$, $O_w$ to the $G$-actions on $(G \times U_w)/B_w$, $(G \times B_w)/B_w$, $G/B_w$ given by left multiplication on the factor of $G$.
\end{rem}

\begin{proof}
The $w = 1$ case, entailing isomorphisms $(G \times U)/B \simeq \tilde{\cal{U}}$ and $(G \times B)/B \simeq \tilde{G}$, is explained in \cite[Ch.\ 4]{slodowy}, and the general case is similar.
\end{proof}

Again, let $T_0$ act on $G_0$ by right multiplication and on $U_{w, 0}$ and $B_{w, 0}$ by right conjugation.
Let $\cal{Y}_{w, 0}$ and $\cal{Y}'_{w, 0}$ be defined by the cartesian squares:
\begin{equation}
\begin{tikzpicture}[baseline=(current bounding box.center), >=stealth]
\matrix(m)[matrix of math nodes, row sep=2em, column sep=2.5em, text height=2ex, text depth=0.5ex]
{	\cal{Y}_{w, 0}
		&(G_0 \times U_{w, 0})/T_0\\
	\cal{Z}_{w, 0}
		&(G_0 \times U_{w, 0})/B_{w, 0}\\
		};
\path[->,font=\scriptsize, auto]
(m-1-1)		edge node{$\sim$} (m-1-2)
(m-1-1)		edge (m-2-1)
(m-1-2)		edge (m-2-2)
(m-2-1)		edge node{$\sim$} (m-2-2);
\end{tikzpicture}
\qquad
\begin{tikzpicture}[baseline=(current bounding box.center), >=stealth]
\matrix(m)[matrix of math nodes, row sep=2em, column sep=2.5em, text height=2ex, text depth=0.5ex]
{	\cal{Y}'_{w, 0}
		&(G_0 \times B_{w, 0})/T_0\\
	\cal{Z}'_{w, 0}
		&(G_0 \times B_{w, 0})/B_{w, 0}\\
		};
\path[->,font=\scriptsize, auto]
(m-1-1)		edge node{$\sim$} (m-1-2)
(m-1-1)		edge (m-2-1)
(m-1-2)		edge (m-2-2)
(m-2-1)		edge node{$\sim$} (m-2-2);
\end{tikzpicture}
\end{equation}
Let $i_{\cal{Y}, w} : \cal{Y}_{w, 0} \to \cal{Y}'_{w, 0}$ be the pullback of $i_{\cal{Z}, w} : \cal{Z}_{w, 0} \to \cal{Z}'_{w, 0}$.
The advantage of $i_{\cal{Y}, w}$ over $i_{\cal{Z}, w}$ is that the former can be factored as
\begin{align}
\cal{Y}_{w, 0} \xrightarrow{\sim} \cal{Y}_{w, 0} \times \{1\}
\to 
\cal{Y}_{w, 0} \times T_0 \xrightarrow{\sim} \cal{Y}'_{w, 0}.
\end{align}
If we let $G_0$ act trivially on $T_0$, then these maps are all $G_0$-equivariant.

\subsection{}

In what follows, if $M = \bigoplus_n M^n$ is a $\bb{Z}$-graded $\QL[F]$-module, then we write 
\begin{align}\begin{split}
M^\dagger 
&\simeq \bigoplus_n {\gr_n^{\bm{w}} M^n}.
\end{split}\end{align}
Let $i_{\cal{Z}, w}^\ast : \ur{H}_{-\ast}^{!, G}(\cal{Z}'_w) \to \ur{H}_{-\ast}^{!, G}(\cal{Z}_w)\langle 2r\rangle$ be the pullback with restricted support induced by $\tilde{i}$ via the top square of \eqref{eq:tilde} and the left square of \eqref{eq:z-w}, and let
\begin{align}
i_{\cal{Z}, w}^\dagger : \ur{H}_{-\ast}^{!, G}(\cal{Z}'_w)^\dagger \to \ur{H}_{-\ast}^{!, G}(\cal{Z}_w)\langle 2r\rangle
\end{align}
be its restriction to $\ur{H}_{-\ast}^{!, G}(\cal{Z}'_w)^\dagger$.
Let $i_{\cal{Y}, w}^\dagger$, $i_{\cal{Y}, w}^\dagger$ be defined similarly.

\begin{lem}
The map $i_{\cal{Y}, w}^\dagger$ is an isomorphism.
\end{lem}

\begin{proof}
Writing $N_w = \dim U_w$, we have
\begin{align}
\ur{H}_{-\ast}^{!, G}(\cal{Y}_w)
\simeq \ur{H}_{-\ast}^{!, G}(G/T)\langle 2N_w\rangle
\simeq \ur{H}_{-\ast}^{!, T}(\point)\langle 2N_w\rangle.
\end{align}
By Poincar\'e duality, $\ur{H}_{-i}^{!, T}(\point) \simeq \ur{H}_{!, T}^{-i}(\point) \simeq \ur{H}_T^{i - 2r}(\point)$, where $\ur{H}_{!, T}^\ast$ denotes $T$-equivariant, compactly-supported cohomology.
So Frobenius acts on $\ur{H}_{-i}^{!, T}(\point)$ by $q^{\frac{i}{2} - r}$, which means it is concentrated in weight $i - 2r$.

Again by Poincar\'e duality, $\ur{H}_{-j}^!(T) \simeq \ur{H}_!^{-j}(T)^\vee \simeq \ur{H}^{j + 2r}(T)$, where $\ur{H}_\ast^!$, \emph{resp.}\ $\ur{H}_!^\ast$, denotes non-equivariant Borel--Moore homology, \emph{resp.}\ compactly-supported cohomology.
So Frobenius acts on $\ur{H}_{-j}^!(T)$ by $q^{j + 2r}$, which means it is concentrated in weight $2j + 4r$.

Altogether, the K\"unneth isomorphism
\begin{align}
\ur{H}_{-n}^{!, G}(\cal{Y}'_w) \xrightarrow{\sim} \bigoplus_{i + j = n} \ur{H}_{-i}^{!, G}(\cal{Y}_w) \otimes \ur{H}_{-j}^!(T)
\end{align}
restricts to an isomorphism 
\begin{align}
\ur{H}_{-n}^{!, G}(\cal{Y}'_w)^\dagger \xrightarrow{\sim} \ur{H}_{-(n + 2r)}^{!, G}(\cal{Y}_w) \otimes H_{2r}^!(T).
\end{align}
But $i_{\cal{Y}, w}^\ast$ is the composition of the K\"unneth isomorphism with
\begin{align}\begin{split}
\ur{H}_{-\ast}^{!, G}(\cal{Y}_w) \otimes \ur{H}_{-\ast}^!(T) 
	&\to \ur{H}_{-\ast}^{!, G}(\cal{Y}_w) \otimes \ur{H}_{2r}^!(T)\\	
	&\xrightarrow{\sim} \ur{H}_{-\ast}^{!, G}(\cal{Y}_w) \otimes \ur{H}_0^!(\point)\langle 2r\rangle\\
	&\xrightarrow{\sim} \ur{H}_{-\ast}^{!, G}(\cal{Y}_w)\langle 2r\rangle,
\end{split}\end{align}
so we win.
\end{proof}

\begin{lem}\label{lem:i-z-pur}
The map $i_{\cal{Z}, w}^\dagger$ is an isomorphism.
\end{lem}

\begin{proof}
The maps $\cal{Y}_{w, 0} \to \cal{Z}_{w, 0}$ and $\cal{Y}'_{w, 0} \to \cal{Z}'_{w, 0}$ define affine-space bundles of rank $N_w$, so they induce pullback isomorphisms
\begin{align}
\ur{H}_{-\ast}^{!, G}(\cal{Z}_w) &\xrightarrow{\sim} \ur{H}_{-\ast}^{!, G}(\cal{Y}_w)\langle -2N_w\rangle,\\
\ur{H}_{-\ast}^{!, G}(\cal{Z}'_w) &\xrightarrow{\sim} \ur{H}_{-\ast}^{!, G}(\cal{Y}'_w)\langle -2N_w\rangle
\end{align}
that preserve weights.
Now we are done by the previous lemma.
\end{proof}

\begin{proof}[Proof of Theorem \ref{thm:pure-iso}]
Fix a total order $\preceq$ on $W$ in which $w_1 \preceq w_2$ implies $|w_1| \leq |w_2|$, where $|{-}|$ denotes Bruhat length.
Such a total order is necessarily a refinement of the Bruhat partial order.
For any $w \in W$, let
\begin{align}
\cal{Z}_{\preceq w, 0} 
&= \coprod_{x \preceq w} \cal{Z}_{x, 0} \subseteq \cal{Z}_0,\\
\cal{Z}_{\prec w, 0} 
&= \coprod_{x \prec w} \cal{Z}_{x, 0} \subseteq \cal{Z}_0.
\end{align}
Then $\cal{Z}_{w, 0} \to \cal{Z}_{\preceq w, 0}$ is a closed immersion with complement $\cal{Z}_{\prec w, 0} \to \cal{Z}_{\preceq w, 0}$.
Let $\cal{Z}'_{\preceq w, 0}, \cal{Z}'_{\prec w, 0}$ be defined similarly.
The map $i_\cal{Z} : \cal{Z}_0 \to \cal{Z}'_0$ gives rise to the following commutative diagram whose rows are long exact sequences and whose vertical arrows are pullbacks with restricted support:
\begin{equation}
\begin{tikzpicture}[baseline=(current bounding box.center), >=stealth]
\matrix(m)[matrix of math nodes, row sep=2.5em, column sep=1.5em, text height=2ex, text depth=0.5ex]
{	\cdots 
		&\ur{H}_n^{!, G}(\cal{Z}'_w)
		&\ur{H}_n^{!, G}(\cal{Z}'_{\preceq w})
		&\ur{H}_n^{!, G}(\cal{Z}'_{\prec w})
		&\cdots\\
	\cdots 
		&\ur{H}_n^{!, G}(\cal{Z}_w)\langle 2r\rangle
		&\ur{H}_n^{!, G}(\cal{Z}_{\preceq w})\langle 2r\rangle
		&\ur{H}_n^{!, G}(\cal{Z}_{\prec w})\langle 2r\rangle
		&\cdots\\
		};
\path[->,font=\scriptsize, auto]
(m-1-1)		edge (m-1-2)
(m-1-2)		edge (m-1-3)
(m-1-3)		edge (m-1-4)
(m-1-4)		edge (m-1-5)
(m-2-1)		edge (m-2-2)
(m-2-2)		edge (m-2-3)
(m-2-3)		edge (m-2-4)
(m-2-4)		edge (m-2-5)
(m-1-2)		edge node[left]{$i_{\cal{Z}, w}^\ast$} (m-2-2)
(m-1-3)		edge node{$i_{\cal{Z}, \preceq w}^\ast$} (m-2-3)
(m-1-4)		edge node{$i_{\cal{Z}, \prec w}^\ast$} (m-2-4);
\end{tikzpicture}
\end{equation}
Since $(-)^\dagger$ is exact, we obtain another commutative diagram
\begin{equation}
\begin{tikzpicture}[baseline=(current bounding box.center), >=stealth]
\matrix(m)[matrix of math nodes, row sep=2.5em, column sep=1.5em, text height=2ex, text depth=0.5ex]
{	\cdots 
		&\ur{H}_n^{!, G}(\cal{Z}'_w)^\dagger
		&\ur{H}_n^{!, G}(\cal{Z}'_{\preceq w})^\dagger
		&\ur{H}_n^{!, G}(\cal{Z}'_{\prec w})^\dagger
		&\cdots\\
	\cdots 
		&\ur{H}_n^{!, G}(\cal{Z}_w)\langle 2r\rangle
		&\ur{H}_n^{!, G}(\cal{Z}_{\preceq w})\langle 2r\rangle
		&\ur{H}_n^{!, G}(\cal{Z}_{\prec w})\langle 2r\rangle
		&\cdots\\
		};
\path[->,font=\scriptsize, auto]
(m-1-1)		edge (m-1-2)
(m-1-2)		edge (m-1-3)
(m-1-3)		edge (m-1-4)
(m-1-4)		edge (m-1-5)
(m-2-1)		edge (m-2-2)
(m-2-2)		edge (m-2-3)
(m-2-3)		edge (m-2-4)
(m-2-4)		edge (m-2-5)
(m-1-2)		edge node[left]{$i_{\cal{Z}, w}^\dagger$} (m-2-2)
(m-1-3)		edge node{$i_{\cal{Z}, \preceq w}^\dagger$} (m-2-3)
(m-1-4)		edge node{$i_{\cal{Z}, \prec w}^\dagger$} (m-2-4);
\end{tikzpicture}
\end{equation}
in which the top row remains exact.

We induct along $\preceq$.
If $w = 1$, then $i_{\cal{Z}, \prec w}^\dagger = i_{\cal{Z}, w}^\dagger$.
By Lemma \ref{lem:i-z-pur}, the latter is an isomorphism in every homological degree $n$.
Now let $w$ be arbitrary.
If $x$ is the immediate predecessor of $w$ in the total order $\preceq$, then $\cal{Z}_{\prec w} = \cal{Z}_{\preceq x}$ and $i_{\cal{Z}, \prec w}^\dagger = i_{\cal{Z}, \preceq x}^\dagger$.
Lemma \ref{lem:i-z-pur} and the inductive hypothesis respectively imply that $i_{\cal{Z}, w}^\dagger$ and $i_{\cal{Z}, \prec w}^\dagger$ are isomorphisms in every degree.
So, by the five lemma, $i_{\cal{Z}, \preceq w}^\dagger$ is an isomorphism in every degree as well, completing the induction.
\end{proof}

\newpage
\section{Steinberg Schemes of Braids}\label{sec:varieties}

\subsection{}

Let $S \subseteq W$ be the set of simple reflections in the chosen Coxeter presentation of $W$.
Let $\Br_W$ be the Artin braid group of $W$, and let $\Br_W^+ \subseteq \Br_W$ be the monoid of positive braids determined by $S$.
Each element $s \in S$ lifts to an element $\sigma_s \in \Br_W^+$, such that $\{\sigma_s\}_{s \in S}$ is a generating set for $\Br_W^+$.
The following result is \cite[\S{1.11}]{deligne_1997}.

\begin{thm}[Deligne]
There is a map from elements $\beta \in \Br_W^+$ to $G_0$-varieties $O(\beta)_0$ over $\cal{B}_0 \times \cal{B}_0$ such that:
\begin{enumerate}
\item 	$O(\sigma_s)_0 = O_{s, 0}$ for all $s \in S$.
\item 	If $\beta = \beta'\beta''$, then there is a fixed isomorphism 
\begin{align}
O(\beta)_0 \xrightarrow{\sim} O(\beta')_0 \times_{\cal{B}_0} O(\beta'')_0,
\end{align}
where the fiber product is formed with respect to the right projection $O(\beta')_0 \to \cal{B}_0$ and the left projection $O(\beta'')_0 \to \cal{B}_0$.
		These isomorphisms are associative in the sense that the diagram
		\begin{equation}
\begin{tikzpicture}[baseline=(current bounding box.center), >=stealth]
\matrix(m)[matrix of math nodes, row sep=2em, column sep=2em, text height=2ex, text depth=0.5ex]
{	O(\beta'\beta''\beta''')_0
		&O(\beta'\beta'')_0 \times_{\cal{B}_0} O(\beta''')_0\\
	O(\beta')_0 \times_{\cal{B}_0} O(\beta''\beta''')_0
		&O(\beta')_0 \times_{\cal{B}_0} O(\beta'')_0 \times_{\cal{B}_0} O(\beta''')_0\\
		};
\path[->,font=\scriptsize, auto]
(m-1-1)		edge node{$\sim$} (m-1-2)
			edge node[rotate=90,above]{$\sim$} (m-2-1)
(m-1-2)		edge node[rotate=90,below]{$\sim$} (m-2-2)
(m-2-1)		edge node{$\sim$} (m-2-2);
\end{tikzpicture}
\end{equation}
is commutative.

In particular, if $w = s_1 \cdots s_\ell$ is a reduced expression, then 
\begin{align}
O(\sigma_w)_0 
\simeq O(\sigma_{s_1})_0 \times_{\cal{B}_0} \cdots \times_{\cal{B}_0} O(\sigma_{s_\ell})_0
\simeq O_{w, 0}
\end{align}
by way of these fixed isomorphisms.
\end{enumerate}
\end{thm}

Deligne attributes the varieties $O(\beta)_0$ to Brou\'e--Michel \cite{bm_1997}, though they are implicit in Lusztig's earlier papers on character sheaves.
Note that $\dim O(\beta) = |\beta| + N$, where $|\beta|$ is the writhe of $\beta$ (see \S\ref{subsec:braid-group}).

\subsection{}

Rouquier introduced an extension of Deligne's result in which the full group $\Br_W$ replaces the monoid $\Br_W^+$, and in which objects of $\sf{H}_W$ replace varieties over $\cal{B}_0 \times \cal{B}_0$.
The fiber product over $\cal{B}_0$ is replaced by the convolution on $\sf{H}_W$.

Each element $s \in S$ gives rise to objects $\cal{R}_s^+, \cal{R}_s^- \in \sf{H}_W$ that can be represented by complexes of the following form:
\begin{align}
\begin{array}{rcllllll}
\cal{R}_s^+ &= &0 &\to &\underline{{\IC_s}} &\to &\IC_1\langle 1\rangle,\\[1ex]
\cal{R}_s^- &= &\IC_1\langle -1\rangle &\to &\underline{{\IC_s}} &\to &0.
\end{array}
\end{align}
Above, the underlined terms are the terms in degree zero.
To describe the morphisms, factor $j_1 : O_{1, 0} \to \cal{B}_0 \times \cal{B}_0$ as the composition
\begin{align}
O_{1, 0} \xrightarrow{i_s} \bar{O}_{s, 0} \xrightarrow{\bar{j}_s} \cal{B}_0 \times \cal{B}_0.
\end{align}
Since $O_{1, 0}$ and $\bar{O}_{s, 0}$ are smooth and closed in $\cal{B}_0^2$, we have
\begin{align}
\IC_1 &= j_{1, \ast}\QL\langle -r - N\rangle,\\
\IC_s &= \bar{j}_{s, \ast}\QL \langle 1 - r - N\rangle.
\end{align}
Then the nontrivial morphism in $\cal{R}_s^+$, \emph{resp.}\ $\cal{R}_s^-$, is induced by pushing forward
\begin{align}
\QL
	&\to i_{s, \ast} i_s^\ast \QL
	\xrightarrow{\sim} i_{s, \ast} \QL,\\
\emph{resp.}\qquad
i_{s, \ast} \QL\langle -2\rangle 
	&\xrightarrow{\sim} i_{s,!} i_s^! \QL
	\to \QL,
\end{align}
along $\bar{j}_s$.

In \cite{soergel}, Soergel gave a module-theoretic description of $\sf{C}(\cal{B} \times \cal{B})$ (see Remark \ref{rem:by}), and hence, of $\sf{H}_W$.
Under it, the complexes $\cal{R}_s^+$ and $\cal{R}_s^-$ correspond to the complexes of bimodules over $\Sym(\bb{V})$ defined by Rouquier in \S{9.2.1} and \S{9.2.4} of \cite{rouquier_2006}.
Thus, Thm.\ 10.5 of \emph{ibid.}\ amounts to the following result:

\begin{thm}[Rouquier]\label{thm:rouquier}
There is a map from elements $\beta \in \Br_W^+$ to objects $\cal{R}(\beta) \in \sf{H}_W$ such that:
\begin{enumerate}
\item 	$\cal{R}(\sigma_s) = \cal{R}_s^+$ and $\cal{R}(\sigma_s^{-1}) = \cal{R}_s^-$ for all $s \in S$.
\item 	If $\beta = \beta' \beta''$, then there is a fixed isomorphism
		\begin{align}
		\cal{R}(\beta) \xrightarrow{\sim} \cal{R}(\beta') \star \cal{R}(\beta'').
		\end{align}
		These isomorphisms are associative in the sense that the diagram
		\begin{equation}
\begin{tikzpicture}[baseline=(current bounding box.center), >=stealth]
\matrix(m)[matrix of math nodes, row sep=2em, column sep=2em, text height=2ex, text depth=0.5ex]
{	\cal{R}(\beta'\beta''\beta''')
		&\cal{R}(\beta'\beta'')_0 \star \cal{R}(\beta''')\\
	\cal{R}(\beta')_0 \star \cal{R}(\beta''\beta''')
		&\cal{R}(\beta') \star \cal{R}(\beta'') \star \cal{R}(\beta''')\\
		};
\path[->,font=\scriptsize, auto]
(m-1-1) 	edge node{$\sim$} (m-1-2)
			edge node[rotate=90,above]{$\sim$} (m-2-1)
(m-1-2) 	edge node[rotate=90,below]{$\sim$} (m-2-2)
(m-2-1) 	edge node{$\sim$} (m-2-2);
\end{tikzpicture}
\end{equation}
commutes.
\end{enumerate}
\end{thm}


\subsection{}

The results of Deligne and Rouquier are related as follows.
For all $w \in W$, let $j_\beta$ be the map $O(\beta)_0 \to \cal{B}_0 \times \cal{B}_0$.
Let
\begin{align}
\Delta(\beta) &= j_{\beta, !} \QL\langle |\beta| - r - N\rangle,\\
\nabla(\beta) &= j_{\beta, \ast} \QL\langle |\beta| - r - N\rangle.
\end{align}
In the case where $\beta = \sigma_s$, we abbreviate by writing $\Delta_s = \Delta(\beta)$ and $\nabla_s = \nabla(\beta)$.
The following lemma is strongly suggested by the statements in \cite[\S{11}]{rouquier_2006}, though it does not seem to appear there.

\begin{lem}\label{lem:realization-rouquier}
We have
\begin{align}
\rho(\cal{R}(\beta)) 
	&\simeq \Delta(\beta),\\
\rho(\cal{R}(\beta^{-1}))
	&\simeq \nabla(\beta)
\end{align}
for all $\beta \in \Br_W^+$.
\end{lem}

\begin{proof}
The explicit definitions of $\cal{R}_s^+$ and $\cal{R}_s^-$ show that
\begin{align}
\rho(\cal{R}_s^+) &\simeq \Delta_s,\\
\rho(\cal{R}_s^-) &\simeq \nabla_s
\end{align}
for all $s \in S$.
We now apply the commutativity of \eqref{eq:convolution-realization}.
\end{proof}

\subsection{}

We define $G_0$-schemes $G(\beta)_0$, $\cal{U}(\beta)_0$, $\cal{Z}(\beta)_0$ by stipulating that all squares in the commutative diagram below are cartesian.
\begin{equation}
\begin{tikzpicture}[baseline=(current bounding box.center), >=stealth]
\matrix(m)[matrix of math nodes, row sep=2.5em, column sep=3.5em, text height=2ex, text depth=0.5ex]
{	O(\beta)_0
		&G(\beta)_0
		&\cal{U}(\beta)_0
		&\cal{Z}(\beta)_0\\
	\cal{B}_0 \times \cal{B}_0
		&G_0 \times \cal{B}_0
		&\cal{U}_0 \times \cal{B}_0
		&\tilde{\cal{U}}_0 \times \cal{B}_0\\
		};
\path[->,font=\scriptsize, auto]
(m-1-2)		edge (m-1-1)
(m-1-3) 	edge (m-1-2)
(m-1-4) 	edge (m-1-3)
(m-1-1)		edge node[left]{$\pr_0 \times \pr_\ell$} (m-2-1)
(m-1-2)		edge (m-2-2)
(m-1-3) 	edge (m-2-3)			
(m-1-4) 	edge (m-2-4)
(m-2-2)		edge node[above]{$\act$} (m-2-1)
(m-2-3)		edge node[above]{$i \times \id$} (m-2-2)
(m-2-4)		edge node[above]{$p \times \id$} (m-2-3);
\end{tikzpicture}
\end{equation}
Let $\bb{1} \in \Br_W$ denote the identity braid.
Then $\cal{U}(\bb{1})_0 = \tilde{\cal{U}}_0$, the Springer resolution, and $\cal{Z}(\bb{1})_0 = \cal{Z}_0$, the Steinberg scheme from the previous section.

\begin{df}
For all $\beta \in \Br_W^+$, we say that $\cal{Z}(\beta)_0$ is the \dfemph{Steinberg scheme} attached to $\beta$.
\end{df}

In what follows, let $p_\beta$ be the composition $\cal{U}(\beta)_0 \to \cal{U}_0 \times \cal{B}_0 \to \cal{U}_0$.
Note that $p_{\bb{1}, !}\QL \simeq p_{\bb{1}, \ast}\QL \simeq \cal{S}$.
We can now prove Theorem \ref{thm:a_w-to-varieties}, relating the $\bb{A}_W$ trace to the generalized Steinberg schemes $\cal{Z}(\beta)_0$.

\begin{thm}\label{thm:a_w-to-varieties-strong}
For all $\beta \in \Br_W^+$ and $i, j$, we have:
\begin{align}
\sf{AH}^{i, j}(\cal{R}(\beta))
	&\simeq \gr_{j + r}^{\bm{w}} \SHom^{i + r}(p_{\beta, !}\QL, \cal{S}),\\
\sf{AH}^{i, j}(\cal{R}(\beta^{-1}))
	&\simeq \gr_{j + r}^{\bm{w}} \SHom^{i + r}(p_{\beta, \ast}\QL, \cal{S}).
\end{align}
In particular, since $\cal{S} = p_\ast \QL\langle -r\rangle$, we have
\begin{align}\begin{split}
\sf{AH}^{i, j}(\cal{R}(\beta))
	&\simeq \gr_{j + 2r}^{\bm{w}} \ur{H}_{-(i + 2r)}^{!, G}(\cal{Z}(\beta))
\end{split}\end{align}
by Lemma \ref{lem:borel-moore}.
\end{thm}

\begin{proof}
By the commutativity of \eqref{eq:a_w}, Lemma \ref{lem:realization-rouquier}, and both smooth and proper base change,
\begin{align}\begin{split}
\sf{AH}^{i, j}(\cal{R}(\beta)) 
&\simeq 	\gr_j^{\bm{w}} \SHom^i(\sf{K}^b(i^\ast\langle -r\rangle \circ \sf{CH} \circ \rho)(R(\beta)), \cal{S})\\
&\simeq 	\gr_j^{\bm{w}} \SHom^i((i^\ast\langle -r\rangle \circ \sf{CH})(\Delta(\beta)), \cal{S})\\
&\simeq 	\gr_{j + r}^{\bm{w}} \SHom^{i + r}(p_{\beta, !} \QL, \cal{S}).
\end{split}\end{align}
The argument for $\cal{R}(\beta^{-1})$ is completely analogous.
The remaining statement follows from Lemma \ref{lem:borel-moore}.
\end{proof}

\newpage
\section{Khovanov--Rozansky Homology}\label{sec:kr}

\subsection{}

A link in a $3$-manifold is the image of a disjoint union of finitely many circles under a smooth embedding.
The HOMFLY series of a link in $S^3$ is a bivariate Laurent series that only depends on its isotopy class \cite{homfly}.

Dunfield--Gukov--Rasmussen predicted \cite{dgr}, and Khovanov--Rozansky constructed \cite{kr}, a categorification of the HOMFLY series, sending a link in $3$-space to a triply-graded vector space over $\bb{Q}$.
This isotopy invariant is now known as \dfemph{HOMFLY} or \dfemph{Khovanov--Rozansky homology}.
The original construction used the homotopy theory of matrix factorizations.
In \cite{khovanov}, Khovanov gave an alternative approach via the Hochschild homology of Soergel bimodules for the groups $W = S_n$.
Via Soergel's Erweiterungssatz \cite[Thms.\ 15-17]{soergel}, the output of his construction (after base change from $\bb{Q}$ to $\QL$) is equivalent to a functor
\begin{align}
\sf{HHH} : \sf{H}_W \to \sf{K}^b(\Vect_2) \subseteq \Vect_3,
\end{align}
where $\Vect_d$ denotes the category of $\bb{Z}^d$-graded $\QL$-vector spaces.
Khovanov's construction extends to Soergel bimodules for any Coxeter group, so in particular, $\sf{HHH}$ can be defined on the Hecke category of any such group.

\begin{rem}\label{rem:dgr}
To dispel any ambiguity from our conventions, we state explicitly:
If $\beta$ is a topological braid on $n = r + 1$ strands, and $\hat{\beta}$ is the link closure of $\beta$, then
\begin{align}
\bb{P}(\beta) = {(at)^{|\beta|}} a^{-r} \sum_{i, j, k} 
{(a^2 \Q^{\frac{1}{2}} t)^{r - i}} \Q^{\frac{j}{2}} t^k
\dim \sf{HHH}^{i, i + j, k}(\cal{R}(\beta))
\end{align}
is an isotopy invariant of $\hat{\beta}$ that equals $1$ when $\hat{\beta}$ is the unknot.
It matches the invariant in \cite{dgr} once we substitute $\Q^{\frac{1}{2}} = q$.
In particular, for $\beta = \bb{1}$, we have
\begin{align}
\bb{P}(\bb{1}) = \pa{\frac{\Q^{\frac{1}{2}}}{a} \pa{\frac{1 + a^2 t}{1 - \Q}}}^r.
\end{align}
Compare to Example \ref{ex:identity}, below.
\end{rem}

\subsection{}

In \cite{ww_2008}, Webster--Williamson interpreted the Hochschild homology of a Soergel bimodule as the hypercohomology of a sum of shifted pure perverse sheaves over $G$.
In subsequent papers \cite{ww_2011, ww_2017}, they showed how to interpret $\sf{HHH}^{\ast, \ast, \ast}$ in a similar way, except in terms of a \emph{complex} of such objects.
We will rewrite their work in our notation.

For any $G_0$-variety $X_0$ over $\bb{F}$, we write 
\begin{align}
\bb{H}_G^\ast(X, -) \vcentcolon= \SHom^\ast(\QL, -) : \underline{\sf{D}}_{G, m}^b(X_0) \to \Mod_1(\QL[F]),
\end{align}
where $\QL[F]$ is concentrated in degree zero.
In other words, $\bb{H}_G^\ast(X, K)$ is the hypercohomology of $\xi K \in \underline{\sf{D}}_G^b(G)$, implicitly endowed with the Frobenius action coming from $K$.
With this notation, the theorem below summarizes the main results of \cite{ww_2008} and \cite{ww_2017}.

\begin{thm}[Webster--Williamson]\label{thm:ww}
We have
\begin{align}\begin{split}
\sf{HHH}^{i, i + j, k}
\simeq 
	\ur{H}^k \circ \sf{K}^b(\gr_{i + j}^{\bm{w}} \bb{H}_G^j(G, (\kappa \circ \sf{CH})(-)))
\end{split}\end{align}
for all $i, j, k \in \bb{Z}$.
In particular, we have
\begin{align}
\sf{HHH}^{i, i + j, k}(\IC_{w, 0}[-n]_\triangle)
&\simeq	\left\{\begin{array}{ll}
\gr_{i + j}^{\bm{w}} \bb{H}_G^j(G, \sf{CH}(\IC_w)) &k = n\\
0	&k \neq n
\end{array}\right.
\end{align}
for all $w \in W$ and $n \in \bb{Z}$.
\end{thm}

\begin{ex}\label{ex:identity}
Below, we review the calculation of $\sf{HHH}(\IC_1) \simeq \gr_\ast^{\bm{w}} \bb{H}_G^\ast(G, \cal{G})$.
Since $\cal{G}$ is the pushforward of the constant sheaf on the Grothendieck--Springer resolution, it suffices to calculate the $G$-equivariant cohomology of $\tilde{G}$ together with its weight filtration.

Fix a maximal torus $T_0 \subseteq G_0$ and a Borel $B_0 \supseteq T_0$.
By Lemma \ref{lem:assoc-bundle}, $[\tilde{G}_0/G_0] \simeq [B_0/B_0]$, where $B_0$ acts on itself by conjugation.
Therefore, there are weight-preserving isomorphisms
\begin{align}\begin{split}
\bb{H}_G^j(G, \cal{G}) 
&\simeq \ur{H}_B^j(B, \QL) \\
&\simeq \ur{H}_T^j(T, \QL) \\
&\simeq \bigoplus_i {\ur{H}_T^{j - i}(\point, \QL) \otimes \ur{H}^i(T, \QL)}.
\end{split}\end{align}
Above, Frobenius acts by $q^{\frac{j - i}{2}}$ on $\ur{H}_T^{j - i}(\point)$ and by $q^i$ on $\ur{H}^i(T)$, so it acts by $q^{\frac{i + j}{2}}$ on the $i$th summand of the last expression.
We deduce that 
\begin{align}
\gr_{i + j}^{\bm{w}} \bb{H}_G^j(G, \cal{G}) \simeq \ur{H}_T^{j - i}(\point, \QL) \otimes \ur{H}^i(T, \QL).
\end{align}
The proof of \cite[Prop.\ 11]{ww_2011} shows that this isomorphism is $W$-equivariant with respect to:
\begin{enumerate}
\item 	The $W$-action on the left-hand side induced by $\QL[W] \subseteq \bb{A}_W \simeq \bb{A}_\cal{G}$, \emph{i.e.}, the Springer action arising from the monodromy of $\cal{G}$ over the regular semisimple locus of $G$.
\item 	The $W$-action on the right-hand side induced by $\ur{H}_T^{j - i}(\point, \QL) \simeq \Sym^{\frac{j - i}{2}}(\bb{V})$ and  $\ur{H}^i(T, \QL) \simeq \Alt^i(\bb{V})$.
\end{enumerate}
\end{ex}

\subsection{}

In the rest of this section, we prove Theorem \ref{thm:a_w-to-kr}, stating that Khovanov--Rozansky homology is a summand of the $\bb{A}_W$ trace.

In what follows, recall from Section \ref{sec:hecke} that for all $\phi \in \hat{W}$, we write $V_\phi$ for an irreducible representation of $W$ that affords the character $\phi$, endowed with the trivial action of Frobenius.

\begin{lem}\label{lem:sheaves-to-reps}
Let $K \in \sf{C}_{\hat{W}}(G_0)$ be an indecomposable shifted perverse sheaf, and write $K = \cal{G}_\phi \otimes M$, where $\phi \in \hat{W}$ and $M$ is the pullback of a pure weight-zero object of $\sf{D}_{G, m}^b(\point_0)$.
Then we have $W$-equivariant isomorphisms:
		\begin{align}
		\label{eq:covariant}		
		\gr_{i + j}^{\bm{w}} \bb{H}_G^j(G, K)
		&\simeq \Hom_W(V_\phi, {(\Sym^{\frac{j - i}{2}} \otimes {\Alt^i})}(\bb{V})) \otimes \xi M,\\[1ex]
		\label{eq:contravariant}
		\gr_0^{\bm{w}} \SHom^0(K, \cal{G}\langle i\rangle)
		&\simeq \Hom_W(V_\phi^\vee, \QL[W] \otimes \Sym^{\frac{j}{2}}(\bb{V})) \otimes \xi M^\vee\\
		&\simeq 
			V_\phi \otimes \Sym^{\frac{i}{2}}(\bb{V}) \otimes \xi M^\vee.\nonumber
		\end{align}
Moreover, these isomorphisms are functorial in $K$.
\end{lem}

\begin{proof}
By Example \ref{ex:identity}, there is a $W$-equivariant isomorphism 
\begin{align}
\gr_{i + j}^{\bm{w}} \bb{H}_G^j(G, \cal{G}) 
	\simeq {(\Sym^{\frac{j - i}{2}} \otimes \Alt^i)}(\bb{V}),
\end{align}
which yields \eqref{eq:covariant}.
Next, Theorems \ref{thm:lusztig-algebra} and \ref{thm:pure-iso} give a $W$-equivariant isomorphism
\begin{align}
\gr_0^{\bm{w}} \SHom^0(\cal{G}, \cal{G}\langle i\rangle) \simeq \QL[W] \otimes \Sym^{\frac{i}{2}}(\bb{V}),
\end{align}
which yields \eqref{eq:contravariant}.
Functoriality is a consequence of the fact that $\bb{A}_\cal{G} \xrightarrow{\sim} \bb{A}_\cal{S} \xrightarrow{\sim} \bb{A}_W$ is an isomorphism of algebras, not just of vector spaces.
\end{proof}

\begin{lem}\label{lem:h-to-g}
We have 
\begin{align}
(\gr_{i + j}^{\bm{w}} \bb{H}_G^j(G, -))^\vee
\simeq \Hom_W(\Alt^i(\bb{V}), \gr_0^{\bm{w}} {\SHom^0(-, \cal{G}\langle j\rangle)})
\end{align}
as contravariant functors $\underline{\sf{C}}(G_0)^\op \to \sf{Vect}_0$.
\end{lem}

\begin{proof}
By the preceding lemma and the self-duality of $\Sym^\ast(\bb{V})$ as a representation of $W$, the statement holds with $\underline{\sf{C}}_{\hat{W}}(G_0)$ in place of $\underline{\sf{C}}(G_0)$.
Now apply Theorem \ref{thm:summand}, observing that the constant sheaf is a summand of $\cal{G}$.
\end{proof}

\begin{lem}\label{lem:g-to-s}
We have
\begin{align}
\gr_0^{\bm{w}} {\SHom^0(-, \cal{G}\langle j\rangle)}
&\simeq \SHom^0(i^\ast(-)\langle-r\rangle, \cal{S}\langle j\rangle)
\end{align}
as contravariant functors $\underline{\sf{C}}(G_0)^\op \to \sf{Mod}_0(\QL[W])$.
\end{lem}

\begin{proof}
By Theorem \ref{thm:pure-iso}, the statement holds with $\underline{\sf{C}}_{\hat{W}}(G_0)$ in place of $\underline{\sf{C}}(G_0)$.
Now apply Theorem \ref{thm:summand}, observing that $i^\ast\langle -r\rangle$ is compatible with the functors $\sf{PR}$.
\end{proof}

\begin{proof}[Proof of Theorem \ref{thm:a_w-to-kr}]
The theorem claims that
\begin{align}
(\sf{HHH}^{i, i + j, k})^\vee
&\simeq \Hom_W(\Alt^i(\bb{V}), \sf{AH}^{j + k, j}(-))
\end{align}
as contravariant functors on $\sf{H}_W = \sf{K}^b(\sf{C}(\cal{B}_0 \times \cal{B}_0))$.
For ease of notation, let 
\begin{align}
\Phi^j &= \Hom_{\underline{\sf{C}}(G_0)}(i^\ast (-)\langle -r\rangle, \cal{S}\langle j\rangle),\\
\emph{resp.}\qquad
\Psi^{i, j} &= \Hom_{\sf{K}^b(\underline{\sf{C}}(G_0))}(i^\ast(-)\langle -r\rangle, \cal{S}\langle j\rangle [i - j]_\triangle),
\end{align}
as a functor $\underline{\sf{C}}(G_0) \to \sf{Mod}_0(\QL[W])$, \emph{resp.}\ $\sf{K}^b(\underline{\sf{C}}(G_0)) \to \sf{Mod}_0(\QL[W])$.
By Lemma \ref{lem:pullback-kappa}, we have $\kappa \circ i^\ast \langle -r\rangle \simeq i^\ast\langle -r\rangle \circ \kappa$ as functors on $\sf{C}(G_0)$, and therefore,
\begin{align}
\sf{AH}^{j + k, j} 
&\simeq
	\Psi^{j + k, j} \circ \sf{K}^b(\kappa \circ \sf{CH}).
\end{align}
At the same time, we have
\begin{align}
\Psi^{j + k, j}
&\simeq
	\ur{H}^k \circ \sf{K}^b\Phi^j.
\end{align}
Furthermore, by Lemmas \ref{lem:unipotent} and \ref{lem:g-to-s},
\begin{align}\begin{split}
\Phi^j
&\simeq
	\SHom^0(i^\ast(-)\langle -r\rangle, \cal{S}\langle j\rangle)
\simeq
	\gr_0^{\bm{w}} \SHom^0(-, \cal{G}\langle j\rangle).
\end{split}\end{align}
Combining these equivalences with Lemma \ref{lem:h-to-g}, we compute
\begin{align}\begin{split}
&\Hom_W(\Alt^i(\bb{V}), \sf{AH}^{j + k, j}(-))\\
&\simeq
	\Hom_W(\Alt^i(\bb{V}), \ur{H}^k \circ \sf{K}^b(\Phi^j \circ \kappa \circ \sf{CH}))
		\\
&\simeq
	\Hom_W(\Alt^i(\bb{V}), \ur{H}^k \circ \sf{K}^b(\gr_0^{\bm{w}} \SHom^0((\kappa \circ \sf{CH})(-), \cal{G}\langle j\rangle)))\\
&\simeq
	\ur{H}^k \circ \sf{K}^b(
	\Hom_W(\Alt^i(\bb{V}),\gr_0^{\bm{w}} \SHom^0((\kappa \circ \sf{CH})(-), \cal{G}\langle j\rangle)))\\
&\simeq
	\ur{H}^k \circ \sf{K}^b(
	(\gr_{i + j}^{\bm{w}} 
		\bb{H}_G^j(G, (\kappa \circ \sf{CH})(-))
	)^\vee).
\end{split}\end{align}
The last expression is equivalent to $(\sf{HHH}^{i, i + j, k})^\vee$ by Theorem \ref{thm:ww}.
\end{proof}

\subsection{}\label{subsec:kr-to-varieties}

Let $\cal{X}(\beta)$ be defined by the cartesian square:
\begin{equation}
\begin{tikzpicture}[baseline=(current bounding box.center), >=stealth]
\matrix(m)[matrix of math nodes, row sep=2em, column sep=2.5em, text height=2ex, text depth=0.5ex]
{	\cal{X}(\beta)
		&\cal{U}(\beta)\\
	\{1\}
		&\cal{U}\\
		};
\path[->,font=\scriptsize, auto]
(m-1-1)		edge (m-1-2)
(m-1-1)		edge (m-2-1)
(m-1-2)		edge node{$p_\beta$} (m-2-2)
(m-2-1)		edge (m-2-2);
\end{tikzpicture}
\end{equation}

\begin{proof}[Proof of Corollary \ref{cor:kr-to-varieties}]
We claim that for all $\beta \in \Br_W^+$, we have
\begin{align}
\sf{HHH}^{0, j, k}(\cal{R}(\beta))^\vee
	&\simeq \gr_{j + 2r}^{\bm{w}} \ur{H}_{-(j + k + 2r)}^{!, G}(\cal{U}(\beta)),\\
\sf{HHH}^{r, r + j, k}(\cal{R}(\beta))^\vee
	&\simeq \gr_{j + 2(r - N)}^{\bm{w}} \ur{H}_{-(j + k + 2(r - N))}^{!, G}(\cal{X}(\beta)).
\end{align}
Let $1$ and $\varepsilon$ be the trivial and sign characters of $W$, respectively.
Then $1$, \emph{resp.}\ $\varepsilon$, is the character of $\Alt^0(\bb{V})$, \emph{resp.}\ $\Alt^r(\bb{V})$, so Theorems \ref{thm:a_w-to-kr} and \ref{thm:a_w-to-varieties-strong} give us
\begin{align}
\sf{HHH}^{0, j, k}(\cal{R}(\beta))^\vee
	&\simeq \gr_{j + r}^{\bm{w}} \SHom^{j + k + r}(p_{\beta, !}\QL, \cal{S}_1)\\
\sf{HHH}^{r, r + j, k}(\cal{R}(\beta))^\vee
	&\simeq \gr_{j + r}^{\bm{w}} \SHom^{j + k + r}(p_{\beta, !}\QL, \cal{S}_\varepsilon).
\end{align}
By Springer theory, we know that:
\begin{itemize}
\item 	$\cal{S}_1$ is the sheaf $\QL\langle -r\rangle$ over $\cal{U}$.
\item 	$\cal{S}_\varepsilon$ is the skyscraper sheaf at $1 \in \cal{U}$ with stalk $\QL\langle -r - 2N\rangle$.
\end{itemize}
Now, both of the desired isomorphisms follow from Lemma \ref{lem:borel-moore}.
\end{proof}

\newpage
\section{The Decategorified Trace}\label{sec:decat}

\subsection{}

To state the main definition of this section, let $W$ denote an \emph{arbitrary} finite Coxeter group.

As before, $H_W$ is the Iwahori--Hecke algebra of $W$.
Moreover, we will draw freely upon the notation introduced in Appendix \ref{sec:coxeter}:
\begin{itemize}
\item 	$R(W)$ is the representation ring of $W$.
\item 	$\bb{Q}_W \subseteq \bar{\bb{Q}}$ is the splitting field of $W$.
\item 	$(-, -)_W : R(W) \times R(W) \to \bb{Z}$ be the multiplicity pairing.
\item 	We write 
\begin{align}
{[\Sym(\bb{V})]_{\Q}} = \sum_i {\Q^i} \Sym^i(\bb{V}).
\end{align}
\item 	If $\phi : W \to \bb{Q}_W$ is an irreducible character, then $\phi_{\Q} : H_W(\Q^{\frac{1}{2}}) \to \bb{Q}_W(\Q^{\frac{1}{2}})$ is the corresponding $\bb{Z}[\Q^{\pm\frac{1}{2}}]$-linear trace.
\item 	$\{-, -\} : \hat{W} \times \hat{W} \to \bb{Q}_W$ is the symmetric pairing (derived from) Lusztig's exotic Fourier transform.
\end{itemize}

\begin{df}
For any finite Coxeter group $W$, the \dfemph{$R(W)$ trace} on its Iwahori--Hecke algebra is the $\bb{Z}[\Q^{\pm\frac{1}{2}}]$-linear trace
\begin{align}
\Tr{-} : H_W \to R(W) \otimes \bb{Q}_W(\Q^{\frac{1}{2}})
\end{align}
defined by
\begin{align}
\Tr{\beta} = \sum_{\phi, \psi \in \hat{W}} {\{\phi, \psi\}} 
\phi_{\Q}(\beta)\psi
\end{align}
for all $\beta \in H_W$.
\end{df}

In the next two subsections, we assume that $W$ is still a Weyl group and prove Theorem \ref{thm:decat}, which relates the $R(W)$ trace to the decategorification of the $\bb{A}_W$ trace.
In the remainder of the section, we allow $W$ to be an arbitrary finite Coxeter group and relate the $R(W)$ trace to the Markov trace studied by Ocneanu and Gomi.
We conclude by proving Propositions \ref{prop:rationality} and \ref{prop:symmetry}, which roughly generalize rationality and symmetry properties of the Markov trace.

\subsection{}

Recall that an object of $\Mod_2^b(\QL[W])$ is a $\QL[W]$-module equipped with a $W$-invariant bigrading where the first grading is bounded below and the second is bounded.
We have an additive map
\begin{align}
\sf{AH} : H_W \xrightarrow{\sim} [\sf{H}_W]_\triangle \to [\Mod_2^b(\QL[W])]_\oplus
\end{align}
induced by the composition
\begin{align}
\sf{H}_W^\op \xrightarrow{\sf{AH}} \Mod_2^b(\bb{A}_W) \to \Mod_2^b(\QL[W]),
\end{align}
where the second arrow is the forgetful functor.
Note that $\Mod_2^b(\QL[W])$ is semisimple, since $\QL[W]$ is.

For any ring $R$, such as $R = R(W)$, we write $R(\!(\Q^{\frac{1}{2}})\!) = R[\![\Q^{\frac{1}{2}}]\!][\frac{1}{\Q}]$.
Let
\begin{align}
[-]_{\Q} : [\sf{Mod}_{1, 1}(\QL[W])]_\oplus \to R(W)(\!(\Q^{\frac{1}{2}})\!)
\end{align}
be the additive map defined by
\begin{align}
[M]_{\Q} = \sum_{i, j}
{(-1)^{i - j}} \Q^{\frac{j}{2}} M^{i, j}.
\end{align}
We say that $[M]_{\Q}$ is the \dfemph{$\Q$-graded character} of $M$.
For all $i, j$, let 
\begin{align}
(\sf{AH}^\vee)^{i, j}(-) = \Hom_{\QL}(\sf{AH}^{i, j}(-), \QL).
\end{align}
Then the composition
\begin{align}
[\sf{AH}^\vee]_{\Q} \vcentcolon= [-]_{\Q} \circ \sf{AH}^\vee
	: H_W \to R(W)(\!(\Q^{\frac{1}{2}})\!)
\end{align}
is $\bb{Z}[\Q^{\pm\frac{1}{2}}]$-linear.

\subsection{}

Recall from Section \ref{sec:hecke} that $\{\gamma_w\}_w$ denotes the Kazhdan--Lusztig basis of the Iwahori--Hecke algebra.
Then Theorem \ref{thm:decat} reduces to the claim that:
\begin{align}
[\sf{AH}^\vee(\IC_w)]_{\Q} = \Tr{\gamma_w} \cdot [\Sym(\bb{V})]_{\Q}
\end{align}
for all $w \in W$.
To prove it, we use a multiplicity formula for unipotent character sheaves, derived from \cite[Cor.\ 14.11]{lusztig_1985_3} and \cite[Thm.\ 23.1]{lusztig_1986}.

As in \S\ref{subsec:decat}, let $[\sf{C}(G_0)]_\oplus$ be endowed with the action of $\bb{Z}[\Q^{\frac{1}{2}}]$ in which $\Q^{\frac{1}{2}}$ acts by $\langle -1\rangle$.
The following version of the formula is the one quoted in \cite[413]{ww_2008}, fixed to incorporate semisimplification.

\begin{thm}[Lusztig]\label{thm:lusztig-character}
For all $w \in W$, we have
\begin{align}
{[\sf{CH}(\IC_w)^{ss}]} = \sum_{\phi, \psi \in \hat{W}}
{\{\phi, \psi\}} \phi_{\Q}(\gamma_w) [\cal{G}_\psi] + \text{cuspidal terms}
\end{align}
in $[\sf{C}(G_0)]_\oplus$, where ``cuspidal terms'' refers to the contribution of simple objects not of the form $\cal{G}_\psi$ for some $\psi \in \hat{W}$, and ``$ss$'' refers to the semisimplification of (shifted) perverse summands.
\end{thm}

\begin{proof}[Proof of Theorem \ref{thm:decat}]
By an argument similar to the proof of the second equivalence in Lemma \ref{lem:sheaves-to-reps}, Theorem  \ref{thm:lusztig-algebra} implies that for all $i, j \in \bb{Z}$, we have
\begin{align}
\gr_j^{\bm{w}} \SHom^i(\cal{S}_\psi, \cal{S}) \simeq 
\left\{\begin{array}{ll}
V_\psi \otimes \Sym^{\frac{i}{2}}(\bb{V})
	&i = j\\
0
	&i \neq j
\end{array}\right.
\end{align}
as representations of $W$.
So Theorem \ref{thm:lusztig-character} yields
\begin{align}\begin{split}
[\sf{AH}^\vee(\IC_w)]_{\Q}
&=	\sum_{i, j} {(-1)^{i - j}} 
	\Q^{\frac{j}{2}}
	\gr_j^{\bm{w}} 
	\SHom^i(i^\ast \sf{CH}(\IC_w) \langle -r\rangle, \cal{S})^\vee\\
&=	\sum_{i, j}
	{(-1)^{i - j}}
	\Q^{\frac{j}{2}}
	\sum_{\phi, \psi \in \hat{W}} 
	{\{\phi, \psi\}}
	\phi_{\Q}(\gamma_w)
	\gr_j^{\bm{w}} \SHom^i(\cal{S}_\psi, \cal{S})^\vee\\
&=	\sum_{\phi, \psi \in \hat{W}} 
	{\{\phi, \psi\}}
	\phi_{\Q}(\gamma_w)
	\sum_i
	{\Q^{\frac{i}{2}}}\,
	(\psi \otimes \Sym^{\frac{i}{2}}(\bb{V}))^\vee.
\end{split}\end{align}
Every representation of $W$ is self-dual because $\bb{Q}_W = \bb{Q} \subseteq \bb{R}$, so the inner sum in the last expression simplifies to $\psi \cdot [\Sym(\bb{V})]_{\Q}$.
\end{proof}

\subsection{}\label{subsec:markov}

We now relate $\Tr{-}$ to Markov traces.

The \dfemph{HOMFLY series} of a link $\lambda \subseteq \bb{R}^3$ is one of the predecessors of the isotopy invariant mentioned in Remark \ref{rem:dgr}.
We will denote it by
\begin{align}
[\lambda]_{a, \Q} \in \bb{Z}[(\Q^{\frac{1}{2}} - \Q^{-\frac{1}{2}})^{\pm 1}, a^{\pm 1}].
\end{align}
We have $[\lambda]_{a, \Q} = \bb{P}(\lambda)|_{t = -1}$, where $\bb{P}$ is the invariant in the remark.
With this normalization, the HOMFLY series of the unknot equals $1$.

There are many different constructions of the HOMFLY invariant \cite{homfly}.
Ocneanu's approach relies on Alexander's theorem, stating that every link in $\bb{R}^3$ is the closure of a braid.
Namely, one constructs for each $n \geq 1$ a certain function on the Iwahori--Hecke algebra of $S_n$, now known as a Markov trace \cite{jones}.
If $\Br_n$ denotes the braid group of $S_n$, \emph{i.e.}, the group of topological braids on $n$ strands, and $\lambda$ is the closure of $\beta \in \Br_n$, then the HOMFLY series of $\lambda$ is the Markov trace of $\beta$, multiplied by a factor depending only on $|\beta|$ and $n$.

Y.~Gomi generalized Ocneanu's Markov traces in a uniform way beyond type $A$, extending work of Geck--Lambropoulou in type $BC$ \cite{gomi}.
We follow Gomi's conventions here.

Let $(W, S)$ be any finite Coxeter system.
Recall from \S\ref{subsec:parabolic} that if $S' \subseteq S$ and $W' \subseteq W$ is the $S$-parabolic subgroup generated by $S'$, then $H_{W'} \subseteq H_W$.
A \dfemph{Markov trace} on $H_W$ is a $\bb{Z}[\Q^{\pm\frac{1}{2}}]$-linear function
\begin{align}
\mathsf{tr} : H_W \to \bb{Z}[\Q^{\pm\frac{1}{2}}](a)
\end{align}
such that:
\begin{enumerate}
\item 	$\mathsf{tr}(1) = 1$.
\item 	For all $\beta, \gamma \in H_W$, we have $\mathsf{tr}(\beta\gamma) = \mathsf{tr}(\gamma\beta)$.
\item 	For all $s \in S$ and $\gamma \in H_{W'}$, where $W'$ is the $S$-parabolic subgroup of $W$ generated by $S \setminus s$, we have
		\begin{align}
		\mathsf{tr}(\sigma_s^{\pm 1}\gamma) 
		= -a^{\mp 1}
			\pa{\frac{\Q^{\frac{1}{2}} - \Q^{-\frac{1}{2}}}{a - a^{-1}}} 
			\mathsf{tr}(\gamma).
		\end{align}
\end{enumerate}
Axioms (2) and (3) are respectively known as the first and second Markov moves.

\begin{thm}[Ocneanu]\label{thm:ocneanu}
The axioms above uniquely define a Markov trace on the Iwahori--Hecke algebra of $S_n$ for all $n \geq 1$.
For all $\beta \in \Br_n$, we have
\begin{align}
[\hat{\beta}]_{a, \Q}
=	(-a)^{|\beta|}
	\pa{\frac{a - a^{-1}}{\Q^{\frac{1}{2}} - \Q^{-\frac{1}{2}}}}^{n - 1} 
	\mathsf{tr}(\beta),
\end{align}
where $\hat{\beta}$ is the link closure of $\beta$.
\end{thm}

Suppose that a Markov trace on $H_W$ exists.
After base change from $\bb{Z}[\Q^{\pm\frac{1}{2}}]$ to $\bb{Q}_W(\Q^{\frac{1}{2}})$, we can view it as a trace on $H_W(\Q^{\frac{1}{2}})$ (see \S\ref{subsec:hecke-generic}).
Then
\begin{align}
\mathsf{tr} = \sum_{\phi \in \hat{W}} 
\mathsf{tr}_\phi {\phi_{\Q}}
\end{align}
for some weights $\mathsf{tr}_\phi \in \bb{Q}_W(\Q^{\frac{1}{2}}, a)$.

Axioms (1)-(3) show that $\mathsf{tr}|_{a^{-2} = 0}$ is the symmetrizing trace $\tau_{\Q} : H_W(\Q^{\frac{1}{2}}) \to \bb{Z}[\Q^{\pm\frac{1}{2}}]$.
Proposition \ref{prop:molien} then suggests that the weights $\mathsf{tr}_\phi$ should involve bivariate Molien series.
These observations motivate the form of Gomi's Markov weights.
Below, let $\bb{V}$ be a realization of $W$ such that $\bb{V}^W = 0$, and for all $\psi \in \hat{W}$, let
\begin{align}\begin{split}
\bm{m}_\psi(\Q, x)
&=	\sum_{i, j}
	{(-x)^i} \Q^j
	(\psi, \Alt^i(\bb{V}) \otimes \Sym^j(\bb{V}))_W.
\end{split}\end{align}
By construction, $\bm{m}_\psi(\Q, 0) = \bm{m}_\psi(\Q)$, the usual Molien series of $\psi$.

\begin{thm}[Gomi]\label{thm:gomi}
The weights
\begin{align}\begin{split}
\mathsf{tr}_\phi
&=	\pa{\frac{1 - \Q}{1 - a^{-2}}}^{r}
	\sum_{\psi \in \hat{W}}
	{\{\phi, \psi\}}
	\bm{m}_\psi(\Q, a^{-2})
\end{split}\end{align}
define a Markov trace on $H_W$ for any finite Coxeter group $W$ of rank $r$.
\end{thm}

This is a bivariate refinement of Proposition \ref{prop:molien}.

Henceforth, $\mathsf{tr}$ denotes Gomi's Markov trace.
To relate it to the $R(W)$ trace, we introduce the normalization
\begin{align}
\Tr{\beta}^0 = {(-\Q^{\frac{1}{2}})^{|\beta|}} \Tr{\beta} \cdot \varepsilon [\Sym(\bb{V})]_{\Q},
\end{align}
where $\varepsilon$ denotes the sign character of $W$ (see \S\ref{subsec:representation}).
We emphasize that whenever we use $\Tr{-}^0$, \emph{we always assume that $\bb{V}^W = 0$}.
For Weyl groups, the result below follows from Theorems \ref{thm:ww} and \ref{thm:decat}.

\begin{prop}\label{prop:markov}
Under the hypotheses of Theorem \ref{thm:gomi}, we have
\begin{align}
\mathsf{tr}(\beta)
=	(-\Q^{-\frac{1}{2}})^{|\beta|}
	\pa{\frac{1 - \Q}{1 - a^2}}^r
	\sum_i
	{(-a^2)^i}
	(\Alt^i(\bb{V}), \Tr{\beta}^0)_W
\end{align}
for all $\beta \in \Br_W$.
\end{prop}

\begin{lem}\label{lem:self-dual}
We have
\begin{align}
\frac{1}{(1 - x^{-1})^r}
\sum_i {(-x^{-1})^i} \Alt^i(\bb{V})
=	
\frac{\varepsilon}{(1 - x)^r}
\sum_i {(-x)^i} \Alt^i(\bb{V})
\end{align}
in $R(W)(x)$.
\end{lem}

\begin{proof}
Use the identity $\Alt^{r - i}(\bb{V}) = \varepsilon \Alt^i(\bb{V})$.
\end{proof}

\begin{proof}[Proof of Proposition \ref{prop:markov}]
By Lemma \ref{lem:self-dual} and the self-duality of $\Sym(\bb{V})$,
\begin{align}
\mathsf{tr}_\phi
=	\pa{\frac{1 - \Q}{1 - a^2}}^r 
\sum_{\psi \in \hat{W}}
{\{\phi, \psi\}}
\sum_{i, j} {(-a^2)^i} \Q^j
(\Alt^i(\bb{V}), \varepsilon \psi \cdot \Sym^j(\bb{V}))_W.
\end{align}
Comparing with the terms that make up $\Tr{-}^0$, we get the result.
\end{proof}

\begin{cor}
Under the hypotheses of Theorem \ref{thm:gomi}, we have
\begin{align}
\tau_{\Q} = (-\Q^{-\frac{1}{2}})^{|\beta|} (1 - \Q)^r (\varepsilon, \Tr{\beta}^0)_W
\end{align}
as $\bb{Z}[\Q^{\pm \frac{1}{2}}]$-linear traces on $H_W$.
\end{cor}

\begin{cor}\label{cor:homfly}
Under the hypotheses of Theorem \ref{thm:ocneanu}, we have
\begin{align}
{[\hat{\beta}]_{a, \Q}}
=	(a\Q^{-\frac{1}{2}})^{|\beta| - n + 1}
	\sum_{0 \leq i \leq n - 1}
	{(-a^2)^i}
	(\Alt^i(\bb{V}), \Tr{\beta}^0)_W
\end{align}
for all $\beta \in \Br_n$.
\end{cor}

\subsection{}

To conclude this section, we prove:
\begin{enumerate}
\item 	Proposition \ref{prop:rationality}, a rationality result that will simplify the proof of Theorem \ref{thm:pole} in Section \ref{sec:dl}.
\item 	Proposition \ref{prop:symmetry}, a symmetry result.
		Together with (1) and Theorem \ref{thm:pole}, it implies Theorem \ref{thm:symmetry}, as we will also explain.
\end{enumerate}
These results about the $R(W)$ trace respectively generalize, and are inspired by, the following facts about the HOMFLY series:
\begin{enumerate}
\item[(1$'$)] 	HOMFLY has integer coefficients.
		Moreover, it is a $\Q^{\frac{1}{2}}$-monomial shift of an element of $\bb{Z}[\![\Q]\!][a^{\pm 1}]$, not just of $\bb{Z}[\![\Q^{\frac{1}{2}}]\!][a^{\pm 1}]$.
\item[(2$'$)]	HOMFLY is invariant under $\Q^{\frac{1}{2}} \mapsto -\Q^{-\frac{1}{2}}$.
\end{enumerate}
Note that the second claim of (1$'$) can be proved via the Markov trace, while (2$'$) is a consequence of the HOMFLY skein relation (see \cite[\S{2.4}]{dgr}, noting that their $q$ is our $\Q^{\frac{1}{2}}$).
By contrast, the proofs below are representation-theoretic.
Throughout, we assume that the realization $\bb{V}$ satisfies $\bb{V}^W = 0$.

\subsubsection{}

Let $d_1, \ldots, d_r$ be the invariant degrees of the $W$-action on $\bb{V}$.
Let:
\begin{align}\label{eq:exception}
\Q_0 = 
\left\{\begin{array}{ll}
\Q^{\frac{1}{2}}
	&\text{$W$ has a factor of type $E_7$ or $E_8$}\\
\Q
	&\text{else}
	\end{array}\right.
\end{align}
The reason we make an exception for $E_7$ and $E_8$ is the existence of so-called exceptional characters in these types.
See \cite[\S{9.5}]{gp} and the references therein.

\begin{prop}\label{prop:rationality}
For all $\beta \in \Br_W$, we have
\begin{align}
\Tr{\beta} &\in (\Q^{-\frac{1}{2}})^{|\beta|} R(W)[\Q_0],\\
\Tr{\beta}^0 &\in \frac{1}{(1 - \Q^{d_1}) \cdots (1 - \Q^{d_r})}\, R(W)[\Q_0],\\
(1, \Tr{\beta}^0)_W &\in \frac{1}{(1 - \Q^{d_1}) \cdots (1 - \Q^{d_r})}\, (1 + \Q_0\bb{Z}[\Q_0]).
\end{align}
\end{prop}

\begin{proof}
Let $\tilde{\beta} = \Q^{\frac{|\beta|}{2}} \beta$.
To prove the first statement, we must show that
\begin{align}
(\psi, \Tr{\tilde{\beta}})_W
=	\sum_{\phi \in \hat{W}}
	{\{\phi, \psi\}}
	\phi_{\Q}(\tilde{\beta}) \in \bb{Q}_W(\Q^{\frac{1}{2}})
\end{align}
actually belongs to $\bb{Z}[\Q_0]$ for all $\psi \in \hat{W}$.
We make a series of reductions:

By Lemma-Postulate \ref{lempost:fourier}(1), we can assume that $W$ is irreducible.

We can expand $\tilde{\beta}$ as a $\bb{Z}[\Q]$-linear combination of the elements $\tilde{\sigma}_w = \Q^{\frac{|\sigma_w|}{2}} \sigma_w$ for $w \in W$.
So we can assume that $\tilde{\beta} = \tilde{\sigma}_w$ for some $w$.

We can now check the non-crystallographic types case by case, using \cite[\S{8.3}]{gp} and \cite{lusztig_1994} for the dihedral types, \cite[Table 9.1]{gp} and \cite[\S{7}]{bm_1993} for type $H_3$ (but see Remark \ref{rem:h3}, below), and \cite{al, malle} for type $H_4$.
So we can assume that $W$ is crystallographic.

Theorem \ref{thm:decat} implies that $(\psi, \Tr{\tilde{\beta}})_W \in \bb{Z}[\Q^{\pm \frac{1}{2}}]$.
So it remains to show that $(\psi, \Tr{\tilde{\beta}})_W \in \bb{Q}[\Q_0]$.
Since $\{-, -\}$ takes values in $\bb{Q}_W = \bb{Q}$, it is enough to get $\phi_{\Q}(\tilde{\sigma}_w) \in \bb{Q}[\Q_0]$ for all $\phi$.
For crystallographic irreducible $W$, this is a theorem of Benson--Curtis \cite{bc} (joint with Springer in the case of $E_7$ and $E_8$).

The remaining two statements follow from the first statement, once we establish the following two lemmas.
\end{proof}

\begin{lem}
We have 
\begin{align}
{[\Sym(\bb{V})]_{\Q}} \in \frac{1}{(1 - \Q^{d_1}) \cdots (1 - \Q^{d_r})}\, R(W)[\Q].
\end{align}
\end{lem}

\begin{proof}
In the notation of \S\ref{subsec:realization} and \S\ref{subsec:degrees}, we have 
\begin{align}
{[\Sym(\bb{V})]_{\Q}} 
= \bm{m}_1(\Q)\sum_{\phi \in \hat{W}} {\FDeg_\phi(\Q)\phi},
\end{align}
where $\bm{m}_1(\Q) = \prod_i {(1 - \Q^{d_i})^{-1}}$ and $\FDeg_\phi(\Q) \in \bb{Z}[\Q]$ for each $\phi$.
\end{proof}

\begin{lem}
For all $\beta \in \Br_W$, we have $(\varepsilon, \Tr{\beta})_W = \varepsilon_{\Q}(\beta) = (-\Q^{-\frac{1}{2}})^{|\beta|}$.
\end{lem}

\begin{proof}
By \cite[252]{gp}, we have $\varepsilon_{\Q}(\sigma_s) = -\Q^{-\frac{1}{2}}$ for all $s \in S$.
(Note that $T_s = \Q^{\frac{1}{2}}\sigma_s$ in the notation of \emph{ibid.}.)
But $\varepsilon_{\Q}$ is the character of a $1$-dimensional $\bb{Q}_W(\Q^{\frac{1}{2}})$-vector space, so it is multiplicative on $\Br_W$.
We deduce that $\varepsilon_{\Q}(\beta) = (-\Q^{-\frac{1}{2}})^{|\beta|}$ for all $\beta \in \Br_W$.

As for the other equality:
By \cite[Ch.~4]{lusztig_1984}, we have $\{\phi, \varepsilon\} = 0$ for all $\phi \neq \varepsilon$.
This implies $(\varepsilon, \Tr{\beta})_W = \varepsilon_{\Q}(\beta)$.
\end{proof}

\begin{rem}\label{rem:h3}
The proof of Proposition \ref{prop:rationality} uses the explicit description of the exotic Fourier transform in type $H_3$.
For this group, there are three nontrivial families (see \S\ref{subsec:families}), one of which is exceptional.
The $4 \times 4$ matrix in \cite[\S{7}]{bm_1993} represents the block of the Fourier matrix for the non-exceptional families.
I learned from George Lusztig that by contrast, the data for the exceptional family matches that of the exceptional family in type $E_7$.
The assertion in \cite[189]{gm} that all three families have the same Fourier data is incorrect.
\end{rem}

\subsubsection{}

Turning to the symmetry result:

\begin{prop}\label{prop:symmetry}
For all $\beta \in \Br_W$, we have
\begin{align}
\Tr{\beta} 
&\in R(W)[\Q^{\frac{1}{2}} - \varepsilon \Q^{-\frac{1}{2}}],\\
\Tr{\beta}^0
&\in 
	(\Q^{\frac{1}{2}})^{|\beta| - r}
	R(W)[(\Q^{\frac{1}{2}} - \Q^{-\frac{1}{2}})^{\pm 1}].
\end{align}
\end{prop}

By Proposition \ref{prop:rationality}, $\Tr{-}$ takes values in $R(W)[\Q^{\pm\frac{1}{2}}]$, so the proof amounts to the following lemmas:

\begin{lem}\label{lem:symmetry}
For all $\beta \in \Br_W$, we have $\Tr{\beta} \mapsto \varepsilon \Tr{\beta}$ under $\Q^{\frac{1}{2}} \mapsto -\Q^{-\frac{1}{2}}$.
\end{lem}

\begin{proof}
Recall that for all $\beta \in H_W$, we have $(\psi, \Tr{\beta})_W = \sum_{\phi \in \hat{W}} {\{\phi, \psi\}} \phi_{\Q}(\beta)$.
Using Lemma-Postulate \ref{lempost:fourier}(3), we compute
\begin{align}\begin{split}
(\psi, \varepsilon \Tr{\beta})_W 
&= (\varepsilon\psi, \Tr{\beta})_W \\
&= \sum_\phi
	{\{\phi, \varepsilon \psi\}}
	\phi_{\Q}(\beta)\\
&=	\sum_\phi
	{\{\varepsilon \phi, \varepsilon \psi\}}
	(\varepsilon \phi)_{\Q}(\beta)\\
&=	\sum_\phi
	{\{\phi, \psi\}}
	(\varepsilon \phi)_{\Q}(\beta).
\end{split}\end{align}
It remains to show that if $\beta$ is in the image of $\Br_W$, then 
\begin{align}
\phi_{\Q}(\beta) \mapsto (\varepsilon \phi)_{\Q}(\beta)
\quad\text{under}\quad
\Q^{\frac{1}{2}} \mapsto -\Q^{-\frac{1}{2}}.
\end{align}
This follows implicitly from the proof of \cite[Prop.\ 9.4.1(a)]{gp}.
\end{proof}

\begin{lem}
We have $\Q^{\frac{r}{2}} [\Sym(\bb{V})]_{\Q} \mapsto \varepsilon \Q^{\frac{r}{2}} [\Sym(\bb{V})]_{\Q}$ under $\Q^{\frac{1}{2}} \mapsto -\Q^{-\frac{1}{2}}$.
\end{lem}

\begin{proof}
Observe that $\sum_i {(-\Q)^i} \Alt^i(\bb{V})$ is the multiplicative inverse of $[\Sym(\bb{V})]_{\Q}$.
So the claim follows from Lemma \ref{lem:self-dual} with $x = \Q$.
\end{proof}

\subsubsection{}

We show that Proposition \ref{prop:rationality}, Proposition \ref{prop:symmetry}, and Theorem \ref{thm:pole} together imply Theorem \ref{thm:symmetry}.

\begin{proof}[Proof of Theorem \ref{thm:symmetry}]
Proposition \ref{prop:rationality} shows that if $W$ has no factors of type $E_7$ or $E_8$, then $\Tr{\beta}^0$ and $(1, \Tr{\beta}^0)_W$ are indeed rational functions in $\Q$, not just power series in $\Q^{\frac{1}{2}}$.
Theorem \ref{thm:pole} extends this result to all $W$.
These results, combined with Proposition \ref{prop:symmetry}, establish
\begin{align}
\Tr{\beta} 
&\in \Q^{-\frac{|\beta|}{2}}R(W)[\Q] \cap R(W)[\Q^{\frac{1}{2}} - \varepsilon \Q^{-\frac{1}{2}}],\\
\Tr{\beta}^0
&\in 
	R(W)[\![\Q]\!] \cap R(W)(\Q) \cap
	(\Q^{\frac{1}{2}})^{|\beta| - r}
	R(W)[(\Q^{\frac{1}{2}} - \Q^{-\frac{1}{2}})^{\pm 1}].
\end{align}
In particular, the $\Q$-degrees of $\Tr{\beta}^0$ and $(1, \Tr{\beta}^0)_W$ are well-defined.
The following observations show that their degrees must be $|\beta| - r$:
\begin{itemize}
\item 	Proposition \ref{prop:rationality} shows that $\Tr{\beta}^0$, and hence $(1, \Tr{\beta}^0)_W$, belong to $\bb{Z}[\![\Q]\!]$, \emph{i.e.}, their lowest-order term in $\Q$ is the constant term.
\item 	By Proposition \ref{prop:symmetry}, we have
\begin{align}
\Tr{\beta}^0 \mapsto (-\Q)^{r - |\beta|}\Tr{\beta}^0
\quad\text{under}\quad
\Q_0 \mapsto \Q_0^{-1}.
\end{align}
\item 	Proposition \ref{prop:rationality} also shows that $(1, \Tr{\beta}^0)_W$, and hence $\Tr{\beta}^0$, have nonzero constant term.\qedhere
\end{itemize}
\end{proof}

\newpage
\section{Deligne--Lusztig Theory}\label{sec:dl}

\subsection{}

The goal of this section is to prove Theorems \ref{thm:virtual}, \ref{thm:induction}, and \ref{thm:pole}.
Throughout, we assume that $W$ is a Weyl group and use the notation and hypotheses of Sections \ref{sec:mixed}-\ref{sec:hecke}.
In particular, we continue to assume that $\bb{F}$ is of large characteristic.
We will write $\ur{H}_!^\ast(-) = \ur{H}_!^\ast(-, \QL)$ to denote cohomology with compact support.

\subsection{}

The proofs of Theorems \ref{thm:virtual} and \ref{thm:induction} rely on properties of the virtual characters of $G^F$ introduced by Deligne--Lusztig \cite{dl}, which we now review.

For all $w \in W$, the \dfemph{Deligne--Lusztig variety} $X_w \subseteq \cal{B} \times \cal{B}$ is the transverse intersection of $O_w$ with the graph of the Frobenius map $F : \cal{B} \to \cal{B}$.
The $G$-action on $\cal{B} \times \cal{B}$ restricts to a $G^F$-action on $X_w$, so the formal sum
\begin{align}
R_w = \sum_n
{(-1)^n}\ur{H}_!^n(X_w)
\end{align}
defines a virtual representation of $G^F$.
In general, a virtual representation of $G^F$ is \dfemph{unipotent} iff it occurs with nonzero multiplicity in $R_w$ for some $w$.

We have $X_1 = \cal{B}^F$, a scheme of finitely many points, and $R_1 = \QL[\cal{B}^F]$, the $\QL$-vector space of functions on $\cal{B}^F$.
For all $B \in \cal{B}^F$, let $\bb{1}_B$ be the indicator function on $\{B\}$.
Recall, \emph{cf.}\ Remark \ref{rem:iwahori}, that Iwahori constructed an isomorphism
\begin{align}
\QL \otimes H_W|_{\Q^{1/2} = q^{1/2}} \xrightarrow{\sim} \End_{G^F}(R_1).
\end{align}
In terms of the functions $\bb{1}_B$, the action of $\sigma_w \in H_W$ on $R_1$ is given by
\begin{align}\label{eq:hecke-operator}
\tilde{\sigma}_w \cdot \bb{1}_B = \sum_{(B, B') \in O_w^F}
\bb{1}_{B'},
\end{align}
where $\tilde{\sigma}_w = q^{\frac{|w|}{2}} \sigma_w$.

Irreducible representations of $W$ give rise to unipotent virtual representations of $G^F$ in two ways:

\subsubsection{}
For any $\phi \in \hat{W}$, the formal $\bb{Q}$-linear combination
\begin{align}\label{eq:almost}
R_\phi = \frac{1}{|W|} \sum_{w \in W} \phi(w)R_w
\end{align}
has nonzero multiplicity in $R_1$ (where we mean $1 \in W$, not $1 \in \hat{W}$), so it is unipotent.
The characters of the virtual representations $R_\phi$ are called \dfemph{almost-characters}, being very close to irreducible characters: 
See Lemma-Postulate \ref{lempost:almost}, below.

\subsubsection{}

For any $\phi \in \hat{W}$, the trace $\phi_{\Q} : \bb{Q}(\Q^{\frac{1}{2}}) \otimes H_W \to \bb{Q}(\Q^{\frac{1}{2}})$ (see \S\ref{subsec:hecke-generic}) restricts to a function $H_W \to  \bb{Z}[\Q^{\pm\frac{1}{2}}]$ \cite[Thm.\ 9.3.5]{gp}.
So we can set
\begin{align}
\phi_q = \phi_{\Q}|_{\Q^{1/2} = q^{1/2}} : H_W \to \bb{Z}[q^{\pm\frac{1}{2}}].
\end{align}
Then the double centralizer theorem gives an isomorphism of $(H_W \times G^F)$-modules
\begin{align}\label{eq:double-centralizer}
R_1 \simeq \bigoplus_{\phi \in \hat{W}}
V_{\phi, q} \otimes \rho_\phi,
\end{align}
where $V_{\phi, q}$, \emph{resp.}\ $\rho_\phi$, is a simple $H_W$-module of character $\phi_q$, \emph{resp.}\ a unipotent irreducible representation of $G^F$.
The representations $\rho_\phi$ are called the \dfemph{unipotent principal series}.

If $\rho$ is a representation of $G^F$, then $\Hom_{G^F}(R_1, \rho)$ forms an $\End_{G^F}(R_1)$-module under precomposition.
In particular, applying $\Hom_{G^F}(R_1, -)$ to both sides of \eqref{eq:double-centralizer} yields an isomorphism of $H_W$-bimodules:
\begin{align}
\End_{G^F}(R_1) \simeq \bigoplus_{\phi \in \hat{W}}
V_{\phi, q} \otimes \Hom_{G^F}(R_1, \rho_\phi).
\end{align}
By Artin--Wedderburn, we deduce:

\begin{lem}\label{lem:restriction}
For all $\phi \in \hat{W}$, we have $\Hom_{G^F}(R_1, \rho_\phi) \simeq V_{\phi, q}$ as $\End_{G^F}(R_1)$-modules, hence as $H_W$-modules.
\end{lem}

The map $\rho \mapsto \Hom_{G^F}(R_1, \rho)$ extends uniquely to an additive map from virtual representations to virtual modules.
We emphasize that \emph{$\Hom_{G^F}(R_1, \rho)$ may still be nonzero even when its virtual dimension $(R_1, \rho)_{G^F}$ is zero}.

\subsection{}

Lusztig showed that the unipotent irreducible representations of $G^F$ can be indexed by a set that only depends on $W$ \cite[Ch.\ 4]{lusztig_1984}.
This is the set denoted $\UCh(W)$ in \S\ref{subsec:fourier}.
Abusing notation, we will conflate the elements of $\UCh(W)$ with the unipotent irreducible representations themselves.
Then the embedding $\hat{W} \to \UCh(W)$ is precisely the map $\phi \mapsto \rho_\phi$.
Cor.\ 4.24-4.25 of \cite{lusztig_1984} entail:

\begin{lempost}\label{lempost:almost}
There is a function $\Delta : \UCh(W) \to \{\pm 1\}$ such that  for all $\phi \in \hat{W}$, we have:
\begin{enumerate}
\item 	$\Delta(\rho_\phi) = 1$.
\item 	The transformation rule
\begin{align}
R_\phi
=	\sum_{\rho \in \UCh(W)}
{\{\rho_\phi, \rho\}}\Delta(\rho)\rho.
\end{align}
Equivalently, $(\Delta(\rho)\rho, R_\phi) = \{\rho_\phi, \rho\}$  for all $\rho \in \UCh(W)$.
\end{enumerate}
\end{lempost}

\begin{rem}
We have $R_\phi(1) = \FDeg_\phi(q)$ and $\rho_\phi(1) = \UDeg_\phi(q)$ for all $\phi \in \hat{W}$.
\end{rem}

\subsection{}

Springer conjectured \cite{springer_1976}, and Kazhdan proved \cite{kazhdan}, a remarkable identity relating the character values of the virtual representations $R_w$ to the action of $W$ on the fibers of the Grothendieck--Springer map.
To state it, let $\cal{B}_g$ denote the fiber of $\tilde{G} \to G$ above $g \in G^F$, and let
\begin{align}
Q_g(w) = \sum_n
{\tr(wF \mid \ur{H}^{2n}(\cal{B}_g))}.
\end{align}
(By \cite{dlp}, $\cal{B}_g$ has no odd cohomology.)
The main result of \cite{kazhdan} implies:

\begin{thm}[Kazhdan]\label{thm:kazhdan}
For all $g \in G^F$ and $w \in W$, we have $Q_g(w) = R_w(g)$.
\end{thm}

\begin{rem}
Lusztig generalized the functions $Q_g$ to bad characteristic in \cite{lusztig_1981}, and similarly generalized Kazhdan's theorem in \cite{lusztig_1990}.
\end{rem}

\subsection{}

In this subsection, we state our arguments in such a way that \emph{we do not need to assume $G$ is semisimple}.
The following result recovers Theorem \ref{thm:virtual} in the semisimple case.

\begin{thm}\label{thm:virtual-reductive}
For all $\beta \in \Br_W^+$, we have
\begin{align}
(\Tr{\tilde{\beta}} \cdot \varepsilon[\Sym(\bb{V})]_{\Q})|_{\Q^{1/2} = q^{1/2}}
&= \frac{(-1)^{r}}{|G^F|}
	\sum_{u \in \cal{U}^F}
	|\cal{U}(\beta)_u^F| Q_u
\end{align}
in $\bb{Q}(q^{\frac{1}{2}}) \otimes R(W)$, where $\tilde{\beta} = q^{\frac{|\beta|}{2}} \beta$.
\end{thm}

\begin{rem}\label{rem:special}
An algebraic group $H$ is \dfemph{special} iff every principal $H$-bundle in the \'etale topology is Zariski-locally trivial \cite{grothendieck_1958}.
If $G$ is special, then \cite[Appendix]{hr} and \cite[Thm.\ 4.10]{joyce} together imply that the right-hand side of the identity is the virtual weight series of $[\cal{Z}(\beta)/G]$.
In this case, the identity is essentially a corollary of Theorem \ref{thm:decat}.
The sign character appears when we use Verdier duality to match the dual of Borel--Moore homology with compactly-supported cohomology, \emph{cf.}\ \cite[Lem.\ 4.5]{ahjr}.

However, not all reductive algebraic groups are special.
If $X_0$ is a $G_0$-variety and $G$ is not special, then the virtual weight series of $[X/G]$ need not be the quotient of the virtual weight polynomial of $X$ by that of $G$:
See \cite[Ex.\ 4.9]{joyce}.
The force of Theorem \ref{thm:virtual} is that for $X_0 = \cal{Z}(\beta)_0$, that property holds anyway.
\end{rem}

The following orthogonality result for the characters $Q_u$ will reduce our work to computing the multiplicities $(Q_u, \Tr{\tilde{\beta}})_W$.
We are using the version quoted by Shoji in \cite[510]{shoji}, in turn based on \cite[Thm.\ 5.6]{springer_1976}.
The counterpart for the characters $R_w$ is \cite[Thm.\ 6.9]{dl}, which holds in any characteristic.

\begin{thm}[Deligne--Lusztig, Springer]\label{thm:orthogonality}
For all $\phi \in R(W)$, we have
\begin{align}
q^N \phi \cdot \varepsilon Q_1 
= 
\sum_{u \in \cal{U}^F}
{(Q_u, \phi)_W} Q_u
\end{align}
in $R(W)$.
\end{thm}

Before we can use Theorem \ref{thm:orthogonality}, we must rewrite the left-hand side of the orthogonality identity:

\begin{lem}\label{lem:q1}
We have $[\Sym(\bb{V})]_{\Q} = \frac{(-1)^r}{|G^F|} \Q^N Q_1$.
\end{lem}

\begin{proof}
We have $\cal{B}_1 = \cal{B}$, so the Borel isomorphism states that $Q_1$ is the algebra of $W$-coinvariants of $\bb{V}$, graded by $\Q$.
That is,
\begin{align}
Q_1 = [\Sym(\bb{V})]_{\Q} \cdot \prod_{1 \leq i \leq r} {(1 - \Q^{d_i})},
\end{align}
where $d_1, \ldots, d_r$ are the invariant degrees of the $W$-action on $\bb{V}$ \cite[149]{gp}.
(If $G$ is not semisimple, then we might have $d_i = 1$ for one or more $i$.)
At the same time, if $B_0 \subseteq G_0$ is a fixed Borel, then
\begin{align}
|G^F| = |B^F||\cal{B}^F| = \Q^N (\Q - 1)^r \sum_{w \in W} \Q^{|w|}.
\end{align}
By the Bott--Solomon formula \cite{solomon}, $\sum_{w \in W} \Q^{|w|} = \prod_{i : d_i \neq 1} \frac{\Q^{d_i} - 1}{\Q - 1} = \prod_i \frac{\Q^{d_i} - 1}{\Q - 1}$.
\end{proof}

Combining Theorem \ref{thm:orthogonality} and Lemma \ref{lem:q1}, it remains to prove the following multiplicity formula:

\begin{prop}\label{prop:multiplicity} 
For any $G$-orbit $C \subseteq \cal{U}$ and $\beta \in \Br_W^+$, we have
\begin{align}
\sum_{u \in C^F}
{(Q_u, \Tr{\tilde{\beta}})_W|_{\Q^{1/2} = q^{1/2}}}
= 
\sum_{u \in C^F}
|\cal{U}(\beta)_u^F|,
\end{align}
where $\tilde{\beta} = q^{\frac{|\beta|}{2}} \beta$.
\end{prop}

\begin{lem}\label{lem:reciprocity}
For all $\beta \in H_W$ and $w \in W$, we have 
\begin{align}
\Tr{\beta}(w)|_{\Q^{1/2} = q^{1/2}} = \tr(\beta \mid \Hom_{G^F}(R_1, R_w)).
\end{align}
\end{lem}

\begin{proof}
By Lemma \ref{lem:restriction} and Lemma-Postulate \ref{lempost:almost}, the left-hand side equals
\begin{align}
\sum_{\phi, \psi \in \hat{W}}
	(\rho_\phi, R_\psi)_{G^F} \tr(\beta \mid \Hom_{G^F}(R_1, \rho_\phi)) \psi(w).
\end{align}
If $\rho \in \UCh(W)$ does not take the form $\rho_\phi$ for some $\phi \in \hat{W}$, then $\Hom_{G^F}(R_1, \rho) = 0$.
So above, the sum over $\phi \in \hat{W}$ can be extended to a sum over $\rho \in \UCh(W)$, giving
\begin{align}\begin{split}
&\sum_{\psi \in \hat{W}} \sum_{\rho \in \UCh(W)}
(\rho, R_\psi)_{G^F} 
\tr(\beta \mid \Hom_{G^F}(R_1, \rho)) \psi(w)\\
&=	\sum_{\psi \in \hat{W}}
	\tr(\beta \mid \Hom_{G^F}(R_1, R_\psi)) \psi(w).
\end{split}\end{align}
By \eqref{eq:almost} and Schur orthogonality, the last expression is $\tr(\beta \mid \Hom_{G^F}(R_1, R_w))$.
\end{proof}

\begin{lem}\label{lem:point-count}
For all $\beta \in \Br_W^+$ and $g \in G^F$, we have 
\begin{align}
|G(\beta)_g^F| = \tr(\tilde{\beta} \times g \mid R_1),
\end{align}
where $\tilde{\beta} = q^{\frac{|\beta|}{2}} \beta$.
\end{lem}

\begin{proof}
Fix $B_0 \in \cal{B}^F$.
Suppose that $\beta = \sigma_{s_1} \cdots \sigma_{s_\ell}$, where $s_1, \ldots, s_\ell \in S$.
Then
\begin{align}
G(\beta)_{g, B_0}^F 
\simeq 
\left\{(B_1, \ldots, B_\ell) \in (\cal{B}^F)^\ell : 
\begin{array}{ll}
B_0 = g^{-1} B_\ell g,\\
(B_{i - 1}, B_i) \in O_{s_i} &\text{for $1 \leq i \leq \ell$}
\end{array}
\right\}.
\end{align}
By \eqref{eq:hecke-operator}, the cardinality of the right-hand side is precisely the coefficient of $\bb{1}_{B_0}$ in the expansion of $(\tilde{\beta}, g) \cdot \bb{1}_{B_0}$ with respect to the basis $\{\bb{1}_B\}_{B \in \cal{B}^F}$ of $R_1$.
Summing over $B_0$ gives the result.
\end{proof}

\begin{proof}[Proof of Proposition \ref{prop:multiplicity}]
By Schur orthogonality, Lemma \ref{lem:reciprocity}, and Theorem \ref{thm:kazhdan},
\begin{align}\begin{split}
(Q_u, \Tr{\tilde{\beta}})_W
&=	\frac{1}{|W|} \sum_{w \in W} {\Tr{\beta}(w)} Q_u(w)\\
&=	\frac{1}{|W|} \sum_{w \in W} {\tr(\beta \mid \Hom_{G^F}(R_1, R_w))} R_w(u).
\end{split}\end{align}
On the other hand, by Lemma \ref{lem:point-count}, \eqref{eq:double-centralizer}, and Lemma \ref{lem:restriction},
\begin{align}\begin{split}
|\cal{U}(\beta)_u^F| 
&= 	\tr(\tilde{\beta} \times u \mid R_1)\\
&=	\sum_{\phi \in \hat{W}}
	\phi_{\Q}(\tilde{\beta}) \rho_\phi(u)\\
&=	\sum_{\rho \in \UCh(W)}
	\tr(\beta \mid \Hom_{G^F}(R_1, \rho)) \rho(u).
\end{split}\end{align}
So the following result concludes the proof.
\end{proof}

\begin{prop}[Geck--Lusztig]
We have
\begin{align}
\sum_{u \in C^F} \rho(u) = 
\sum_{u \in C^F}
\frac{1}{|W|}
\sum_{w \in W}
{(\rho, R_w)_{G^F}}
R_w(u)
\end{align}
for any $G$-orbit $C \subseteq \cal{U}$ and $\rho \in \UCh(W)$.
\end{prop}

\begin{proof}
This is shown on pages 497-498 of \cite{lusztig_2011}, relying on \cite[Prop.\ 1.3]{geck_1996}.
\end{proof}

\begin{rem}
If $G$ is of type $A$, then the identity holds without summing over $C^F$ on both sides, \emph{i.e.}, at the level of individual $u \in C^F$.
\end{rem}

\subsection{}

Recall from Section \ref{sec:decat} that for any topological braid $\beta$, we write $\hat{\beta}$ for its link closure and $[\hat{\beta}]_{\Q} \in \bb{Z}(\!(\Q^{\frac{1}{2}})\!)[a^{\pm 1}]$ to denote the HOMFLY series of the link; for any finite Coxeter group $W$, we write $\mathsf{tr} : H_W \to \bb{Z}[\Q^{\pm\frac{1}{2}}](a)$ to denote Gomi's Markov trace.
In \cite{kalman}, K\'alm\'an proved the following result.

\begin{thm}[K\'alm\'an]
For all $n$ and $\beta \in \Br_n$, we have an equality of $\Q^{\frac{1}{2}}$-series:
\begin{align}
\text{coefficient of $a^{|\beta| + n - 1}$ in $[\widehat{\beta\pi}]$}
=
\text{coefficient of $a^{|\beta| - n + 1}$ in $[\hat{\beta}]$}.
\end{align}
More generally, if $W$ is any finite Coxeter group of rank $r$, then we have an equality
\begin{align}
\text{coefficient of $a^{2r}$ in $\overline{\mathsf{tr}}(\beta \pi)$}
=
\text{coefficient of $a^0$ in $\overline{\mathsf{tr}}(\beta)$},
\end{align}
where $\overline{\mathsf{tr}} = \pa{\frac{1 - a^2}{1 - \Q}}^r\mathsf{tr}$.
\end{thm}

More precisely, \cite{kalman} includes two proofs of the statement about HOMFLY.
The second proof only uses the structure of $H_W$, so it generalizes to a proof of the statement about $\mathsf{tr}$.
By Proposition \ref{prop:markov}, we deduce:

\begin{cor}
For any finite Coxeter group $W$ and $\beta \in \Br_W$, we have
\begin{align}
(\varepsilon, \Tr{\beta\pi}^0)_W = \Q^{\frac{|\pi|}{2}} (1, \Tr{\beta}^0)_W.
\end{align}
\end{cor}

\begin{cor}\label{cor:virtual}
Under the hypotheses of Theorem \ref{thm:virtual-reductive}, we have
\begin{align}
|\cal{X}(\beta\pi)^F| = |\cal{U}(\beta)^F|.
\end{align}
\end{cor}

\begin{proof}
The proof is similar to that of Corollary \ref{cor:kr-to-varieties} in Section \ref{sec:kr}.
Namely, by Springer theory, we have:
\begin{enumerate}
\item 	$(1, Q_u)_W = 1$ for all $u \in \cal{U}^F$.
\item 	$(\varepsilon, Q_1)_W = \Q^{\dim \cal{B}}$ and $(\varepsilon, Q_u)_W = 0$ for all $u \neq 1$.
\end{enumerate}
Now, Theorem \ref{thm:virtual-reductive} and the preceding corollary together imply $q^{\dim \cal{B}} |\cal{X}(\beta\pi)^F| = q^{\frac{|\pi|}{2}} |\cal{U}(\beta)^F|$.
But we have $\dim \cal{B} = N = |w_0| = \frac{1}{2}|\pi|$.
\end{proof}

\begin{rem}
We claim that, in the case where $\beta = \sigma_w$, Corollary \ref{cor:virtual} recovers Proposition 2.2 of \cite{lusztig_2021}.
In our notation, the latter states that 
\begin{align}\label{eq:r-polynomial}
|\cal{U}(\sigma_w)^F| = q^N |\cal{B}^F| R_{1, w}(q),
\end{align}
where $R_{1, w}(\Q) \in \bb{Z}[\Q]$ is the so-called $R$-polynomial defined in \cite{kl}.
If $B_+, B_- \in \cal{B}^F$ are fixed Borels in opposite position, then
\begin{align}
R_{1, w} = |\{B \in \cal{B}^F : (B_+, B) \in O_w,\, (B, B_-) \in O_{w_0}\}|
\end{align}
by \cite[Lem.\ A.4]{kl}.
So, by the transitivity of the $G^F$-action on $O_{w_0}^F$, we have
\begin{align}
|\cal{X}(\sigma_w\pi)^F| = |O_{w_0}^F| R_{1, w}(q) 
= q^{\dim \cal{B}} |\cal{B}^F| R_{1, w}(q).
\end{align}
Now Corollary \ref{cor:virtual} implies \eqref{eq:r-polynomial}.

\end{rem}

\subsection{}

In the next subsection, we will prove Theorem \ref{thm:induction}.
Before doing so, we set notation for parabolic subgroups and review how they interact with the virtual representations $R_w$.

Fix a split torus $T_0 \subseteq G_0$.
We identify the character lattice $\bb{X}$ in the root datum of $G$ with the character lattice of $T$.
We identify the set $S \subseteq W$ with the set of reflections in $\bb{V}^\vee$ corresponding to a system of simple roots in $\bb{X}$.
This system of simple roots determines a Borel $B_0 = T_0 \ltimes U_0$.

If $P_0$ is a parabolic subgroup of $G_0$ containing $B_0$, then the Weyl group of the Levi quotient of $P_0$ is an $S$-parabolic subgroup of $W$ (see \S\ref{subsec:parabolic}).
(Conversely, every $S$-parabolic subgroup of $W$ arises in this way.)
We fix such a parabolic $P_0$.
We write $U_{P, 0}$ and $L_0$ for its unipotent radical and Levi quotient, respectively.
In general, we will distinguish the data attached to $L_0$ with a subscript.
Thus,
\begin{enumerate}
\item 	$W_L$ is the Weyl group of $L_0$.
\item 	$\cal{B}_{L, 0}$ is the flag variety of $L_0$.
\item 	$\cal{U}_{L, 0}$ is the unipotent locus of $L_0$.
\item 	If $\beta \in \Br_{W_L}^+$, then we write $O_L(\beta)_0$ for the analogue of $O(\beta)$ with $L_0$ in place of $G_0$.
\item 	If $w \in W_L$, then $R_{L, w}$ is the Deligne--Lusztig virtual character with $L^F$ in place of $G^F$.
\end{enumerate}
The following formula for the parabolic restriction of the $R_w$ is \cite[Prop.\ 7.4.4]{carter}, though it originates in \cite[\S{8}]{dl}.

\begin{prop}[Deligne--Lusztig]\label{prop:induction-deligne-lusztig}
For all $w \in W$, we have
\begin{align}
(R_w)^{U_P^F} = \frac{1}{|W_L|} \sum_{\substack{x \in W \\ xwx^{-1} \in W_L}} R_{L, xwx^{-1}}
\end{align}
as virtual representations of $L^F$, where $(-)^{U_P^F}$ means we take $U_P^F$-invariants.
\end{prop}

\subsection{}

We now use Theorem \ref{thm:virtual} to prove Theorem \ref{thm:induction}.
In what follows, let $W' \subseteq W$ be a parabolic subgroup, and let $d \geq 0$ be the difference in their ranks.

\subsubsection{}

First, we establish the commutative diagram:
\begin{equation}\label{eq:induction-br}
\begin{tikzpicture}[baseline=(current bounding box.center), >=stealth]
\matrix(m)[matrix of math nodes, row sep=2.5em, column sep=3em, text height=2ex, text depth=0.5ex]
{	\Br_{W'}
		&R(W')[\![\Q]\!]\\
	\Br_W
		&R(W)[\![\Q]\!]\\
		};
\path[->,font=\scriptsize, auto]
(m-1-1)		edge node{$\Tr{-}^0$} (m-1-2)
			edge (m-2-1)
(m-1-2) 	edge node{$\frac{1}{(1 - \Q)^d} \Ind_{W'}^W$} (m-2-2)
(m-2-1) 	edge node{$\Tr{-}^0$} (m-2-2);
\end{tikzpicture}
\end{equation}
We can and will assume that $W' = W_L$ in the setup of the previous subsection.

To disambiguate notation, we will write $\Tr{-}^0$, \emph{resp.}\ $\Tr{-}^{0'}$, for the function on $\Br_W$, \emph{resp.}\ $\Br_{W'}$.
For all $w \in W'$, let the $G_0$-schemes $\cal{U}_{w, 0}$, $\cal{U}_{L, w, 0}$ over $\bb{F}$ be defined by the cartesian squares:
\begin{equation}
\begin{tikzpicture}[baseline=(current bounding box.center), >=stealth]
\matrix(m)[matrix of math nodes, row sep=2.5em, column sep=3em, text height=2ex, text depth=0.5ex]
{	O_{w, 0}
		&\cal{U}_{w, 0}\\
	\cal{B}_0 \times \cal{B}_0
		&\cal{U}_0 \times \cal{B}_0\\
		};
\path[->,font=\scriptsize, auto]
(m-1-2)		edge (m-1-1)
			edge (m-2-2)
(m-1-1) 	edge (m-2-1)
(m-2-2) 	edge node[above]{$\act$} (m-2-1);
\end{tikzpicture}
\qquad
\begin{tikzpicture}[baseline=(current bounding box.center), >=stealth]
\matrix(m)[matrix of math nodes, row sep=2.5em, column sep=3em, text height=2ex, text depth=0.5ex]
{	O_{L, w, 0}
		&\cal{U}_{L, w, 0}\\
	\cal{B}_{L, 0} \times \cal{B}_{L, 0}
		&\cal{U}_{L, 0} \times \cal{B}_{L, 0}\\
		};
\path[->,font=\scriptsize, auto]
(m-1-2)		edge (m-1-1)
			edge (m-2-2)
(m-1-1) 	edge (m-2-1)
(m-2-2) 	edge node[above]{$\act$} (m-2-1);
\end{tikzpicture}
\end{equation}
Let $\bb{V}' = \bb{V}/\bb{V}^{W'}$, so that $\bb{V}'$ is a realization of $W'$ satisfying $(\bb{V}')^{W'} = 0$.
Then $d = \dim \bb{V} - \dim \bb{V}'$.
So for all $x \in W'$, Theorems \ref{thm:kazhdan} and \ref{thm:virtual-reductive} give
\begin{align}
\Tr{\sigma_w}^0(x)|_{\Q = q}
&=	\frac{(-1)^{r - |\beta|}}{|G^F|}
	\sum_{u \in \cal{U}^F} 
	|(\cal{U}_w)_u^F| R_x(u),\\
\frac{(\Ind_{W'}^W \Tr{\sigma_w}^{0'})(x)|_{\Q = q}}{(1 - q)^d}
&=	\frac{(-1)^{r - |\beta|}}{|L^F||W'|}
	\sum_{u' \in \cal{U}_L^F} 
	\sum_{\substack{y \in W \\ yxy^{-1} \in W'}}
	|(\cal{U}_{L, w})_{u'}^F| 
	R_{L, yxy^{-1}}(u').
\end{align}
Since $\beta \mapsto {(-\Q^{-\frac{1}{2}})^{|\beta|}}\Tr{\beta}^0$ is linear as a function on $H_W$, it remains to match the right-hand sides above for all $w$ and $x$.
By Proposition \ref{prop:induction-deligne-lusztig}, we further reduce the problem to:
\begin{align}\label{eq:l-to-g}
\frac{1}{|G^F|}
	\sum_{u \in \cal{U}^F} 
	|(\cal{U}_w)_u^F| R_x(u)
&=	
	\frac{1}{|L^F|}
	\sum_{u' \in \cal{U}_L^F} 
	|(\cal{U}_{L, w})_{u'}^F| 
	(\Res_{P^F}^{G^F} R_x)^{U_P^F}(u').
\end{align}
We will rewrite the point counts on both sides.
Let $(\cal{U}_w)_B \subseteq \cal{U}_w$ be the fiber of $\cal{U}_w \to \cal{U} \times \cal{B} \to \cal{B}$ above our fixed Borel $B \in \cal{B}^F$, so that
\begin{align}
(\cal{U}_w)_B^F = \{u \in \cal{U}^F : (u^{-1}Bu, B) \in O_w\}.
\end{align}
Let $B^F$ act on $G^F$, \emph{resp.}\ $(\cal{U}_w)_B^F$, by right multiplication, \emph{resp.}\ right conjugation.
(Thus its action on $(\cal{U}_w)_B^F \times G^F$ is free.)
Analogously to Lemma \ref{lem:assoc-bundle}, we have:

\begin{lem}\label{lem:bruhat}
We have a $G^F$-equivariant bijection
\begin{align}
\begin{array}{rcl}
((\cal{U}_w)_B^F \times G^F)/B^F 
	&\xrightarrow{\sim} 
	&\cal{U}_w^F\\
{[u, g]}
	&\mapsto 
	&gug^{-1}
\end{array}
\end{align}
where $G^F$ acts on $((\cal{U}_w)_B^F \times G^F)/B^F$ by left multiplication on the factor of $G^F$.
\end{lem}

Let $\{\cal{U}_i^F\}_i$ be the set of $G^F$-orbits of $\cal{U}^F$.
For each $i$ such that $\cal{U}_i^F$ intersects $(\cal{U}_w)_B^F$, we pick a point $u_i \in \cal{U}_i^F$.
Using Lemma \ref{lem:bruhat}, we expand:
\begin{align}\begin{split}
\frac{1}{|G^F|}
\sum_{u \in \cal{U}^F} 
	|(\cal{U}_w)_u^F| R_x(u)
&=	\frac{1}{|B^F|}
\sum_i
|\cal{U}_i^F \cap (\cal{U}_w)_B^F| R_x(u_i)\\
&=	\frac{1}{|B^F|}
\sum_{u \in (\cal{U}_w)_B^F}
R_x(u).
\end{split}\end{align}
Let $B_L = L \cap B$.
As above, let $\{\cal{U}_{L, j}^F\}_j$ be the set of $L^F$-orbits of $\cal{U}_L^F$.
For each $j$ such that $\cal{U}_{L, j}^F$ intersects $(\cal{U}_{L, w})_{B_L}^F$, we pick a point $u'_j \in (\cal{U}_{L, j})_{B_L}^F$.
Then we can also expand:
\begin{align}\begin{split}
\frac{1}{|L^F|}
	\sum_{u' \in \cal{U}_L^F}
	|(\cal{U}_{L, w})_{u'}^F|
	(R_x)^{U^F}(u')
&=	\frac{1}{|B_L^F|}
	\sum_j
	|(\cal{U}_{L, j})_B^F|
	(R_x)^{U^F}(u'_j)
	\\
&=	\frac{1}{|B_L^F|}
	\sum_{u' \in (\cal{U}_{L, j})_B^F}
	(R_x)^{U^F}(u')\\
&=	\frac{1}{|B_L^F||U_P^F|}
	\sum_{\substack{u' \in (\cal{U}_{L, j})_B^F \\ u'' \in U_P^F}}
	R_x(u'u'').
\end{split}\end{align}
Above, $|B_L^F||U_P^F| = |B^F|$.
So the following lemma concludes the proof of \eqref{eq:l-to-g}.

\begin{lem}
For all $w \in W_L$, we have $(\cal{U}_w)_B = (\cal{U}_{L, w})_{B_L} U_P \simeq (\cal{U}_{L, w})_{B_L} \times U_P$.
\end{lem}

\begin{proof}
Lifting $w \in W_L \simeq N_L(T)/T$ to a point $\dot{w} \in N_L(T)$, we observe that
\begin{align}
B\dot{w}B = U_PB_L\dot{w}B_LU_P = (B_L\dot{w}B_L)U_P \simeq B_L\dot{w}B_L \times U_P.
\end{align}
But $(\cal{U}_w)_B = B\dot{w}B \cap \cal{U}$ and $(\cal{U}_{L, w})_{B_L} = B_L\dot{w}B_L \cap \cal{U}_L$.
\end{proof}

\subsubsection{}

We claim that commutativity of \eqref{eq:induction-br} implies commutativity of:
\begin{equation}
\begin{tikzpicture}[baseline=(current bounding box.center), >=stealth]
\matrix(m)[matrix of math nodes, row sep=2.5em, column sep=3em, text height=2ex, text depth=0.5ex]
{	H_{W'}
		&R(W')[\Q^{\pm\frac{1}{2}}]\\
	H_W
		&R(W)[\Q^{\pm\frac{1}{2}}]\\
		};
\path[->,font=\scriptsize, auto]
(m-1-1)		edge node{$\Tr{-}$} (m-1-2)
			edge (m-2-1)
(m-1-2) 	edge node{$\Ind_{W'}^W$} (m-2-2)
(m-2-1) 	edge node{$\Tr{-}$} (m-2-2);
\end{tikzpicture}
\end{equation}
This follows from the projection formula for induced representations, as we now explain.
If $[\Alt(\bb{V})]_{\Q} = \sum_i {(-\Q)^i} \Alt^i(\bb{V})$, then $\varepsilon [\Alt(\bb{V})]_{\Q}$ is the multiplicative inverse of $\varepsilon [\Sym(\bb{V})]_{\Q}$ in $R(W)[\![\Q]\!]$.
By the projection formula, we have
\begin{align}\begin{split}
(-\Q^{\frac{1}{2}})^{|\beta|} 
\Ind_{W'}^W \Tr{\beta} 
&=	
	\Ind_{W'}^W
	{(\Tr{\beta}^0 \cdot \varepsilon_{W'} [\Alt(\bb{V}')]_{\Q})}\\
&=	
	\Ind_{W'}^W
	{(\Tr{\beta}^0 \cdot \Res_{W'}^W(\varepsilon_W [\Alt(\bb{V}')]_{\Q}))}\\
&=	
	(\Ind_{W'}^W \Tr{\beta}^0) \cdot \varepsilon_W [\Alt(\bb{V}')]_{\Q}
\end{split}\end{align}
Since $d = \dim \bb{V} - \dim \bb{V}'$, we have
\begin{align}\begin{split}
(-\Q^{\frac{1}{2}})^{|\beta|} \Tr{\beta}
&=	
	\Tr{\beta}^0 \cdot \varepsilon_W [\Alt(\bb{V})]_{\Q}\\
&=
	\frac{1}{(1 - \Q)^d}\,\Tr{\beta}^0 \cdot \varepsilon_W [\Alt(\bb{V}')]_{\Q}.
\end{split}\end{align}
So the commutativity of \eqref{eq:induction-br} implies $\Ind_{W'}^W(\Tr{\beta}) = \Tr{\beta}$, as needed.

\subsection{}

Finally, we prove Theorem \ref{thm:pole}, stating that
\begin{align}
\Tr{\beta}^0 \in \frac{1}{(1 - \Q)^{r(w)}}\, R(W)[\Q],
\end{align}
where $w$ is the image of $\beta$ in $W$ and $r(w) = \dim \bb{V}^w$.

First, we introduce a function $\delta : W \to \bb{Z}$.
For all $w \in W$, we claim that
\begin{align}
\delta(w) = \lim_{\Q \to 1} 
\frac{\det(1 - \Q w \mid \bb{V})}{(1 - \Q)^{r(w)}}
\end{align}
exists and takes values in $\bb{Z}$.
Indeed, since $w$ has finite order, its action on $\bb{V}$ is diagonalizable.
Thus $r(w)$ equals the multiplicity of $1 - \Q$ in the characteristic polynomial $\det(1 - \Q w \mid \bb{V})$.
This shows that $\delta(w)$ exists.
It is an integer because $\bb{Q}_W = \bb{Q}$ and character values are algebraic integers.
We can now state:

\begin{thm}[Lusztig]\label{thm:isotropy}
If $G$ is semisimple and $\beta \in \Br_W^+$ maps to $w \in W$, then the isotropy groups of the $G$-action on $G(\beta)$ are diagonalizable groups $D$ such that 
\begin{align}
|D^F|
\quad\text{divides}\quad
(q - 1)^{r(w)} \delta(w)|Z(G)|,
\end{align}
where $Z(G)$ is the center of $G$.
In particular, the isotropy groups of the $G^F$-action on $G(\beta)^F$ satisfy the same divisibility condition.
\end{thm}

\begin{proof}
Fix a maximal split torus $T_0 \subseteq G_0$ and a lift $\dot{w} \in N_G(T)$ of $w$.
In the case where $G$ is adjoint and $w$ is \dfemph{elliptic}, meaning $T^w$ is discrete, the proof is analogous to the proof of \cite[Thm.\ 5.2]{lusztig_2011}.
However, the argument there actually shows more:
For general $G$ and $w$, it shows that we can assume $D \subseteq T^{\dot{w}}$, independently of $\beta$.
Let $\bar{G} = G/Z(G)$, the adjoint group of $G$, and let $\bar{T}, \bar{D}$ be the images of $T, D$ in $\bar{G}$.
Writing $\bar{T}'$ for the identity component of $\bar{T}^{\dot{w}}$, we see that $\bar{T}'$ is a split torus of rank $r(w)$ such that $\bar{T}^{\dot{w}}/\bar{T}'$ is discrete.
So we get $|(\bar{T}^{\dot{w}})^F| = (q - 1)^{r(w)}\delta(w)$ from the elliptic case.
Finally, $|(T^{\dot{w}})^F| = |(\bar{T}^{\dot{w}})^F||Z(G)|$.
\end{proof}

\begin{proof}[Proof of Theorem \ref{thm:pole}]
Combining Theorems \ref{thm:virtual} and \ref{thm:isotropy}, we see that 
\begin{align}\label{eq:pole-1}
\Tr{\beta}^0|_{\Q = q} \in \frac{1}{(1 - q)^{r(w)} \delta(w)|Z(G)|}\, R(W)
\end{align}
for all $q \gg 0$.
At the same time, by Proposition \ref{prop:rationality},
\begin{align}\label{eq:pole-2}
\Tr{\beta}^0 \in \frac{1}{(1 - \Q^{d_1}) \cdots (1 - \Q^{d_r})}\, R(W)[\Q_0],
\end{align}
where $\Q_0$ is defined by \eqref{eq:exception} and $d_i \in \bb{Z}$ for all $i$.
Therefore, \eqref{eq:pole-1} simplifies to
\begin{align}
\Tr{\beta}^0|_{\Q = q} \in \frac{1}{(1 - q)^{r(w)}}\, R(W)
\end{align}
for all $q \gg 0$.
Therefore, \eqref{eq:pole-2} simplifies to
\begin{align}
\Tr{\beta}^0 \in \frac{1}{(1 - \Q)^{r(w)}}\, R(W)[\Q],
\end{align}
as needed.
\end{proof}

\newpage
\section{Periodic Braids}\label{sec:periodic}

\subsection{}

In this section and the next, \emph{we no longer assume that $W$ is crystallographic}.
Let $w_0$ be the longest element of $W$ with respect to the chosen Coxeter presentation (see \S\ref{subsec:coxeter-group}).
The \dfemph{full twist} of $\Br_W$ is the positive braid
\begin{align}
\pi = \sigma_{w_0}^2 \in \Br_W^+.
\end{align}
It is central in $\Br_W$ \cite[\S{2}]{bm_1997}.
We say that a braid $\beta \in \Br_W$ is \dfemph{periodic} iff 
\begin{align}
\beta^n = \pi^m
\end{align}
for some $\frac{m}{n} \in \bb{Q}$, in which case we say that $\frac{m}{n}$ is the \dfemph{slope} of $\beta$.

\begin{thm}[Lee--Lee]\label{thm:ll}
If $\beta \in \Br_W$ is periodic of slope $\frac{m}{n}$, then $\beta = \gamma^m$ and $\pi = \gamma^n$ for some $\gamma \in \Br_W$.
\end{thm}

\begin{proof}
Let $d$ be the minimal positive exponent such that $\sigma_{w_0}^d$ is central in $\Br_W$.
Either $d = 1$ or $d = 2$.

In \cite{ll}, E.-K.\ Lee and S.-J.\ Lee say that an element $\gamma \in \Br_W$ is \dfemph{primitive} iff, whenever $\gamma = \delta^\ell$ for some $\delta \in \Br_W$, we have $\ell = \pm 1$.
Theorem 3.14 of \emph{ibid.}\ implies that if $\gamma$ is primitive and periodic, then $\gamma$ is a positive root of $\sigma_{w_0}^d$.

If $\beta$ is periodic, then any root of $\beta$ is also periodic.
So it remains to check that some root of $\beta$ is primitive.
Indeed, this follows from the well-ordering principle, because the writhe of $\beta$ is an integer.
\end{proof}

\begin{rem}
In type $A$, where the elements of $\Br_W$ represent topological braids, Theorem \ref{thm:ll} was first shown in 1934 by Ker\'ekj\'art\'o--Eilenberg \cite{eilenberg}.
\end{rem}

\begin{rem}
A result of Bessis implies that if an $n$th root of $\pi$ exists, then the $n$th roots of $\pi$ form a single conjugacy class in $\Br_W$ \cite[Thm.\ 12.4(ii)]{bessis}.
\end{rem}

\subsection{}

The following result, an incarnation of Schur's lemma, is \cite[Thm.\ 9.4.3]{gp}.

\begin{thm}[Springer]
For all $\phi \in \hat{W}$, the full twist acts on the underlying $H_W$-module of $\phi_{\Q}$ by the scalar $\Q^{\bb{c}(\phi)}$, where $\bb{c}(\phi)$ is the content of $\phi$ (see \S\ref{subsec:families}).
In particular, $\phi_{\Q}(\pi) = \Q^{\bb{c}(\phi)}\phi(1)$.
\end{thm}

The following corollary is a mild generalization of \cite[Prop.\ 9.2.8]{bm_1997}, but we give the proof anyway.
The proof is based on \cite[Lem.\ 9.4]{jones}, which Jones used to calculate the HOMFLY polynomials of torus knots.

\begin{cor}\label{cor:br-to-w}
If $\phi \in \hat{W}$ and $\beta \in \Br_W$ is periodic of slope $\slope \in \bb{Q}$, then 
\begin{align}
\phi_{\Q}(\beta) = \Q^{\slope \bb{c}(\phi)} \phi(w),
\end{align}
where $w \in W$ is the image of $\beta$ under the surjection $\Br_W \surject W$.
\end{cor}

\begin{proof}
It suffices to prove the result when $\Q^{\frac{1}{2}}$ has been specialized to a generic complex number $q^{\frac{1}{2}}$.
We set 
\begin{align}
H_W[q^{\frac{1}{2}}] = H_W \otimes_{\bb{Z}[\Q^{\pm 1/2}]} \bb{C},
\end{align}
where $\bb{Z}[\Q^{\pm\frac{1}{2}}] \to \bb{C}$ sends $\Q^{\frac{1}{2}} \mapsto q^{\frac{1}{2}}$, and 
\begin{align}
\phi_q = \phi_{\Q}|_{\Q^{1/2} = q^{1/2}}.
\end{align}
Write $\slope = \frac{m}{n}$ in lowest terms.
By the theorem, $q^{-m\bb{c}(\phi)}\pi^m$ acts on the underlying $H_W[q^{\frac{1}{2}}]$-module of $\phi_q$ by the identity operator $I$.
Since $\beta^n = \pi^m$, we deduce that $q^{-\slope \bb{c}(\phi)} \beta$ acts by an $n$th root of $I$.

Recall that $H_W|_{q = 1} = \bb{Z}[W]$.
Choose a generic path in the complex plane from $q$ to $1$.
Since the (diagonalizable) roots of $I$ do not deform, the deformation of $q$ along this path shows that the following operators must have the same eigenvalue spectrum:
\begin{itemize}
\item 	The action of $q^{-\slope\bb{c}(\phi)}\beta$ on the underlying $H[q^{\frac{1}{2}}]$-module of $\phi_q$.
\item 	The action of $w$ on the underlying $\bb{C}[W]$-module of $\phi$.
\end{itemize}
The result follows.
\end{proof}

\subsection{}

Let $\bb{V}$ be a realization of $W$ over $\bb{C}$ such that $\bb{V}^W = 0$.
A vector of $\bb{V}$ is \dfemph{regular} with respect to $W$ iff its stabilizer in $W$ is trivial, or equivalently, it avoids the reflecting hyperplanes that generate the $W$-action.
Suppose that an element $w \in W$ admits a regular eigenvector in $\bb{V}$ with eigenvalue $\zeta \in \bb{C}^\times$.
In this case, $\zeta$ must be a root of unity.
We say that $\zeta$ is \dfemph{regular} with respect to $W$ and that $w$ is a \dfemph{$\zeta$-regular element} of $W$.
The following results are parts (iv) and (v) of \cite[Thm.\ 4.2]{springer_1974}, respectively.

\begin{thm}[Springer]\label{thm:conjugacy}
For any $\zeta \in \bb{C}^\times$, the set of $\zeta$-regular elements of $W$ is either empty or forms a single conjugacy class.
\end{thm}

\begin{thm}[Springer]\label{thm:fake-deg}
If $w \in W$ is a $\zeta$-regular element, then
\begin{align}
\phi(w) = \FDeg_\phi(\zeta^{-1}) = \FDeg_\phi(\zeta)
\end{align}
for all $\phi \in \hat{W}$, where $\FDeg_\phi$ is the fake degree of $\phi$ (see \S\ref{subsec:degrees}).
\end{thm}

\begin{rem}
In \cite[170]{springer_1974}, Springer claims that if $n$ is a fixed integer, then the regular elements of $W$ of order $n$ form a single conjugacy class.
This claim, which is stronger than Theorem \ref{thm:conjugacy}, fails when $W$ is not crystallographic.
For example, let $W$ be the Coxeter group of type $I_2(5)$ and let $w$ be a Coxeter element.
Then $w$ and $w^2$ are both regular elements of order $5$, but they are not conjugate.
\end{rem}

Regular elements of $W$ are related to periodic elements of $\Br_W$ by the following results, which are respectively Proposition 3.11 and Th\'eor\`eme 3.12 in \cite{bm_1997}.

\begin{prop}[Brou\'e--Michel]\label{prop:broue-michel}
If $w \in W$ is regular of order $n$, then $\sigma_w^n = \pi$.
\end{prop}

\begin{thm}[Brou\'e--Michel]\label{thm:broue-michel}
If $\gamma \in \Br_W^+$ satisfies $\gamma^n = \pi$ for some $n$, then its image in $W$ is a $e^{\frac{2\pi i}{n}}$-regular element.
\end{thm}

\begin{cor}\label{cor:regular}
If $\beta \in \Br_W$ is periodic of slope $\nu \in \bb{Q}$, then the image of $\beta$ in $W$ is a $e^{2\pi i\nu}$-regular element.
\end{cor}

\begin{proof}
This is immediate from Theorem \ref{thm:broue-michel} when $\beta$ is a root of $\pi$.
Theorem \ref{thm:ll} lets us bootstrap to the general case.
\end{proof}

\begin{rem}
We emphasize that Theorem \ref{thm:broue-michel} is deeper than Proposition \ref{prop:broue-michel}:
Its proof relies on results used by Charney to prove the biautomaticity of Artin groups \cite[\S{3.19}]{bm_1997}.
\end{rem}

\subsection{}

Theorem \ref{thm:periodic} states that if $\beta \in \Br_W$ is periodic of slope $\slope \in \bb{Q}$, then its $R(W)$ trace is given by
\begin{align}
\Tr{\beta} 
=
	\sum_{\psi \in \hat{W}} 
	{\Q^{\slope\bb{c}(\psi)}}
	\UDeg_\psi(\zeta) \psi,
\end{align}
where $\zeta = e^{2\pi i\slope}$ and $\UDeg_\phi$ is the unipotent degree of $\phi$ (see \S\ref{subsec:degrees}).

\begin{proof}[Proof of Theorem \ref{thm:periodic}]
By Corollary \ref{cor:regular}, the image of $\beta$ in $W$ is a $\zeta$-regular element.
Therefore, by Corollary \ref{cor:br-to-w} and Theorem \ref{thm:fake-deg}, 
\begin{align}
\phi_{\Q}(\beta) = \Q^{\slope\bb{c}(\phi)} \FDeg_\phi(\zeta).
\end{align}
Therefore, by Lemma-Postulate \ref{lempost:fourier}(2) and Corollary \ref{cor:content}, 
\begin{align}\begin{split}
\sum_{\phi \in \hat{W}} 
{\{\phi, \psi\}} \phi_{\Q}(\beta) 
&= 	\Q^{\slope\bb{c}(\psi)}
	\sum_{\phi \in \hat{W}} 
	{\{\phi, \psi\}}
	\FDeg_\phi(\zeta) \\
&= 	\Q^{\slope \bb{c}(\psi)}
	\UDeg_\psi(\zeta),
\end{split}\end{align}
as needed.
\end{proof}

\newpage
\section{The Rational DAHA}\label{sec:daha}

\subsection{}

As in the previous section, \emph{we do not assume that $W$ is crystallographic}.
Let $\mathrm{Ref}(W) \subseteq W$ be the subset of elements that act on $\bb{V}$ by reflections, \emph{cf.}\ \S\ref{subsec:families}.
We continue to assume that $\bb{V}^W = 0$.

We write $\langle-, -\rangle : \bb{V} \times \bb{V}^\vee \to \bb{Q}_W$ for the evaluation pairing.
For each reflection $t \in \mathrm{Ref}(W)$, we fix nonzero vectors $\alpha_t \in \bb{V}$ and $\alpha_t^\vee \in \bb{V}^\vee$ such that 
\begin{enumerate}
\item 	$\langle\alpha_t, -\rangle = 0$ is the hyperplane of $\bb{V}^\vee$ fixed by $t$.
\item 	$\langle-, \alpha_t^\vee\rangle = 0$ is the hyperplane of $\bb{V}$ fixed by $t$.
\item 	$\langle\alpha_t, \alpha_t^\vee\rangle = 2$.
\end{enumerate}
In the literature, the vectors $\alpha_t$ and $\alpha_t^\vee$ are called roots and coroots even when $W$ is not crystallographic.
Condition (3) will ensure that in what follows, our constructions do not depend on the exact choices of $\alpha_t, \alpha_t^\vee$.

For any $\slope \in \bb{C}$, the \dfemph{rational Cherednik algebra} or \dfemph{rational DAHA of slope $\slope$} is the $\bb{C}$-algebra
\begin{align}
\bb{D}_\slope^\rat
=	\frac{\bb{C}[W] \ltimes (\Sym(\bb{V}) \otimes \Sym(\bb{V}^\vee))}{\left\langle 
yx - xy - \langle x, y\rangle + \slope \displaystyle\sum_{t \in \mathrm{Ref}(W)} \langle x, \alpha^\vee\rangle\langle\alpha, y\rangle t
: x \in \bb{V},\,y \in \bb{V}^\vee
\right\rangle}.
\end{align}
Note that $\bb{D}_\slope^\rat$ contains $\bb{A}_W = \bb{C}[W] \ltimes \Sym(\bb{V})$ as a subalgebra, regardless of $\slope$.

\begin{rem}
If $\slope = 0$, then the defining relations of $\bb{D}_\slope^\rat$ simplify to $[x, y] = \langle x, y\rangle$ for all $x$ and $y$.
This shows that $\bb{D}_0^\rat$ is isomorphic to $\bb{C}[W] \ltimes \cal{D}(\bb{V})$, the semidirect product of $\bb{C}[W]$ with the Weyl algebra of $\bb{V}$.
\end{rem}

\subsection{}

Like the universal enveloping algebra of a semisimple Lie algebra, $\bb{D}_\slope^\rat$ admits a triangular decomposition in the sense of Poincar\'e--Birkhoff--Witt.
It therefore admits an analogue of the Bernstein--Gelfand--Gelfand category $\sf{O}$, which we will denote by $\sf{O}_\slope$ \cite{ggor}.

Any $\bb{D}_\slope^\rat$-module $M \in \sf{O}_\slope$ can be viewed as a $W$-equivariant quasicoherent sheaf on $\bb{V}^\vee = \Spec \Sym(\bb{V})$.
The \dfemph{support} of $M$ is the $W$-stable subvariety of $\bb{V}^\vee$ that forms its support as a sheaf.
Let $\sf{O}_\slope^\tor \subseteq \sf{O}_\slope$ be the Serre subcategory of modules that are torsion over $\Sym(\bb{V}^\vee) \subseteq \bb{D}_\slope^\rat$.
By \cite[Lem.\ 5.1]{ggor}, $\sf{O}_\slope^\tor$ is also the full subcategory of $\sf{O}_\slope$ of modules supported on the discriminant locus of $\bb{V}^\vee$, \emph{i.e.}, the common complement of the reflection hyperplanes of the $W$-action.

The Verma modules in $\sf{O}_\slope$ are indexed by the set $\hat{W}$.
For all $\phi \in \hat{W}$, we write $\Delta_\slope(\phi)$ for the corresponding Verma module of $\bb{D}_\slope^\rat$, and we write $L_\slope(\phi)$ for its simple quotient.
Explicitly, $\Delta_\slope(\phi)$ can be constructed by inflating $\phi$ from $\bb{C}[W]$ up to $\bb{A}_W$, then tensoring up to $\bb{D}_\slope^\rat$.
In the literature, $\Delta_\slope(1)$ is known as the polynomial module and $L_\slope(1)$ as the simple spherical module.

Like the Weyl algebra of $\bb{V}$, the rational Cherednik algebra contains a canonical ``Euler element'' $\bm{h}$.
For the precise definition, see \cite[Ch.\ 11]{etingof_2007}.
It commutes with the subalgebra $\bb{C}[W] \subseteq \bb{D}_\slope^\rat$, and its action on any object of $\sf{O}_\slope$ is locally finite.
Thus, each object $M \in \sf{O}$ is endowed with a decomposition into finite-dimensional, $W$-stable $\bm{h}$-eigenspaces.
On all $\bb{D}_\slope^\rat$-modules that we will need, the eigenvalues will be half-integers.
We define the \dfemph{$\Q$-graded character} of such a module $M$ to be
\begin{align}
[M]_{\Q} 
= 	\sum_n {\Q^{\frac{n}{2}}} M_{\frac{n}{2}} \in R(W)(\!(\Q^{\frac{1}{2}})\!)
\end{align}
where $M_{\frac{n}{2}} \subseteq M$ is the $\bm{h}$-eigenspace of eigenvalue $\frac{n}{2}$.
By \cite[Ch.\ 11]{etingof_2007},
\begin{align}\label{eq:verma}
[\Delta_\slope(\phi)]_{\Q}
=	\Q^{\frac{r}{2} - \slope\bb{c}(\phi)} \phi \cdot [\Sym(\bb{V})]_{\Q}
\end{align}
for all $\phi \in \hat{W}$, where we set $[\Sym(\bb{V})]_{\Q} = \sum_i {\Q^i} \Sym^i(\bb{V})$ as in Sections \ref{sec:decat}-\ref{sec:dl} and \S\ref{subsec:realization} of Appendix \ref{sec:coxeter}.

The basic problem in the representation theory of $\bb{D}_\slope^\rat$ is to understand how the graded characters $[\Delta_\slope(\phi)]_{\Q}$, which are explicit, can be used to express the graded characters $[L_\slope(\phi)]_{\Q}$, which are initially mysterious.
The problem is highly sensitive to the value of $\slope$.
When $\slope \notin \bb{Q}$, category $\sf{O}$ is semisimple and $[L_\slope(\phi)]_{\Q} = [\Delta_\slope(\phi)]_{\Q}$.
When $\slope \in \bb{Q}$, category $\sf{O}$ is no longer semisimple:
It grows more complex as the denominator of $\slope$ grows smaller.

\subsection{}

Let $\zeta^{\frac{1}{2}} = e^{\pi i\slope}$.
We draw freely upon the notation in \S\ref{subsec:cyclotomic}, so that
\begin{align}
H_W(\zeta^{\frac{1}{2}}) = H_W \otimes_{\bb{Z}[\Q^{\pm\frac{1}{2}}]} \bb{Q}_W(\zeta^{\frac{1}{2}})
\end{align}
and $\bm{d}_\zeta : R^+(H_W(\Q^{\frac{1}{2}})) \to R^+(H_W(\zeta^{\frac{1}{2}}))$ is the decomposition map that corresponds to the specialization $\Q^{\frac{1}{2}} \mapsto \zeta^{\frac{1}{2}}$.
Let $\sf{Mod}(H_W(\zeta^{\frac{1}{2}}))$ be the category of $H_W(\zeta^{\frac{1}{2}})$-modules, and for each $\phi \in \hat{W}$, let
\begin{align}
\phi_\zeta = \bm{d}_\zeta(\phi_{\Q}) \in \sf{Mod}(H_W(\zeta^{\frac{1}{2}})).
\end{align}
In \cite{ggor}, Ginzburg, Guay, Opdam, and Rouquier introduced the \dfemph{Knizhnik--Zamolodchikov (KZ) functor}, an exact monoidal functor
\begin{align}
\sf{KZ} : \sf{O}_\slope \to \sf{Mod}(H_W(\zeta^{\frac{1}{2}})).
\end{align}
The following facts comprise \cite[Thm.\ 5.14, Cor.\ 5.18, Thm.\ 6.8]{ggor}:

\begin{thm}[Ginzburg--Guay--Opdam--Rouquier]\label{thm:ggor}
For all $\slope \in \bb{C}$,
\begin{enumerate}
\item 	$\sf{KZ}$ induces an equivalence $\sf{O}_\slope/\sf{O}_\slope^\tor \simeq \sf{Mod}(H_W(\zeta^{\frac{1}{2}}))$.
\item 	$\sf{KZ}$ induces an isomorphism between the (vertical) center of $\sf{O}_\slope$, \emph{i.e.}, the algebra of endomorphisms of its unit object, and the center of $H_W(\zeta^{\frac{1}{2}})$.
\item 	$\sf{KZ}(\Delta_\slope(\phi)) = \phi_\zeta$.
\end{enumerate}
\end{thm}

\subsection{}

As preparation for the proofs of Corollary \ref{cor:periodic} and Theorem \ref{thm:slopes}, we introduce some subsets of $\bb{Q}$ depending on $\slope$.
It will be convenient to write $\{d_i\}_{i \in I(W)}$ for the multiset of invariant degrees of $W$.
For any slope $\slope \in \bb{C}$, let
\begin{align}
I_\slope(W) = \{i \in I(W) : \slope \in \tfrac{1}{d_i}\bb{Z}\}.
\end{align}
Concretely, if $n$ is the denominator of $\slope$ in lowest terms, then $I_\slope(W)$ is the set of indices $i$ such that $n$ divides the invariant degree $d_i$.

\begin{df}
If $\slope \in \bb{Q}$ and $\zeta = e^{2\pi i\slope}$, then we say that $\slope$ is:
\begin{enumerate}
\item 	A \dfemph{singular slope} iff $|I_\slope(W)| \geq 1$.
\item 	A \dfemph{regular slope} iff $W$ contains a $\zeta$-regular element.
\item 	A \dfemph{regular elliptic slope} iff $W$ contains a $\zeta$-regular element $w$ such that $\bb{V}^w = 0$.
		(Recall that we assume $\bb{V}^W = 0$.)
\item 	A \dfemph{cuspidal slope} iff $|I_\slope(W)| = 1$.
\end{enumerate}
We say that an integer $n > 0$ is a \dfemph{singular number}, etc., iff it is the denominator in lowest terms of a singular slope, etc.
\end{df}

\begin{lem}\label{lem:implication}
There is a chain of implications:
\begin{align}
\text{cuspidal}
\implies \text{regular elliptic}
\implies \text{regular}
\implies \text{singular}.
\end{align}
\end{lem}

\begin{proof}
It is tautological that regular elliptic implies regular.
That regular implies singular follows from \cite[Thm.\ 4.2(iii)]{springer_1974}.
That cuspidal implies regular elliptic is \cite[Rem.\ 3.28]{be}.
\end{proof}

\begin{ex}
In type $A_r$, the invariant degrees are $2, 3, 4, \ldots, r + 1$.
An integer is regular if and only if it is a divisor of either $r$ or $r + 1$ \cite[\S{5.1}]{springer_1974}.
The only regular elliptic number is the Coxeter number $r + 1$ \cite[Ex.\ 1.1.2]{vv}.
\end{ex}

\begin{ex}
In type $BC_r$, the invariant degrees are $2, 4, 6, \ldots, 2r$.
An integer is regular, \emph{resp.}\ regular elliptic, if and only if it is a divisor of $2r$ \cite[\S{5.2}]{springer_1974}, \emph{resp.}\ an even divisor of $2r$ \cite[Ex.\ 1.1.2]{vv}.
To demonstrate:
In type $BC_5$:
\begin{itemize}
\item 	The singular numbers are $1, 2, 3, 4, 5, 6, 8, 10$.
\item 	The regular numbers are $1, 2, 5, 10$.
\item 	The regular elliptic numbers are $2, 10$.
\item 	The only cuspidal number is $10$, the Coxeter number.
\end{itemize}
This is a minimal example where each implication in Lemma \ref{lem:implication} is strict.
\end{ex}

Throughout the rest of this section, we assume $\slope \in \bb{Q}$ and set $\zeta^{\frac{1}{2}} = e^{\pi i\slope}$. 
We form the virtual $\bb{D}_\slope^\rat$-module
\begin{align}
\Omega_\slope = \bigoplus_{\phi \in \hat{W}} \UDeg_\phi(e^{2\pi i\slope})\Delta_\slope(\phi).
\end{align}
We can immediately prove part (1) of Theorem \ref{thm:slopes}, stating that $L_\slope(1)$ occurs in $\Omega_\slope$ with multiplicity one.
Indeed, we have $D_1(\Q) = 1$, and for any $\slope$, the only Verma module that contains $L_\slope(1)$ as a constituent is $\Delta_\slope(1)$.

\subsection{Singular Slopes}

We prove part (2) of Theorem \ref{thm:slopes}, stating that if $\slope$ is singular, then $\Omega_\slope \in \sf{O}_\slope^\tor$.

\begin{lem}\label{lem:principal-series}
We have 
\begin{align}
\sum_{\phi \in \hat{W}}
D_\phi(\Q)\phi_{\Q}(\sigma_w) = \left\{\begin{array}{ll}
\bm{s}(1_{\Q})	&w = 1\\
0				&w \neq 1
\end{array}\right.
\end{align}
in $\bb{Q}_W(\Q^{\frac{1}{2}})$.
\end{lem}

\begin{proof}
If $W \simeq W_1 \times W_2$, then $D_{\phi_1 \times \phi_2}(\Q) = D_{\phi_1}(\Q) \cdot D_{\phi_2}(\Q)$ for all $\phi_i \in \hat{W}_i$.
So we can assume that $W$ is irreducible.
After checking the non-crystallographic types by hand, we can further assume that $W$ is crystallographic.

Henceforth, we will use the notations and hypotheses of Section \ref{sec:dl}.
Recall that by \eqref{eq:double-centralizer}, we have
\begin{align}
R_1 \simeq \bigoplus_{\phi \in \hat{W}}
V_{\phi, q}^{\oplus \UDeg_\phi(q)}
\end{align}
as $H_W(q^{\frac{1}{2}})$-modules, where $V_{\phi, q}$ is a free $H_W(q^{\frac{1}{2}})$-module that admits the character $\phi_q$.
From \eqref{eq:hecke-operator}, we deduce that:
\begin{enumerate}
\item 	$\sigma_1$ acts on $R_1$ by the identity operator.
		Its trace is $|\cal{B}^F|$.
\item 	If $w \neq 1$, then $\sigma_w$ acts on $R_1$ by an operator with zeros along the diagonal.
\end{enumerate}
Since $|\cal{B}^F| = \bm{s}(1_{\Q})|_{\Q = q}$, we are done.
\end{proof}

By Theorem \ref{thm:ggor}(1), it remains to show that the $H_W(\zeta^{\frac{1}{2}})$-module $\sf{KZ}(\Omega_\slope)$ is the zero module.
By Theorem \ref{thm:ggor}(3), we have
\begin{align}
\sf{KZ}(\Omega_\slope) = \sum_{\phi \in \hat{W}}
\UDeg_\phi(\zeta)\phi_\zeta.
\end{align}
This coincides with the image of the virtual $H_W(\Q^{\frac{1}{2}})$-module $\sum_\phi \UDeg_\phi(\Q)\phi_{\Q}$ under the decomposition map from \S\ref{subsec:cyclotomic}.
So by Lemma \ref{lem:principal-series}, it remains to show
\begin{align}
\bm{s}(1_{\Q})|_{\Q^{1/2} = \zeta^{1/2}} = 0.
\end{align}
This follows from the Bott--Solomon formula \cite{solomon}
\begin{align}\label{eq:bott-solomon}
\bm{s}(1_{\Q}) = \prod_{i \in I(W)} 
\frac{1 - \Q^{d_i}}{1 - \Q}
\end{align}
together with the definition of singular slope.

\subsection{Regular Slopes}

We prove Corollary \ref{cor:periodic}, stating that if $\beta$ is periodic of slope $\slope$, then $[\Omega_\slope] = (\Q^{\frac{1}{2}})^{r - |\beta|} \cdot \Tr{\beta}^0$.
On the one hand,
\begin{align}\begin{split}
(-\Q^{-\frac{1}{2}})^{|\beta|} \Tr{\beta}^0
&=	
	\Tr{\beta} \cdot \varepsilon [\Sym(\bb{V})]_{\Q}\\
&=	
	\varepsilon \sum_{\phi \in \hat{W}}
	{\Q^{\slope \bb{c}(\phi)}} 
	\UDeg_\phi(\zeta) \phi \cdot [\Sym(\bb{V})]_{\Q}
\end{split}\end{align}
by Theorems \ref{thm:decat} and \ref{thm:periodic}.
On the other hand,
\begin{align}\begin{split}
\sum_{\phi \in \hat{W}}
\UDeg_\phi(\zeta) [\Delta_\slope(\phi)]_{\Q}
&=	\Q^{\frac{r}{2}}
	\sum_{\phi \in \hat{W}}
	{\Q^{-\slope \bb{c}(\phi)}}
	\UDeg_\phi(\zeta)
	\phi \cdot [\Sym(\bb{V})]_{\Q}\\
&=	\Q^{\frac{r}{2}}\varepsilon
	\sum_{\phi \in \hat{W}}
	{\Q^{\slope \bb{c}(\phi)}}
	\UDeg_{\varepsilon\phi}(\zeta)
	\phi \cdot [\Sym(\bb{V})]_{\Q}
\end{split}\end{align}
by \eqref{eq:verma} and the identity $\bb{c}(\varepsilon\phi) = -\bb{c}(\phi)$.
Since $|\beta| = \slope|\pi|$, it remains to check:

\begin{lem}
If $\slope \in \bb{Q}$ is a regular slope and $\zeta = e^{2\pi i\slope}$, then 
\begin{align}
\UDeg_{\varepsilon\psi}(\zeta) = (-1)^{\slope|\pi|} \UDeg_\psi(\zeta)
\end{align}
for all $\psi \in \hat{W}$.
\end{lem}

\begin{proof}
Pick a $\zeta$-regular element $w \in W$.
By the discussion in Section \ref{sec:periodic} and Lemma-Postulate \ref{lempost:fourier}(3),
\begin{align}\begin{split}
\UDeg_{\varepsilon\psi}(\zeta)
&=	\sum_{\phi \in \hat{W}}
	{\{\phi, \varepsilon\psi\}}
	\phi(w)\\
&=	\sum_{\phi \in \hat{W}}
	{\{\phi, \psi\}}
	(\varepsilon\phi)(w)\\
&=	\varepsilon(w) \UDeg_\psi(\zeta).
\end{split}\end{align}
Finally, $\varepsilon(w) = (-1)^{\slope|\pi|}$ because $|w| = |\sigma_w| = (\slope - \lfloor \slope\rfloor)|\pi|$.
\end{proof}

\subsection{Regular Elliptic Slopes}

We prove part (3) of Theorem \ref{thm:slopes}, stating that if $\slope$ is regular elliptic, $\slope > 0$, and $W$ is a Weyl group, then $[\Omega_\slope]_{\Q} \in R(W)[\Q^{\pm\frac{1}{2}}]$.

Write $\slope = \frac{m}{n}$, where $m, n > 0$.
Since $\frac{1}{n}$ is also a regular elliptic slope, we can pick a $e^{\frac{2\pi i}{n}}$-regular element $w \in W$ such that $\bb{V}^w = 0$.
By Proposition \ref{prop:broue-michel}, $\sigma_w^m$ is periodic of slope $\slope$, so by Corollary \ref{cor:periodic}, 
\begin{align}
[\Omega_\slope]_{\Q}
\in \Tr{\sigma_w}^0 \cdot R(W)[\Q^{\pm\frac{1}{2}}].
\end{align}
But $W$ is a Weyl group, so by Theorem \ref{thm:pole}, the condition $\bb{V}^w = 0$ implies that $\Tr{\sigma_w}^0 \in R(W)[\Q]$.

\subsection{Cuspidal Slopes}

We prove part (4) of Theorem \ref{thm:slopes}, stating that if $\slope$ is cuspidal and positive, then
\begin{align}
{[\Omega_\slope]_{\Q}}
=	\sum_{\psi \in \hat{W}_1}
	{[L_\slope(\psi)]_{\Q}},
\end{align}
where $\hat{W}_1 \subseteq \hat{W}$ is the set of elements of greatest $\bb{a}$-value within their $\zeta$-blocks, among the $\zeta$-blocks of $H_W$ of defect $1$ in the sense of \S\ref{subsec:cyclotomic}.

\begin{lem}\label{lem:defect-invariants-principal}
If $\zeta^{\frac{1}{2}} = e^{\pi i\slope}$, then the $\zeta$-defect of the principal $\zeta$-block of $H_W(\Q^{\frac{1}{2}})$ equals $|I_\slope(W)|$.
\end{lem}

\begin{proof}
The Bott--Solomon formula \eqref{eq:bott-solomon} can be rewritten
\begin{align}
\bm{s}(1_{\Q}) = \prod_{n \geq 1}
\Phi_n(\Q)^{|\{i \in I(W)\,:\,\text{$n$ divides $d_i$}\}|},
\end{align}
where $\Phi_n$ is the minimal polynomial of $e^{\frac{2\pi i}{n}}$ over $\bb{Q}_W$.
\end{proof}

\begin{lem}\label{lem:defect-invariants}
The $\zeta$-defect of any $\zeta$-block of $H_W(\Q^{\frac{1}{2}})$ is bounded above by $I_\slope(W)$.
Equality holds if and only if $\UDeg_\phi(\zeta) \neq 0$.
Moreover, if $|I_\slope(W)| \leq 1$, then there are two kinds of blocks:
\begin{enumerate}
\item 	Any block of defect $0$ is a singleton.
		For $\phi$ in such a block, $\UDeg_\phi(\zeta) = 0$.
\item 	Any block of defect $1$ is a line of characters ordered by $\bb{a}$-value.
		If
		\begin{align}
		\{\phi_0 > \phi_1 > \cdots > \phi_{n - 1}\}
		\end{align}
		is such a block, then $\UDeg_{\phi_k}(\zeta) = (-1)^k$ 	for all $k$.
\end{enumerate}
\end{lem}

\begin{proof}
The only claim that does not follow from Theorem \ref{thm:geck} and Lemma \ref{lem:defect-invariants-principal} is the claim about generic degrees in case (2).
Recall that, by definition, 
\begin{align}
\UDeg_{\phi_k}(\zeta) 
=
	\left.\frac{\bm{s}(1_{\Q})}{\bm{s}(\phi_{k, \Q})}\right|_{\Q^{1/2} \to \zeta^{1/2}}.
\end{align}
By work of Brou\'e, such ratios of Schur elements at roots of unity are Morita invariants of the block ideal \cite[\S{3.7}]{broue}.
(Strictly speaking, Brou\'e worked with group algebras, not Iwahori--Hecke algebras, but see the footnote on \cite[286]{cm}.)
Thus, by Theorem \ref{thm:geck}(2), these ratios match the corresponding ratios for the Brauer tree algebra of type $A_{n - 1}$.
In this way, we reduce to the setting where $W = S_n$ and $\slope = \frac{1}{n}$.
Here, we have $\phi_k = \Alt^k(\bb{V})$, so it remains to show that 
\begin{align}
\UDeg_{\Alt^k(\bb{V})}(e^{\frac{2\pi i}{n}}) = (-1)^k.
\end{align}
Let $w$ be an $e^{\frac{2\pi i}{n}}$-regular element of $W$.
By \cite[Rem.\ 3.7]{jones}, the right-hand side above equals $\tr(w \mid \Alt^k(V))$.
By Lemma-Postulate \ref{lempost:fourier}(2) and Theorem \ref{thm:fake-deg}, the left-hand side also equals $\tr(w \mid \Alt^k(V))$.
\end{proof}

We return to the proof of Theorem \ref{thm:slopes}(5).
Since $\slope$ is cuspidal, Theorem \ref{thm:slopes}(3) and Lemma \ref{lem:defect-invariants} together imply:
\begin{align}
[\Omega_\slope]_{\Q} 
=
	\sum_{\Gamma = \{\phi_k^\Gamma\}_k} 
	\sum_k
	(-1)^k
	[\Delta_\slope(\phi_k^\Gamma)]_{\Q},
\end{align}
where the outer sum runs over $\zeta$-blocks of $H_W(\Q^{\frac{1}{2}})$ of defect $1$.
To conclude, we claim that the inner sum on the right-hand side simplifies to $[L_\slope(\phi_0^\Gamma)]_{\Q}$, where $\phi_0^\Gamma$ is the element of greatest $\bb{a}$-value in the block $\Gamma$.
Indeed, this follows from the theorem below, itself a corollary of Theorem \ref{thm:ggor}(2).

\begin{thm}\label{thm:bgg}
Suppose that $\{\phi_0 > \phi_1 > \cdots > \phi_{n - 1}\}$ is a $\zeta$-block of $H_W(\Q^{\frac{1}{2}})$, ordered by $\bb{a}$-value.
Then, in $\sf{O}_\slope$, there is a ``Bernstein--Gelfand--Gelfand'' resolution of $L_\slope(\phi_0)$ of the form
\begin{align}
\Delta_\slope(\phi_{n - 1}) \to \cdots \to \Delta_\slope(\phi_1) \to \Delta_\slope(\phi_0) \to L_\slope(\phi_0) \to 0.
\end{align}
In particular, we have
\begin{align}
[L_\slope(\phi_0)]_{\Q} = \sum_{0 \leq k \leq n - 1} {(-1)^k} [\Delta_\slope(\phi_k)]_{\Q}
\end{align}
in $R(W)[\![\Q]\!]$.
\end{thm}

\begin{rem}
We believe that the first appearance of this theorem in the literature is \cite[Thm.\ 5.15]{rouquier_2008}.
The type $A$ case was used much earlier, in \cite{beg_imrn}.
\end{rem}

\subsection{Cuspidal Slopes, \emph{cont.}}

We prove Corollary \ref{cor:cuspidal}.

There are only two cases where $W$ is an irreducible Coxeter group, $\slope$ is a cuspidal slope, and there is a $\zeta$-block of $H_W(\Q^{\frac{1}{2}})$ of defect $1$ besides the principal block.
Namely, $W$ must be of type $E_8$ or type $H_4$ and the denominator of $\slope$ in lowest terms must be $15$. 

In these cases, there is one other $\zeta$-block of defect $1$ and its element of greatest $\bb{a}$-value is the character of $\bb{V}$, assuming as before that $\bb{V}^W = 0$.
We refer to these as the two \dfemph{exceptional} cases for $(W, \slope)$.
Parts (3)-(4) of Theorem \ref{thm:slopes} now imply the following stronger form of Corollary \ref{cor:cuspidal}:

\begin{cor}
If $W$ is irreducible and $\slope$ is cuspidal, then:
\begin{align}\label{eq:cuspidal}
[\Omega_\slope]_{\Q} 
= 
	\left\{\begin{array}{ll}
	{[L_\slope(1)]_{\Q} + [L_\slope(\bb{V})]_{\Q}}
		&\text{$(W, \slope)$ is exceptional}\\
	{[L_\slope(1)]_{\Q}}
		&\text{else}
	\end{array}\right.
\end{align}
in $R(W)[\Q]$.
In particular, the right-hand side is finite-dimensional.
\end{cor}

\begin{rem}
This statement also recovers the observation of Bezrukavnikov--Etingof \cite[Rem.\ 3.31]{be} that for irreducible $W$ and cuspidal $\slope$, the only finite-dimensional simple $\bb{D}_\slope^\rat$-modules are those appearing on the right-hand side of \eqref{eq:cuspidal}.
\end{rem}

\begin{rem}
If $W \simeq W_1 \times W_2$ and $\bb{V} \simeq \bb{V}_1 \oplus \bb{V}_2$, such that $\bb{V}_i$ is a realization of $W_i$ for $i = 1, 2$, then:
\begin{itemize}
\item 	$|I_\slope(W)| = |I_\slope(W_1)| + |I_\slope(W_2)|$, where the right-hand side is defined in terms of $\bb{V}_1$ and $\bb{V}_2$ rather than $\bb{V}$.
\item 	Category $\sf{O}$ for the rational Cherednik algebra of $W$ is the tensor product of the corresponding categories for $W_1$ and $W_2$.
\end{itemize}
So the above result can be generalized to any pair $(W, \slope)$ such that $\slope$ is cuspidal with respect to each irredcible factor of $W$.
\end{rem}

\subsection{}

To conclude this section, we explain how to deduce \cite[Thm.\ 1.1]{gors} from Corollaries \ref{cor:cuspidal} and \ref{cor:homfly}.

For $W = S_n$, we have $r = n - 1$ and $|\pi| = n(n - 1)$.
Let $s_i \in S_n$ be the transposition $(i~i + 1)$.
Then the (positive) $(m, n)$ torus knot is the link closure of 
\begin{align}
\beta = (\sigma_{s_1} \cdots \sigma_{s_{n - 1}})^m \in \Br_W^+,
\end{align}
a periodic braid of slope $\frac{m}{n}$.
Since $n$ is the Coxeter number of $S_n$, it is a cuspidal number for $S_n$.
Corollary \ref{cor:cuspidal} gives
\begin{align}\begin{split}
{[L_\slope(1)]_{\Q}}
	= {[\Omega_\slope]_{\Q}}
	&= (\Q^{\frac{1}{2}})^{r - |\beta|} \cdot \Tr{\beta}^0.
\end{split}\end{align}
Since $|\beta| - n + 1 = \frac{m}{n}|\pi| - (n - 1) = (m - 1)(n - 1)$, Corollary \ref{cor:homfly} gives
\begin{align}
{[\hat{\beta}]_{a, \Q}}
=	a^{(m - 1)(n - 1)}
	\sum_{0 \leq i \leq n - 1}
	{(-a^2)^i}
	(\Alt^i(\bb{V}), {[L_\slope(1)]_{\Q}})_W.
\end{align}
This is \cite[Thm.\ 1.1]{gors}.

\newpage

\appendix

\section{Coxeter Groups and Their Characters}\label{sec:coxeter}

\subsection{}\label{subsec:coxeter-group}

Our main reference for this appendix is the book \cite{gp}.

A \dfemph{Coxeter group} is a group $W$ that admits a finite presentation of the form
\begin{align}
W = \langle s \in S : (st)^{m_{s, t}} = 1\rangle,
\end{align}
where $m_{s,s} = 1$ and $m_{s, t}  \geq 1$ for all $s, t \in S$.
We say that $W$ is \dfemph{irreducible} iff it is both nontrivial and not a product of smaller Coxeter groups.
The pair $(W, S)$ is called a \dfemph{Coxeter system}, and $S$ is called its set of \dfemph{simple reflections}.
We say that $|S|$ is the \dfemph{rank} of $W$.

If $w \in W$, then an \dfemph{expression} for $w$ is a sequence $(s_1, \ldots, s_\ell)$ of elements of $S$ such that $w = s_1 \cdots s_\ell$.
The \dfemph{reduced} expressions for $w$ are those of minimal length.
The \dfemph{Bruhat order} on $W$ is the partial order where $w' \leq w$ iff some reduced expression for $w'$ is a subsequence of a reduced expression for $w$, not necessarily a contiguous one.
The common length of the reduced expressions for $w$ is called the \dfemph{Bruhat length} of $w$ and denoted $|w|$.

If $W$ is finite, then there is a unique element $w_0 \in W$ of maximal Bruhat length, known as the \dfemph{longest element} of $(W, S)$ \cite[\S{1.5}]{gp}.

\subsection{}\label{subsec:braid-group}

The \dfemph{Artin braid monoid} of $(W, S)$, which we denote $\Br_W^+$, is the monoid freely generated by formal elements $\sigma_s$ for each $s \in S$, modulo the relations 
\begin{align}
\overbrace{\sigma_s \sigma_t \sigma_s \cdots}^{\text{$m_{s, t}$ terms}}
=
\overbrace{\sigma_t \sigma_s \sigma_t \cdots}^{\text{$m_{s, t}$ terms}}
\end{align}
for \emph{distinct} $s, t \in S$.
The \dfemph{Artin braid group} of $(W, S)$, denoted $\Br_W$, is the \emph{group} with the same presentation.
By \cite[Prop.\ 4.1.10]{gp}, the natural morphism $\Br_W^+ \to \Br_W$ is injective.
The elements of $\Br_W$ are called \dfemph{braids}.

There is a surjective morphism $\Br_W \to W$ that sends $\sigma_s \mapsto s$.
The kernel is generated by the elements of the form $\sigma_s^2$.

\begin{ex}
In the Coxeter system of type $A_{n - 1}$, the group $W$ is the symmetric group $S_n$.
We can take $S$ to consist of the transpositions $(i~i + 1)$ for $1 \leq i < n$.
We have $\Br_W \simeq \Br_n$, the group of topological braids on $n$ strands.
\end{ex}

If $\beta \in \Br_W^+$, then an \dfemph{expression} for $\beta$ is a sequence $(s_1, \ldots, s_\ell)$ of elements of $S$ such that $\beta = \sigma_{s_1} \cdots \sigma_{s_\ell}$.
All expressions for $\beta$ have the same length, which is called the \dfemph{writhe} of $\beta$ and denoted $|\beta|$.
Unlike Bruhat length, writhe is additive:
It defines a morphism of monoids $|{-}| : \Br_W^+ \to \bb{Z}_{\geq 0}$, which extends uniquely to a morphism of groups $|{-}| : \Br_W \to \bb{Z}$.

If $w \in W$ and $(s_1, \ldots, s_\ell)$ is a reduced expression for $w$, then $\sigma_w = \sigma_{s_1} \cdots \sigma_{s_\ell}$ depends only on $w$, not on the chosen expression.
The map $w \mapsto \sigma_w$ defines a set-theoretic right inverse of the morphism $\Br_W^+ \to W$ such that $|\sigma_w| = |w|$.

\subsection{}\label{subsec:hecke}

The \dfemph{Iwahori--Hecke algebra} of $(W, S)$ is a variant of the group ring $\bb{Z}[W]$ where we modify the relations $s^2 = 1$ to depend on a variable $\Q^{\frac{1}{2}}$.
Namely, it is
\begin{align}
H_W = \frac{\bb{Z}[\Q^{\pm\frac{1}{2}}][\Br_W^+]}{\left\langle (\sigma_s - \Q^{\frac{1}{2}})(\sigma_s + \Q^{-\frac{1}{2}}) : s \in S\right\rangle}.
\end{align}
The set $\{\sigma_w\}_{w \in W}$ forms a free $\bb{Z}[\Q^{\pm\frac{1}{2}}]$-basis for $H_W$.
By induction on the Bruhat order, one can show that the multiplication table for this basis is given by
\begin{align}\label{eq:hecke-multiplication}
\sigma_w \sigma_s =
\left\{\begin{array}{ll}
\sigma_{ws}
	&ws > w\\
\sigma_{ws} + (\Q^{\frac{1}{2}} - \Q^{-\frac{1}{2}})\sigma_w 
	&ws < w
\end{array}\right.
\end{align}
for all $s \in S$ and $w \in W$.
Thus, $\sigma_w \mapsto w$ defines an explicit isomorphism $H_W|_{\Q^{1/2} = 1} \simeq \bb{Z}[W]$.

\subsection{}\label{subsec:parabolic}

We say that a subgroup $W' \subseteq W$ is \dfemph{$S$-parabolic} iff it is generated by a subset of $S$.
In this case, the inclusion $W' \to W$ induces a writhe-preserving inclusion of positive braid monoids $\Br_{W'}^+ \to \Br_W^+$ by way of the map $w \mapsto \sigma_w$.
This, in turn, induces an inclusion of Iwahori--Hecke algebras $H_{W'} \to H_W$.

\subsection{}\label{subsec:representation}

All representations of $W$ in this paper are finite-dimensional linear representations over (a subfield of) $\bb{C}$ or an isomorphic field.
We write $R(W)$ for the Grothendieck ring of virtual representations of $W$, and we write $\hat{W} \subseteq R(W)$ for the subset of irreducible representations.
Abusing notation, \emph{we treat the elements of $R(W)$ interchangeably with their characters}.
Let $\bb{Q}_W \subseteq \bb{C}$ be the field obtained by adjoining to $\bb{Q}$ the character values $\phi(w)$ for all $w \in W$ and $\phi \in \hat{W}$.

We write $1$ and $\varepsilon$ for the trivial and sign characters of $W$, respectively.
By definition, $\varepsilon(w) = (-1)^{|w|}$.

Let $(-, -)_W : R(W) \times R(W) \to \bb{Z}$ be the pairing defined by
\begin{align}\begin{split}
(\phi, \psi)_W 
&= \dim \Hom(1, \phi \otimes \psi)\\
&= \frac{1}{|W|}
\sum_{w \in W} \phi(w)\psi(w^{-1})
\end{split}\end{align}
for all $\phi, \psi \in R(W)$.
For any ring $A$, we extend $(-, -)_W$ to a pairing on $R(W) \otimes A$ by linearity.

\subsection{}\label{subsec:realization}

If $K$ is a subfield of $\bb{C}$ or a field isomorphic to $\bb{C}$, then a \dfemph{realization} of $(W, S)$ over $K$ consists of a finite-dimensional $K$-vector space $\bb{V}$ and a (faithful) representation $W \to \GL(\bb{V})$ that identifies $S$ with a set of reflections in $\bb{V}$.
It will be convenient to set
\begin{align}
{[\Sym(\bb{V})]_{\Q}} = \sum_i {\Q^i} \Sym^i(\bb{V}) \in R(W)[\![\Q]\!].
\end{align}
If $\bb{V}^W = 0$, then $[\Sym(\bb{V})]_{\Q}$ does not depend on the choice of $\bb{V}$.
In this case,
\begin{align}
\bm{m}_\phi(\Q) = (\phi, [\Sym(\bb{V})]_{\Q})_W
\end{align}
is known as the \dfemph{Molien series} of $\phi \in \hat{W}$.

Every finite Coxeter system admits a realization over $\bb{R}$.
We say that $(W, S)$ is \dfemph{crystallographic} iff it admits a realization $\bb{V}$ over $\bb{Q}$, in which case $W$ is the Weyl group of a root system $\Phi \subseteq \bb{V}$ and $S$ is the set of reflections corresponding to a system of simple roots in $\Phi$.
For crystallographic Coxeter systems, $\bb{Q}_W = \bb{Q}$ \cite[Cor.\ 4.8]{springer_1974}.

\subsection{}\label{subsec:hecke-generic}

We will abbreviate
\begin{align}
H_W(\Q^{\frac{1}{2}}) = \bb{Q}_W(\Q^{\frac{1}{2}}) \otimes_{\bb{Z}[\Q^{\pm\frac{1}{2}}]} H_W.
\end{align}
By \cite[Thm.\ 9.3.5]{gp}, we have an isomorphism of $\bb{Q}_W(\Q^{\frac{1}{2}})$-algebras $H_W(\Q^{\frac{1}{2}}) \simeq \bb{Q}_W(\Q^{\frac{1}{2}})[W]$.
In particular, $H_W(\Q^{\frac{1}{2}})$ is semisimple.
Writing $R(H_W(\Q^{\frac{1}{2}}))$ for the Grothendieck ring of $H_W(\Q^{\frac{1}{2}})$-modules, we obtain a ring isomorphism:
\begin{align}
\begin{array}{rcl}
R(W) &\xrightarrow{\sim} &R(H_W(\Q^{\frac{1}{2}}))\\
\phi &\mapsto &\phi_{\Q}
\end{array}
\end{align}
Again, we will identify the representations $\phi_{\Q}$ with their characters.
Thus we will write $\phi_{\Q}(\beta) = \tr(\beta \mid \phi_{\Q})$ for all $\beta \in H_W(\Q^{\frac{1}{2}})$ and $\phi \in \hat{W}$.
For example, we have $1_{\Q}(\sigma_w) = (\Q^{\frac{1}{2}})^{|w|}$ for all $w \in W$.

Let $\tau_{\Q} : H_W(\Q^{\frac{1}{2}}) \to \bb{Q}_W(\Q^{\frac{1}{2}})$ be the symmetrizing trace defined by
\begin{align}
\tau_{\Q}(\sigma_w) = \left\{
\begin{array}{ll}
1	&w = 1\\
0	&w \neq 1
\end{array}\right.
\end{align}
The corresponding symmetrizer is $\sum_{w \in W} \sigma_w \otimes \sigma_{w^{-1}} \in H_W \otimes H_W$.
By Schur orthogonality \cite[Cor.\ 7.2.4]{gp}, we can write 
\begin{align}\label{eq:schur}
\tau_{\Q} = \sum_{\phi \in \hat{W}}
\frac{1}{\bm{s}(\phi_{\Q})}\,\phi_{\Q},
\end{align}
where the Schur element $\bm{s}(\phi_{\Q}) \in \bb{Q}_W(\Q^{\frac{1}{2}})$ is defined by
\begin{align}
\bm{s}(\phi_{\Q}) = \frac{1}{\phi(1)} \sum_{w \in W} \phi_{\Q}(\sigma_w)\phi_{\Q}(\sigma_{w^{-1}}).
\end{align}
In particular, $\bm{s}(1_{\Q}) = \sum_{w \in W} \Q^{|w|}$, the so-called Poincar\'e polynomial of $W$.

\subsection{}\label{subsec:degrees}

To each $\phi \in \hat{W}$, we attach two polynomial invariants.

\subsubsection{}

The \dfemph{fake degree} of $\phi$ is the ratio
\begin{align}
\FDeg_\phi(\Q) = \frac{\bm{m}_\phi(\Q)}{\bm{m}_1(\Q)}.
\end{align}
It follows from Chevalley's restriction theorem that $\FDeg_\phi(\Q) \in \bb{Z}[\Q]$ \cite[\S{2.5-2.6}]{springer_1974}.

\subsubsection{}

The \dfemph{generic} or \dfemph{unipotent degree} of $\phi$ is the ratio
\begin{align}
\UDeg_\phi(\Q) = \frac{\bm{s}(1_{\Q})}{\bm{s}(\phi_{\Q})}.
\end{align}
Benson--Curtis showed \cite[Cor.\ 9.3.6, Rem.\ 9.3.7]{gp} that $\UDeg_\phi(\Q) \in \bb{Q}_W[\Q]$.

\subsection{}\label{subsec:fourier}

In \cite{lusztig_1984, lusztig_1993, lusztig_1994}, Lusztig assigned to every finite Coxeter group $W$ a finite set $\UCh(W)$, equipped with:
\begin{enumerate}
\item 	An embedding $\hat{W} \to \UCh(W)$ that we denote $\phi \mapsto \rho_\phi$.
\item 	A function $\UCh(W) \to \bb{Q}_W[\Q]$ that we denote $\rho \mapsto \UDeg_\rho(\Q)$.
		It satisfies $\UDeg_{\rho_\phi}(\Q) = \UDeg_\phi(\Q)$ for all $\phi \in \hat{W}$. 
\item 	A symmetric orthogonal pairing
\begin{align}
\{-, -\} : \UCh(W) \times \UCh(W) \to \bb{Q}_W
\end{align}
known as the \dfemph{exotic Fourier transform}.
\end{enumerate}
In type $H_4$, his work was completed by Malle \cite{malle}.

For $W$ the Weyl group of a finite group of Lie type, Lusztig used this data to relate almost-characters parametrized by $\hat{W}$ with unipotent irreducible characters parametrized by $\UCh(W)$ \cite{lusztig_1984}, \emph{cf.}\ Section \ref{sec:dl}.
For general $W$, Lusztig gave a list of postulates that characterize this data uniquely \cite[\S{2}]{lusztig_1993}.

Outside of Section \ref{sec:dl}, we only use the restriction of the exotic Fourier transform to a pairing on $\hat{W}$: that is, the pairing
\begin{align}
\{-, -\} : \hat{W} \times \hat{W} \to \bb{Q}_W.
\end{align}
defined by $\{\phi, \psi\} = \{\rho_\phi, \rho_\psi\}$ for all $\phi, \psi \in \hat{W}$.

\begin{lempost}\label{lempost:fourier}
Let $W$ be any finite Coxeter group.
\begin{enumerate}
\item 	
		If $W= W_1 \times W_2$, then 
		\begin{align}
		\UCh(W) = \UCh(W_1) \times \UCh(W_2)
		\end{align}
		and $\{-, -\}$ is similarly multiplicative.
\item 	
		For all $\phi \in \hat{W}$, we have
\begin{align}
\UDeg_\phi(\Q) = \sum_{\psi \in \hat{W}} {\{\phi, \psi\}} \FDeg_\psi(\Q).
\end{align}
\item 	
		For all $\phi, \psi \in \hat{W}$, we have $\{\varepsilon \phi, \varepsilon \psi\} = \{\phi, \psi\}$.
\end{enumerate}
\end{lempost}

\begin{proof}
Part (1) is \cite[Postulate 2.3]{lusztig_1993}.
Part (2) is \cite[\S{13.6}]{carter}.
Part (3) does not seem to be stated in the literature directly, but for Weyl groups, it is a consequence of Alvis--Curtis--Kawanaka duality \cite[\S{8.2}]{carter}.
For general $W$, we can check it by using part (1) to reduce to the irreducible case.
\end{proof}

\begin{rem}
In \cite[Thm.\ 4.3]{marberg} and the remark below it, Marberg claims that in types $E_7$ and $E_8$, the identity in part (2) only holds up to precomposition with a certain involution described in \cite{bl}.
Cross-referencing against the tables in \cite[\S{13.9}]{carter} and \cite[Ch.\ 4]{lusztig_1984} shows that this claim is incorrect.
\end{rem}

Recall the character decomposition \eqref{eq:schur} of the symmetrizing trace $\tau_{\Q}$.
Using Lemma-Postulate \ref{lempost:fourier}(2), we can re-express these weights in terms of Molien series and the exotic Fourier transform.

\begin{prop}\label{prop:molien}
For all $\phi \in \hat{W}$, we have 
\begin{align}
\frac{1}{\bm{s}(\phi_{\Q})}
=
(1 - \Q)^r \sum_{\psi \in \hat{W}}
{\{\phi, \psi\}}
\bm{m}_\psi(\Q)
\end{align}
in $\bb{Q}_W(\Q^{\frac{1}{2}})$.
\end{prop}

\begin{proof}
Let $r = |S|$.
By Lemma-Postulate \ref{lempost:fourier}(2), we must prove $\bm{s}(1_{\Q}) \bm{m}_1(\Q) = (1 - \Q)^{-r}$.
Indeed, if $d_1, \ldots, d_r$ are the invariant degrees of the $W$-action on $\bb{V}$ \cite[149]{gp}, then we have
\begin{align}
\bm{s}(1_{\Q}) = \prod_{1 \leq i \leq r} 
\frac{1 - \Q^{d_i}}{1 - \Q}
\quad\text{and}\quad
\bm{m}_1(\Q) = \prod_{1 \leq i \leq r}
{\frac{1}{1 - \Q^{d_i}}}.
\end{align}
The left-hand side is the Bott--Solomon formula for the Poincar\'e polynomial of $W$ \cite{solomon}.
The right-hand side is true by definition.
\end{proof}

\subsection{}\label{subsec:families}

In \cite{lusztig_1984}, Lusztig introduced functions $\bb{a}, \bb{A} : \hat{W} \to \bb{Z}_{\geq 0}$ defined by
\begin{align}
\bb{a}(\phi) &= \val_{\Q} \UDeg_\phi(\Q),\\
\bb{A}(\phi) &= \deg_{\Q} \UDeg_\phi(\Q).
\end{align}
Using induction on parabolic subgroups of $W$, he also described a partition of $\hat{W}$ into subsets called \dfemph{families}.
The functions $\bb{a}$ and $\bb{A}$ are constant in families, and the matrix of the pairing $\{-, -\}$ is block-diagonal with respect to them.

\begin{ex}
In type $A$, it turns out that $\hat{W} \to \UCh(W)$ is bijective, the matrix of $\{-, -\}$ is the identity matrix, and every family in $\hat{W}$ is a singleton.
\end{ex}

We use $\bb{a}$ directly, whereas we only use $\bb{A}$ by way of the auxiliary function that follows.
Let $\ur{Ref}(W)$ be the set of elements of $W$ that act on any realization by reflections.
(It is the union of the conjugacy classes that meet $S$.)
Motivated by the theory of partitions, we introduce:

\begin{df}
For all $\phi \in \hat{W}$, the \dfemph{content} of $\phi$ is the number
\begin{align}
\bb{c}(\phi) = \frac{1}{\phi(1)} \sum_{t \in \ur{Ref}(W)} \phi(t).
\end{align}
Note that $\bb{c}(\bb{1}) = |\ur{Ref}(W)|$.
If $W$ is the Weyl group of a root system $\Phi$, then this number equals $\frac{1}{2} |\Phi|$.
\end{df}

If $W = S_n$, then there is a bijective correspondence between $\hat{W}$ and the set of integer partitions of $n$.
Under this bijection, $\bb{c}(\phi)$ is precisely the content of the partition attached to $\phi$.
For general $W$, the following result about $\bb{c}$ is \cite[\S{4.21}]{bm_1997}, which relies on \cite[\S{5.12}]{lusztig_1984}:

\begin{lem}[Brou\'e--Michel]
We have $\bb{c}(\phi) = N - \bb{a}(\phi) - \bb{A}(\phi)$ for all $\phi \in \hat{W}$.
\end{lem}

\begin{cor}\label{cor:content}
The function $\bb{c}$ is constant in families.
\end{cor}

\subsection{}\label{subsec:cyclotomic}

To conclude this appendix, we review the block theory of Iwahori--Hecke algebras at roots of unity, relying on \cite[Ch.\ 7]{gp}.

Given any ring $H$, we write $R(H)$ for the Grothendieck group of finite-dimensional $H$-modules.
The relations on $R(H)$ are $[M] = [M'] + [M'']$ for each short exact sequence $0 \to M' \to M \to M'' \to 0$.
We write $R^+(H) \subseteq R(H)$ for the semiring of actual, not virtual, modules.
 
\subsubsection{}

Let $E$ be a field of characteristic zero, and let $A \subseteq E$ be a subring.
Let $H$ be a free $A$-algebra of finite rank such that the $E$-algebra
\begin{align}
H_E = H \otimes_A E
\end{align}
is split semisimple.
Let $k$ be another field and $\theta : A \to k$ a ring morphism.
We set 
\begin{align}
H_\theta = H \otimes_A k.
\end{align}
We will review how Brauer's theory of decomposition maps gives a direct relation between the representation theories of $H_E$ and $H_\theta$, or more precisely, between $R(H_E)$ and $R(H_\theta)$, as long as $A$ is integrally closed in $E$.

Let $t$ be an indeterminate over $K$.
We have a map: 
\begin{align}
\begin{array}{rcl}
R^+(H_E)
	&\xrightarrow{\bm{p}} 
	&\prod_H {(1 + tK[t])}\\[1ex]
M
	&\mapsto
	&\{\det_E(1 - t\alpha \mid M)\}_{\alpha \in H}
\end{array}
\end{align}
The Brauer--Nesbitt theorem \cite[Lem.\ 7.3.2]{gp} implies that $\bm{p}$ is injective.
Furthermore, \cite[ Prop.\ 7.3.8]{gp} states that if $A$ is integrally closed in $E$, then $\bm{p}$ factors through $(1 + tA[t])^H$.

Let $\bm{p}_\theta : R^+(H_\theta) \to (1 + k[t])^H$ be defined by analogy to $\bm{p}$.
Brauer's theorem \cite[Thm.\ 7.4.3]{gp} states that if $A$ is integrally closed in $E$, then there is a unique additive map $\bm{d}_\theta : R^+(H_E) \to R^+(H_\theta)$ making the following diagram commute:
\begin{equation}
\begin{tikzpicture}[baseline=(current bounding box.center), >=stealth]
\matrix(m)[matrix of math nodes, row sep=2.5em, column sep=2.5em, text height=2ex, text depth=0.25ex]
{	R^+(H_E)
		&(1 + tA[t])^H\\
	R^+(H_\theta)
		&(1 + tk[t])^H\\
		};
\path[->,font=\small, auto]
(m-1-1)	edge node{$\bm{p}$} (m-1-2)
		edge node[left]{$\bm{d}_\theta$} (m-2-1)
(m-1-2)	edge node[right]{$\theta$} (m-2-2)
(m-2-1)	edge node{$\bm{p}_\theta$} (m-2-2);
\end{tikzpicture}
\end{equation}
In particular, the map $\bm{d}_\theta$ is compatible with the formation of characters from modules.
It is called the \dfemph{decomposition map} of $\theta$.
Note that if $\phi \in R^+(H_E)$ is irreducible, then $\bm{d}_\theta(\phi) \in R^+(H_\theta)$ is usually no longer irreducible.

Let $\hat{H}_E \subseteq R^+(H_E)$ be the set of isomorphism classes of simple modules.
The \dfemph{Brauer graph} of $H_\theta$ is the (undirected) graph in which:
\begin{enumerate}
\item 	The vertex set is $\hat{H}_E$.
\item 	There is an edge between $M$ and $M'$ iff there is a simple module shared in common by the decompositions of $\bm{d}_\theta(M)$ and $\bm{d}_\theta(M')$ in $R^+(H_\theta)$.
\end{enumerate}
The connected components of this graph are called the \dfemph{$\theta$-blocks} of $H_E$.
The partition of $\hat{H}_E$ into blocks corresponds to a direct-sum decomposition of $H_\theta$ into idempotent ideals called \dfemph{block ideals}, such that the simple modules over any block ideal are in bijection with the characters in the corresponding block.

\subsubsection{}

Fix a number $\slope \in \bb{Q}$.
Let $\zeta^{\frac{1}{2}} = e^{\pi i\slope}$, and let
\begin{align}
H_W(\zeta^{\frac{1}{2}}) = \bb{Q}_W(\zeta^{\frac{1}{2}}) \otimes_{\bb{Z}[\Q^{\pm\frac{1}{2}}]} H_W,
\end{align}
where the map $\bb{Z}[\Q^{\pm\frac{1}{2}}] \to \bb{Q}_W(\zeta^{\frac{1}{2}})$ sends $\Q^{\frac{1}{2}} \mapsto \zeta^{\frac{1}{2}}$.
Let $\bb{Z}_W$ be the ring of integers of $\bb{Q}_W$. 
Then in the notation above, we can take
\begin{align}
E &= \bb{Q}_W(\Q^{\frac{1}{2}}),\\
A &= \bb{Z}_W[\Q^{\pm\frac{1}{2}}],\\
H &= H_W \otimes_{\bb{Z}} \bb{Z}_W,\\
k &= \bb{Q}_W(\zeta^{\frac{1}{2}})).
\end{align}
Let $\theta : A \to k$ be the specialization map $\Q^{\frac{1}{2}} \mapsto \zeta^{\frac{1}{2}}$.
These choices give $H_E = H_W(\Q^{\frac{1}{2}})$ and $H_\theta = H_W(\zeta^{\frac{1}{2}})$.
So, Brauer's theorem yields a decomposition map:
\begin{align}
\bm{d}_\zeta \vcentcolon= \bm{d}_\theta : R^+(H_W(\Q^{\frac{1}{2}})) \to R^+(H_W(\zeta^{\frac{1}{2}})).
\end{align}
We will refer to the $\theta$-blocks of $H_W(\Q^{\frac{1}{2}})$ as \dfemph{$\zeta$-blocks}.
The \dfemph{principal block} is the block that contains the trivial character $1$.

\subsection{}

By \cite[Prop.\ 7.3.9]{gp}, we have 
\begin{align}
\bm{s}(\phi_{\Q}) \in \bb{Z}_W[\Q^{\pm\frac{1}{2}}]
\end{align}
for all $\phi \in \hat{W}$.
The \dfemph{$\zeta$-defect} of $\phi$ is the multiplicity with which $\Q^{\frac{1}{2}} = \zeta^{\frac{1}{2}}$ occurs as a root of the Laurent polynomial $\bm{s}(\phi_{\Q})$ (after we extend scalars).
This number measures the complexity of the $\zeta$-block containing $\phi$ in the following sense:

\begin{thm}[Geck]\label{thm:geck}
The characters within a given $\zeta$-block have the same $\zeta$-defect, so we can speak of the defect of the block itself.
Moreover:
\begin{enumerate}
\item 	If the defect is $0$, then the block is a singleton.
\item 	If the defect is $1$, then the block is a line of characters ordered by $\bb{a}$-value.
		Its block ideal is isomorphic to a Brauer tree algebra of type $A$ (see below).
\end{enumerate}
\end{thm}

\begin{proof}
This restates Prop.~7.4, Prop.~8.2, and Thm.~9.6 of \cite{geck}.
\end{proof}

\begin{rem}
By contrast, we do not know of a structure theorem for blocks of defect $\geq 2$.
\end{rem}

\begin{ex}
Suppose that we are in type $A_{n - 1}$, meaning $W \simeq S_n$, and that $\slope = \frac{m}{n}$ with $m$ coprime to $n$.
Here, the principal $\zeta$-block is a line.
If $\bb{V}$ is the standard representation of $S_n$, then the principal block consists of the exterior powers $\Alt^i(\bb{V})$, ordered along the line from $i = 0$ to $i = n - 1$.
We refer to the associated block ideal as the \dfemph{Brauer tree algebra of type $A_{n - 1}$}.
All other blocks are singletons.
\end{ex}

\newpage
\section{Braid Varieties in the Literature}\label{sec:varieties-other}

\subsection{}

In this appendix, we explain how the varieties $\cal{U}(\beta), \cal{Z}(\beta), \cal{X}(\beta)$ are related to other varieties that have appeared in the literature.
Throughout, we assume the notation and hypotheses of Sections \ref{sec:mixed}-\ref{sec:hecke}. 
All varieties in this section can be constructed over $\bb{F}$, not just $\bar{\bb{F}}$, but we will omit the subscript $0$'s anyway to lighten notation.

\subsection{}

Let $w \in W$.
As in Section \ref{sec:hecke}, let $O_w \subseteq \cal{B} \times \cal{B}$ be the $G$-orbit indexed by $w$, and let $\act : G \times \cal{B} \to \cal{B} \times \cal{B}$ be the map
\begin{align}
\act(g, B) = (g^{-1}Bg, B).
\end{align}
In the following diagram, all squares are cartesian:
\begin{equation}
\begin{tikzpicture}[baseline=(current bounding box.center), >=stealth]
\matrix(m)[matrix of math nodes, row sep=2.5em, column sep=3.5em, text height=2ex, text depth=0.5ex]
{	O_w
		&G_w
		&\cal{U}_w
		&\cal{Z}_w\\
	\cal{B} \times \cal{B}
		&G \times \cal{B}
		&\cal{U} \times \cal{B}
		&\tilde{\cal{U}} \times \cal{B}\\
		};
\path[->,font=\scriptsize, auto]
(m-1-2)		edge (m-1-1)
(m-1-3) 	edge (m-1-2)
(m-1-4) 	edge (m-1-3)
(m-1-1)		edge (m-2-1)
(m-1-2)		edge (m-2-2)
(m-1-3) 	edge (m-2-3)
(m-1-4) 	edge (m-2-4)
(m-2-2)		edge node[above]{$\act$} (m-2-1)
(m-2-3)		edge (m-2-2)
(m-2-4)		edge (m-2-3);
\end{tikzpicture}
\end{equation}
To the best of my knowledge:
\begin{enumerate}
\item 	For fixed $g \in G$ and $w$ a Coxeter element of minimal length, the fiber $(G_w)_g \subseteq G_w$ first appears \cite[\S{4}]{steinberg_1965}.
		For general $w$, it first appears in Lem.\ 3.6 of Kawanaka's paper \cite{kawanaka}, where it is denoted $F_{g, w}$, and independently in Lusztig's note \cite{lusztig_1980}, where it is denoted $Y_{s, w}$ in the case where $g$ is a regular semisimple element $s$.
		More recently, it has appeared in Lusztig's paper \cite{lusztig_2011}, where it is denoted $\cal{B}_g^w$, and in D.\ Kim's papers \cite{kim_2018_algebra, kim_2018_imrn}, where it is denoted $\cal{Y}_{w, g}$.
		
\item 	The variety $G_w$ first appears in Lusztig's paper \cite{lusztig_1985_1}, where it is denoted $Y_w$.
		It is used throughout his series of works on character sheaves.
		It also appears in \cite{lusztig_2011} and its sequels, where it is denoted $\fr{B}_w$.
		
\item 	The varieties $\cal{U}_w^\red$ and $\cal{Z}_w^\red$ first appear in Lusztig's note \cite{lusztig_2021}, where they are respectively denoted $\fr{B}_w^{(1)}$ and $\tilde{\fr{B}}_w^1$.

		More generally:
		Let $\cal{T}$ be the universal maximal torus of $G$, so that there are maps $G \to \cal{T} \sslash W$ and $\tilde{G} \to \cal{T}$.
		For any $\delta \in \cal{T}$ with image $(\delta) \in \cal{T} \sslash W$, Lusztig defines varieties $\fr{B}_w^{(\delta)} \subseteq G_w$ and $\tilde{\fr{B}}_w^\delta \subseteq \tilde{G} \times_G G_w$ as the corresponding fibers.
\end{enumerate}

\subsection{}

Let $\beta \in \Br_W^+$ be a positive braid of writhe $\ell$.
Thus we have $\beta = \sigma_{s_1} \cdots \sigma_{s_\ell}$ for some $s_1, \ldots, s_\ell \in S$.
By definition,
\begin{align}
O(\beta) = \{(B_0, B_1, \ldots, B_\ell) : \text{$(B_{i - 1}, B_i) \in O_{s_i}$ for $1 \leq i \leq \ell$}\}.
\end{align}
In the following diagram, all squares are cartesian:
\begin{equation}
\begin{tikzpicture}[baseline=(current bounding box.center), >=stealth]
\matrix(m)[matrix of math nodes, row sep=2.5em, column sep=3.5em, text height=2ex, text depth=0.5ex]
{	O(\beta)
		&G(\beta)
		&\cal{U}(\beta)
		&\cal{Z}(\beta)\\
	\cal{B} \times \cal{B}
		&G \times \cal{B}
		&\cal{U} \times \cal{B}
		&\tilde{\cal{U}} \times \cal{B}\\
		};
\path[->,font=\scriptsize, auto]
(m-1-2)		edge (m-1-1)
(m-1-3) 	edge (m-1-2)
(m-1-4) 	edge (m-1-3)
(m-1-1)		edge node[left]{$\pr_0 \times \pr_\ell$} (m-2-1)
(m-1-2)		edge (m-2-2)
(m-1-3) 	edge (m-2-3)			
(m-1-4) 	edge (m-2-4)
(m-2-2)		edge node[above]{$\act$} (m-2-1)
(m-2-3)		edge (m-2-2)
(m-2-4)		edge (m-2-3);
\end{tikzpicture}
\end{equation}
In addition, we will abbreviate
\begin{align}
\cal{X}(\beta) \vcentcolon= G(\beta)_1 = \cal{U}(\beta)_1.
\end{align}
To the best of my knowledge:
\begin{enumerate}
\item 	For general $\beta$, the variety $O(\beta)$ first appears in Brou\'e--Michel's paper \cite{bm_1997}, where it is denoted $\cal{B}(\beta)$.
		See also \cite[Application 2]{deligne_1997}.
		However, it is arguably implicit in the works of Bott--Samelson \cite{bs} and Hansen \cite{hansen}.
		In \S{6} of the paper \cite{stz} by Shende--Treumann--Zaslow, $O(\beta)$ is called the open Bott--Samelson variety of $\beta$ and denoted $\cal{M}_1(\beta)$.
\item 	The variety $G(\beta)$ first appears in \cite[\S{2.5}]{lusztig_1985_1}, where it is denoted $Y_{\mathbf{w}}$ in the case where $\mathbf{w}$ is an expression for $\beta$ (see \S\ref{subsec:braid-group}).
		Later, it implicitly appears in \cite[\S{6}]{stz} along with $\cal{X}(\beta)$:
		In the notation of \emph{loc.\ cit.},
		\begin{align}
		\cal{M}(\beta^\circ) &= [G(\beta)/G],\\
		\cal{M}(\beta^\succ) &= [\cal{X}(\beta\pi)/G].
		\end{align}
		(The authors assume $G = \PGL_n$, but this hypothesis is extraneous.)
		From the viewpoint of cluster geometry, the stack $\cal{M}(\beta^\succ)$ has been studied by L.\ Shen and D.\ Weng in \cite{sw}:
		It is the special case of their stack $\ur{Conf}_d^b(\cal{B})$ where $b = \beta$ and $d = \bb{1}$.
\item 	For fixed $B \in \cal{B}$, the variety $\cal{X}(\beta)_B^\red \subseteq \cal{X}(\beta)^\red$ first appears in \S{5.3} of Mellit's paper \cite{mellit} in an isomorphic form denoted $X_0(\beta)$, in the case where $G = \GL_n$.
		In the next subsection, we generalize Mellit's definition to arbitrary reductive $G$ and explain the isomorphism in detail.
		
		More recently, $X_0(\beta)$ has been studied by Casals, Gorsky, Gorsky, and Simental in the paper \cite{cggs}:
		It is the special case of their variety $X_0(\gamma; \pi)$ where $\gamma = \beta$ and $\pi = \bb{1}$.
		(Note that in their notation, $\pi$ is an element of $S_n$, not the full twist.)
\item 	The schemes $\cal{U}(\beta)$ and $\cal{Z}(\beta)$ are constructed similarly to the representation varieties $\cal{X}_\mathrm{sing}^1$ and $\cal{X}_\mathrm{sm}^1$ in \cite[\S{8}]{mellit}.
		However, they are not a special case of those varieties, or vice versa.
\end{enumerate}

\subsection{}

Below, we relate the scheme $\cal{X}(\beta)$ to the variety $X_0(\beta)$ in \cite{mellit}, hence to the varieties in \cite{cggs}.

Fix a maximal torus $T \subseteq G$ and a Borel $B \supseteq T$.
For each $s \in S$, we have a corresponding positive root subgroup
\begin{align}
u_{\alpha_s} : \bb{A}^1 \to U_{\alpha_s} \subseteq B.
\end{align}
Recall that $N_G(T)/T \simeq W$.
Fix a section $W \to N_G(T)$, to be denoted $w \mapsto w_T$.
For all $s \in S$, let $f_{\alpha_s} : \bb{A}^1 \to G$ be the map
\begin{align}
f_{\alpha_s}(z) = s_T \cdot u_{\alpha_s}(z).
\end{align}
Let $f_{\beta, B} : \bb{A}^\ell \to G$ be the map
\begin{align}
f_{\beta, B}(\vec{z}) = f_{\alpha_{s_1}}(z_1) \cdots f_{\alpha_{s_\ell}}(z_\ell).
\end{align}
With these notations, let 
\begin{equation}
X_0(\beta, B) = f_{\beta, B}^{-1}(B) \subseteq \bb{A}^\ell.
\end{equation}
For $G = \GL_n$ and $B$ fixed, the underlying reduced scheme of $X_0(\beta, B)$ is the variety denoted $X_0(\beta)$ in \cite[\S{5.3}]{mellit}.

\begin{ex}
Let $G = \GL_2$ and $B$ the upper-triangular subgroup.
In this case,
\begin{align}
f_{\alpha_s}(z) = \pmat{&1\\ 1&z}.
\end{align}
Let $\sigma$ be the positive generator of $\Br_W = \Br_2$, and let $\bb{G}$ be the punctured affine line.
Then we can check that:
\begin{align}
X_0(\sigma) &= \emptyset,\\
X_0(\sigma^2)
	&= \{\vec{z} \in \bb{A}^2 : z_1 = 0\} \\
	&= \point \times \bb{A}^1,\nonumber\\ 
X_0(\sigma^3) 
	&= \{\vec{z} \in \bb{A}^3 : 1 + z_1z_2 = 0\}\\
	& = \bb{G} \times \point \times \bb{A}^1,\nonumber\\
X_0(\sigma^4) 
	&= \{\vec{z} \in \bb{A}^4 : z_1 + z_3 + z_1z_2z_3 = 0\} \\
	&= (\bb{A}^1 \times \point \times \bb{A}^1) \sqcup (\bb{G} \times \bb{A}^1 \times \bb{G}). \nonumber
\end{align}
The argument in \cite[\S{5.4}]{mellit} shows that for general $G$ and $\beta$, it is always possible to stratify $X_0(\beta)$ into cells of the form $\bb{A}^a \times \bb{G}^b$ with $a, b \geq 0$.
\end{ex}

It would be desirable to organize the varieties $X_0(\beta, B)$ into an algebraic family over the flag variety $\cal{B}$.
I do not know how to do this directly, given the dependence of $f_{\beta, B}$ on the section $w \mapsto w_T$.
Nonetheless:

\begin{prop}
For all $\beta \in \Br_W^+$, there is a commutative diagram
\begin{equation}
\begin{tikzpicture}[baseline=(current bounding box.center), >=stealth]
\matrix(m)[matrix of math nodes, row sep=2.5em, column sep=5.5em, text height=2ex, text depth=0.5ex]
{	G(\beta)
		&B \times \bb{A}^\ell
		&X_0(\beta, B)\\
	G \times \cal{B}
		&G
		&\{1\}\\
		};
\path[->,font=\scriptsize, auto]
(m-1-2)		edge (m-1-1)
(m-1-3) 	edge node[above]{$(f_{\beta, B}(\vec{z}), \vec{z})$} (m-1-2)
(m-1-1)		edge (m-2-1)
(m-1-2)		edge (m-2-2)
(m-1-3) 	edge node{$f_{\beta, B}$} (m-2-3)
(m-2-2)		edge node[above]{$\id \times \{B\}$} (m-2-1)
(m-2-3)		edge (m-2-2);
\end{tikzpicture}
\end{equation}
where the middle vertical map sends $(b, \vec{z}) \mapsto b^{-1} f_{\beta, B}(\vec{z})$ and both squares are cartesian.
In particular, $X_0(\beta, B)$ is isomorphic to the fiber $G(\beta)_{1, B} = \cal{X}(\beta)_B$.
\end{prop}

\begin{proof}
It is enough to construct the map $B \times \bb{A}^\ell \to G(\beta)$.
To this end, let $B^{\ell - 1}$ act on $Bs_{1, T}B \times \cdots \times Bs_{\ell, T}B$ on the right by:
\begin{align}\begin{split}
&(x_1, x_2, \ldots, x_i, \ldots, x_\ell) \cdot (b_1, \ldots, b_{\ell - 1})\\
&=	(x_1b_1, b_1^{-1}x_2b_2, \ldots, b_{i - 1}^{-1}x_ib_i, \ldots, b_{\ell - 1}^{-1}x_\ell).
\end{split}\end{align}
We denote the quotient by
\begin{align}
M(\beta, B)
=	(Bs_{1, T}B \times \cdots \times Bs_{\ell, T}B)/B^{\ell - 1}.
\end{align}
(When $\beta = \sigma_w$, this variety is a dense open stratum in the variety denoted $M_w$ in \cite{hansen}.)
We have an isomorphism:
\begin{align}
\begin{array}{rcl}
B \times \bb{A}^\ell
	&\xrightarrow{\sim}
	&M(\beta, B)\\
(b, \vec{z})
	&\mapsto
	&[bf_{\alpha_{s_1}}(z_1), f_{\alpha_{s_2}}(z_2), \ldots, f_{\alpha_{s_\ell}}(z_\ell)]
	\end{array}
\end{align}
Let $B$ act on $M(\beta, B)$ by outer right conjugation and on $G$ by right multiplication.
Writing $x^g = g^{-1}xg$ and $B^g = g^{-1}Bg$ , we have a further isomorphism:
\begin{align}\label{eq:m}
\begin{array}{rcl}
(M(\beta, B) \times G)/B
	&\xrightarrow{\sim}
	&G(\beta)\\
{[x_1, \ldots, x_\ell, g]}
	&\mapsto
	&((x_1 \cdots x_\ell)^g, B^{x_2 \cdots x_\ell g}, \ldots, B^g)
	\end{array}
\end{align}
We now have a composition
\begin{align}\begin{split}
M(\beta, B)
&\xrightarrow{\sim}
	(M(\beta, B) \times B)/B
\to 
	(M(\beta, B) \times G)/B
\xrightarrow{\sim}
	G(\beta).
\end{split}\end{align}
It forms the top arrow of a cartesian diagram
\begin{equation}
\begin{tikzpicture}[baseline=(current bounding box.center), >=stealth]
\matrix(m)[matrix of math nodes, row sep=2.5em, column sep=3.5em, text height=2ex, text depth=0.5ex]
{	G(\beta)
		&M(\beta, B)\\
	G \times \cal{B}
		&G\\
		};
\path[->,font=\scriptsize, auto]
(m-1-2)		edge (m-1-1)
(m-1-1)		edge (m-2-1)
(m-1-2)		edge node{$m_{\beta, B}$} (m-2-2)
(m-2-2)		edge node[above]{$\id \times \{B\}$} (m-2-1);
\end{tikzpicture}
\end{equation}
in which $m_{\beta, B}([x_1, \ldots, x_\ell]) = x_1\cdots x_\ell$.
This is the left-hand square of the diagram in the proposition, once we identify $M(\beta, B)$ with $B \times \bb{A}^\ell$.
\end{proof}

\begin{rem}
The isomorphism \eqref{eq:m} can be pulled back to an isomorphism
\begin{align}
(X_0(\beta, B) \times G)/B
\xrightarrow{\sim}
\cal{X}(\beta).
\end{align}
Thus, we have an isomorphism of stacks:
\begin{align}
[X_0(\beta, B)/B] \simeq [\cal{X}(\beta)/G].
\end{align}
Note that $[X_0(\beta, B)/B]$ is homotopy equivalent to $[X_0(\beta, B)/T]$.
The $T$-action on $X_0(\beta, B)$ is important in \cite{mellit} and \cite{cggs}.
\end{rem}

\newpage
\section{Examples of the Decategorified Trace}\label{sec:examples}

In this appendix, we list some values of $\Tr{-}$ and $\Tr{-}^0$.
Throughout, we rely on Table 8.1 in \cite{gp} and the descriptions of $\{-, -\}$ in \cite[Ch.\ 4]{lusztig_1984}.

\subsection*{Type $A_1$}

Here, $W = S_2$ and $\hat{W} = \{1, \varepsilon\}$.
Writing $\sigma$ for the positive generator of $\Br_W = \Br_2$, we have:
\begin{align}
\Tr{\sigma^m} &= (\Q^{\frac{1}{2}})^m + (-\Q^{-\frac{1}{2}})^m \varepsilon,\\
\Tr{\sigma^m}^0 &= \frac{1}{1 - \Q^2}
((1 - (-\Q)^{m + 1}) + (\Q + (-\Q)^m)\varepsilon).
\end{align}
For instance, we have $\Tr{\sigma^{2k + 1}}^0 = (1 + \Q^2 + \cdots + \Q^{2k}) + (\Q + \Q^3 + \cdots + \Q^{2k - 1})\varepsilon$ for all $k \geq 1$.

\subsection*{Type $A_2$}

Here, $W = S_3$.
We write $S = \{s, t\}$ and $\hat{W} = \{1, \phi, \varepsilon\}$, where $\phi(1) = 2$.
The following table lists $(\psi, \Tr{\sigma_w})_W$ for all $w \in W$ and $\psi \in \hat{W}$:
\begin{align}
\begin{array}{r@{\quad}|@{\quad}l@{}ll@{}ll@{}ll@{}l}
	&&\bb{1}
	&&\sigma_s, \sigma_t
	&&\sigma_{st}, \sigma_{ts}
	&&\sigma_{sts}\\[1ex]
1
	&&1
	&&\Q^{\frac{1}{2}}
	&&\Q
	&&\Q^{\frac{3}{2}}\\
\phi
	&&2
	&&\Q^{\frac{1}{2}} - \Q^{-\frac{1}{2}}
	&-&1
	&&0\\
\varepsilon
	&&1
	&-&\Q^{-\frac{1}{2}}
	&&\Q^{-1}
	&-&\Q^{-\frac{3}{2}}
\end{array}
\end{align}
We deduce that:
\begin{align}
\Tr{\sigma_w}^0
=	\left\{\begin{array}{ll}
(1 - \Q)^{-2}(1 + 2\phi + \varepsilon)
	&w = 1\\
(1 - \Q)^{-1}(1 + \phi)
	&w \in \{s, t\}\\
1
	&w \in \{st, ts\}\\
(1 - \Q)^{-1}(1 - \Q + \Q^2 + \Q\phi)
	&w = w_0
\end{array}\right.
\end{align}

\subsection*{Type $BC_2$}

Here, $W$ is the dihedral group of the square.
We write $S = \{s, t\}$ and $\hat{W} = \{1, \delta, \phi, \varepsilon\delta, \epsilon\}$, where $\delta(1) = 2$ and $\phi(1) = 2$ and $\delta(s) = 1$.
The following table lists $(\psi, \Tr{\sigma_w})_W$ for all $w$ and $\psi$:
\begin{align}
\begin{array}{r@{\quad}|@{\quad}l@{}ll@{}ll@{}ll@{}ll@{}ll@{}ll@{}l}
	&&\bb{1}
	&&\sigma_s
	&&\sigma_t
	&&\sigma_{st}, \sigma_{ts}
	&&\sigma_{sts}
	&&\sigma_{tst}
	&&\sigma_{w_0}\\[1ex]
1
	&&1
	&&\Q^{\frac{1}{2}}
	&&\Q^{\frac{1}{2}}
	&&\Q
	&&\Q^{\frac{3}{2}}
	&&\Q^{\frac{3}{2}}
	&&\Q^2\\
\delta
	&&1
	&&\Q^{\frac{1}{2}}
	&-&\Q^{-\frac{1}{2}}
	&&0
	&-&\Q^{-\frac{1}{2}}
	&&\Q^{\frac{1}{2}}
	&-&1\\
\phi
	&&2
	&&\Q^{\frac{1}{2}} - \Q^{-\frac{1}{2}}
	&&\Q^{\frac{1}{2}} - \Q^{-\frac{1}{2}}
	&-&1
	&&0
	&&0
	&&0\\
\varepsilon\delta
	&&1
	&-&\Q^{-\frac{1}{2}}
	&&\Q^{\frac{1}{2}}
	&&0
	&&\Q^{\frac{1}{2}}
	&-&\Q^{-\frac{1}{2}}
	&-&1\\
\varepsilon
	&&1
	&-&\Q^{-\frac{1}{2}}
	&-&\Q^{-\frac{1}{2}}
	&&\Q^{-1}
	&-&\Q^{-\frac{3}{2}}
	&-&\Q^{-\frac{3}{2}}
	&&\Q^{-2}
\end{array}
\end{align}
We deduce that:
\begin{align}
\Tr{\sigma_w}^0
=	\left\{\begin{array}{ll}
(1 - \Q)^{-2}(1 + \delta + 2\phi + \varepsilon\delta + \varepsilon)
	&w = 1\\
(1 - \Q)^{-1}(1 + \delta + \phi)
	&w = s\\
(1 - \Q)^{-1}(1 + \varepsilon\delta + \phi)
	&w = t\\
1
	&w \in \{st, ts\}\\
(1 - \Q)^{-1}
	(1 - \Q + \Q^2 + \Q\varepsilon\delta + \Q\phi)
	&w = sts\\
(1 - \Q)^{-1}	
	(1 - \Q + \Q^2 + \Q\delta + \Q\phi)
	&w = tst\\
1 + \Q^2 + \Q\phi
	&w = w_0
\end{array}\right.
\end{align}
Note that this is the smallest $W$ for which $\{-, -\}$ is nontrivial.

\subsection*{An Iterated Torus Braid}

Let $W = S_4$.
We write $S = \{s, t, u\}$ and $\hat{W} = \{1, \phi, \psi, \varepsilon\phi, \varepsilon\}$, where $\phi(1) = 3$ and $\psi(1) = 2$ and $\phi(s) = 1$.
Let
\begin{align}
\beta = (\sigma_s\sigma_t\sigma_u)^6\sigma_s \in \Br_W = \Br_4.
\end{align}
Using SageMath, we calculated:
\begin{align}
\Tr{\beta} 
= \Q^{\frac{19}{2}} 
- \Q^{\frac{7}{2}} \phi 
+ (\Q^{\frac{1}{2}} 
- \Q^{-\frac{1}{2}})\psi
+ \Q^{-\frac{7}{2}} \varepsilon\phi
- \Q^{-\frac{19}{2}}\varepsilon.
\end{align}
The nonzero coefficients of $\Tr{\beta}^0$ are listed below:
\begin{align}\label{eq:zariski}
\small
\begin{array}{r@{\quad}|@{\quad}p{0.75em}p{0.75em}p{0.75em}p{0.75em}p{0.75em}p{0.75em}p{0.75em}p{0.75em}p{0.75em}p{0.75em}p{0.75em}p{0.75em}p{0.75em}p{0.75em}p{0.75em}p{0.75em}p{0.75em}}	&$\Q^0$
&$\Q^1$
&$\Q^2$
&$\Q^3$
&$\Q^4$
&$\Q^5$
&$\Q^6$
&$\Q^7$
&$\Q^8$
&$\Q^9$
&$\Q^{10}$
&$\Q^{11}$
&$\Q^{12}$
&$\Q^{13}$
&$\Q^{14}$
&$\Q^{15}$
&$\Q^{16}$\\[1ex]
1					&1& &1&1&2&1&3&1&3&1&3&1&2&1&1& &1\\
\phi				& &1&1&2&2&4&3&5&3&5&3&4&2&2&1&1&\\
\psi				& & &1& &2&1&3&1&4&1&3&1&2& &1& &\\
\varepsilon\phi	& & & &1&1&2&2&3&2&3&2&2&1&1& & &\\
\varepsilon		& & & & & & &1& &1& &1& & & & & &
\end{array}
\end{align}
This table offers a vivid illustration of Theorem \ref{thm:symmetry}.
By Proposition \ref{prop:markov}, the rows of the table for $1, \phi, \varepsilon\phi$, and $\varepsilon$ list the coefficients of the HOMFLY polynomial of the link closure $\hat{\beta}$, up to shifting by $a^{2i}\Q^{-8}$ for various $i$.

\begin{rem}
As it happens, $\hat{\beta}$ is the link of the complex plane curve singularity given by $y = x^{\frac{3}{2}} + x^{\frac{7}{4}}$.
In particular, it is an iterated torus knot.
\end{rem}

\begin{rem}
Let $G = \SL_4$.
By Theorem \ref{thm:virtual}, it should be possible to decompose $\Tr{\beta}^0$ into a $\bb{Z}[\Q]$-linear combination of total Springer representations $Q_u$, as we run over representatives $u$ of $G^F$-orbits in $\cal{U}^F$.

The $G^F$-orbits are indexed by the partitions of the integer $4$.
Let the orbit of the identity correspond to the partition $(1, 1, 1, 1) \vdash 4$.
Then the $Q_u$ are:
\begin{align}\begin{split}
Q_4
	&= 1,\\
Q_{3, 1}
	&= 1 + \Q\phi,\\
Q_{2, 2}
	&= 1 + \Q\phi + \Q^2\psi,\\
Q_{2, 1, 1}
	&= 1 + (\Q + \Q^2)\phi + \Q^2\psi + \Q^3 \varepsilon\phi,\\
Q_{1, 1, 1, 1}
	&= 1 + (\Q + \Q^2 + \Q^3)\phi + (\Q^2 + \Q^4)\psi + (\Q^3 + \Q^4 + \Q^5)\varepsilon\phi + \Q^6\varepsilon.
\end{split}\end{align}
Using \eqref{eq:zariski}, we find that:
\begin{align}\begin{split}
\Tr{\beta}^0
&= \Q^{16}Q_1 + (\Q^{10} + \Q^{11} + \Q^{12} + \Q^{13} + \Q^{14})Q_{3, 1}\\
&\qquad + (\Q^6 - \Q^7 + \Q^8 + \Q^{10} + \Q^{12})Q_{2, 2}\\
&\qquad + (\Q^3 + \Q^4 + \Q^5 + 2\Q^6 + 2\Q^7 + 2\Q^8 + \Q^9 + \Q^{10})Q_{2, 1, 1}\\
&\qquad + (1 + \Q^2 + \Q^4)Q_{1, 1, 1, 1}.
\end{split}\end{align}
It is curious that the entries in \eqref{eq:zariski} are all positive, whereas the decomposition into $Q_u$'s contains a (single) negative term.
\end{rem}


\newpage

\end{document}